\documentclass[final,reqno,dvipsnames]{elsarticle}

\usepackage[colorlinks=true,breaklinks=true,linkcolor=lightblue,citecolor=lightblue]{hyperref}

\usepackage{amssymb}
\usepackage{amsmath}
\usepackage{amsfonts}
\usepackage{amsthm}

\usepackage{bbm}

\usepackage{latexsym}

\usepackage{subfigure}
\usepackage{stmaryrd}

\usepackage{algorithmic}

\usepackage[utf8]{inputenc}
\usepackage{booktabs}

\usepackage{mathrsfs}
\usepackage{rotating} 
\usepackage{graphicx,float,epsfig,color,fancyhdr}
\usepackage{amsopn}
\usepackage{multirow}

\usepackage{soul}

\usepackage{tikz}
\usetikzlibrary{patterns,decorations.pathmorphing,calc,arrows.meta,positioning}

\def\E{{K}}

\def\vb{{\bf v}}

\renewcommand{\O}{\Omega}

\newtheorem{lemma}{Lemma}[section]
\newtheorem{remark}{Remark}[section] 
\newtheorem{proposition}{Proposition}[section]
\newtheorem{assumption}{Assumption}
\newtheorem{theorem}{Theorem}[section]

\newtheorem{problem}{Problem}

\def\bu{\mathbf{u}}

\newcommand\bn{\mathbf{n}}
\newcommand\bt{\boldsymbol{t}}

\def\CTh{\mathfrak{T}_h}
\def\HdoK{{H^{2}(\E)}}

\def\EE{\mathcal{E}}
\def\EEh{\EE_h}

\def\HdoK{{H^{2}(\E)}}

\def\inte{{\rm int}}
\def\bdry{{\rm bdry}}
\def\EE{\mathcal{E}}
\def\EEh{\EE_h}
\def\EEi{\EE_h^{\inte}}
\def\EEbdry{\EE_h^{\bdry}}
\def\VV{\mathfrak{V}}
\def\VVi{\VV_h^{\inte}}
\def\VVbdry{\VV_h^{\bdry}}
\def\HdoK{{H^{2}(\E)}}

\def\CV{\mathcal{V}}
\def\VtK{\widetilde{\CV}_{h}(\E)}
\newcommand\Vh{{\mathcal V}_h}

\def\vb{{\bf v}}

\def\CD{\mathcal{D}}

\def\CS{\mathcal{S}}

\def\CT{\mathcal{T}}

\def\E{\mathrm{K}}

\def\half{\mbox{$\frac12$}}

\def\N{{\mathbb{N}}}

\def\O{\Omega}

\def\Vh{V_h}

\def\nc{{\rm NC}}
\def\HncT{H^{2,\nc}(\CTh)}
\newcommand{\vertiii}[1]{{\left\vert\kern-0.25ex\left\vert\kern-0.25ex\left\vert #1 
    \right\vert\kern-0.25ex\right\vert\kern-0.25ex\right\vert}}

\renewcommand\sp{\mathop{\mathrm{sp}}\nolimits}

\newcommand*{\jump}[1]{\lbrack\hspace{-1.5pt}\lbrack #1\rbrack\hspace{-1.5pt}\rbrack}

\newcommand{\DOFS}[1]{\mbox{$\mathbf{#1}$}}
\newcommand{\Vt}{\widetilde{V}}

\allowdisplaybreaks
\journal{}
\date{\today}

\begin{document}

\begin{frontmatter}

  \title{Conforming and non-conforming virtual element methods for
    the biharmonic Steklov eigenvalue problem with minimum regularity}
  
  \author[1]{Dibyendu Adak}
  \ead{dibyendu.jumath@gmail.com}
  
  \author[2,3,4]{Daniele Boffi}
  \ead{daniele.boffi@kaust.edu.sa}
  
  \author[3]{Francesca Gardini}
  \ead{francesca.gardini@unipv.it}
  
  \author[4]{Gianmarco Manzini}
  \ead{gm.manzini@gmail.com}
  
  \author[5]{Jesus Vellojin}
  \ead{jesus.vellojinm@usm.cl}

  \address[1]{Department of Mathematics, Indian Institute of Technology-Kharagpur, WB, India.}
  
  \address[5]{Departamento de Ciencias, Universidad T\'ecnica Federico Santa Mar\'ia, Valpara\'iso, Chile}

  \address[2]{King Abdullah University of Science and Technology (KAUST),
    Computer, Electrical and Mathematical Science and Engineering
    Division, Thuwal, Saudi Arabia}

  \address[3]{Universit\`a di Pavia, Dipartimento di Matematica
    ``F. Casorati'', Pavia, Italy}

  \address[4]{IMATI, Consiglio Nazionale delle Ricerche, via Ferrata
    1, 27100 Pavia, Italy.}  
  
  \begin{abstract}
    In this work, we analyze the conforming and $C^0$-non-conforming
    Virtual Element Method for a fourth-order Steklov eigenvalue
    problem on a generally shaped, possibly nonconvex, polygonal
    domain. By employing an {\it enriching } operator, we derive the
    convergence analysis using the discrete $H^2$ seminorm, and the
    $H^1$ and $L^2$ norms.
    We use the Babu\v{s}ka--Osborn spectral theory \cite{BO} to prove
    that the numerical scheme approximates the spectrum without
    introducing any spurious eigenvalue.
    Moreover, we derive the optimal order of convergence for
    eigenfunctions and double order for eigenvalues.
    We assess the performance of the method on several numerical tests
    using different families of polygonal meshes.
  \end{abstract}


\date{\today}
\end{frontmatter}


\setcounter{equation}{0}

\section{Introduction}
 
Elliptic problems featuring eigenvalues in their boundary conditions
are typically called Steklov problems.  There are various instances of
such problems, including the self-adjoint Steklov eigenvalue problem
modeling vibration in incompressible fluid-structure interactions
\cite{bermudez2000finite}, the sloshing problem
\cite{bermudez2003finite}, and the longitudinal vibration of an
elastic bar or the vibration of a thin membrane stretched over a
bounded region \cite{BO}. In \cite{kuttler1972remarks}, the author
proves a bound on the first eigenvalue on a square domain and that the
associated eigenfunction does not change sign. In
\cite{bucur2009first,ferrero2005fourth}, the authors study the
spectrum of a fourth order Steklov eigenvalue problem on a bounded
domain in $\mathbb{R}^n$ and give the explicit form of the spectrum
in the case where the domain is a ball. Reference~\cite{wang2009sharp}
provides bounds on the first non-zero Steklov eigenvalues when
$\Omega$ is isometric to an $n$-dimensional Euclidean ball.


\subsection{Physical Motivation}

The model problem, which will be formally introduced by the Steklov's eigenvalue problem of 
equations~\eqref{eq:model1:A}--\eqref{eq:model1:C} in the next section, admits a natural
mechanical interpretation in the framework of thin plate vibrations
with constrained boundary kinematics.
Consider a thin elastic plate whose edges are supported by
frictionless roller guides that enforce the constraint
$\partial_{\bn} u = 0$ on $\partial\Omega$, meaning that the plate
boundary must remain horizontal (zero slope in the normal direction)
while being free to displace vertically.
Additionally, the plate edges are connected to a distributed elastic
foundation (such as linear springs or elastomeric bearing pads) that
exerts a restoring force proportional to the vertical deflection.
The equilibrium between the internal shear force in the plate and the
elastic reaction from the boundary supports is expressed by the
condition $\partial_{\bn}(\Delta u) = -\lambda u$ on $\partial\Omega$,
where $\lambda$ is the Steklov's eigenvalue.
Figure~\ref{fig:physical:model} illustrates this mechanical
configuration.

This type of boundary configuration arises in several engineering
applications.
Expansion joints in bridge decks commonly employ roller or sliding
bearings combined with elastomeric pads to accommodate thermal and
seismic deformations while maintaining structural
continuity~\cite{chen2014bridge}; the rollers constrain rotation at
the joint while the elastomeric elements provide restoring stiffness.
Seismic base isolation systems for buildings and precision equipment
similarly combine sliding or rolling supports with elastic centering
devices~\cite{naeim1999design}.
In the domain of microelectromechanical systems (MEMS), guided mirror
platforms with electrostatic or magnetic restoring forces exhibit
analogous boundary behavior, where tilt is mechanically constrained
while vertical displacement is elastically
resisted~\cite{senturia2001microsystem}.
Additionally, vibration isolation tables used in optical laboratories
and semiconductor manufacturing facilities employ air bearings or
roller guides together with pneumatic springs, maintaining platform
levelness while providing controlled isolation from ground vibrations.
The eigenvalue $\lambda$ 
characterizes the
natural vibration frequencies of the plate under these guided-elastic
boundary conditions, making this model relevant for the dynamic
analysis of such systems.
The extension of this model to non-convex domains is particularly
relevant for plates with cutouts (e.g., access panels, ventilation
openings) or L-shaped configurations commonly encountered in
structural floor plates and bridge deck segments.


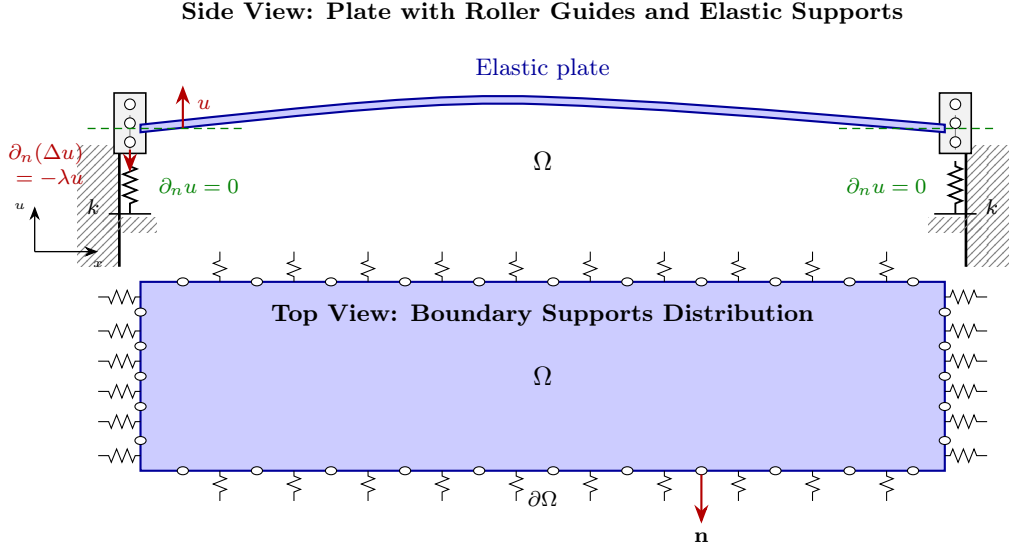
\begin{figure}[t!]
  \begin{center}
       \begin{tikzpicture}[xscale=1.4,
    spring/.style={decorate,decoration={zigzag,amplitude=2.5pt,segment length=4pt,pre length=2pt,post length=2pt}},
    ground/.style={pattern=north east lines,pattern color=gray},
    roller/.style={circle,draw,fill=white,minimum size=4pt,inner sep=0pt},
    thick plate/.style={fill=blue!20,draw=blue!60!black,line width=0.8pt},
    force arrow/.style={-{Stealth[length=2.5mm]},red!70!black,line width=0.8pt},
    dim arrow/.style={{Stealth[length=2mm]}-{Stealth[length=2mm]},black,line width=0.5pt}
  ]

  \begin{scope}[local bounding box=side view]
    
    \fill[ground] (-0.4,-1.8) rectangle (0,-0.2);
    \draw[line width=1pt] (0,-1.8) -- (0,-0.2);
    
    \fill[ground] (8,-1.8) rectangle (8.4,-0.2);
    \draw[line width=1pt] (8,-1.8) -- (8,-0.2);
    
    \draw[line width=0.6pt,fill=gray!10] (-0.05,-0.3) rectangle (0.25,0.5);
    \draw[densely dashed,gray] (0.1,-0.25) -- (0.1,0.45);
    
    \draw[line width=0.6pt,fill=gray!10] (7.75,-0.3) rectangle (8.05,0.5);
    \draw[densely dashed,gray] (7.9,-0.25) -- (7.9,0.45);
    
    \node[roller] (r1l) at (0.1,0.35) {};
    \node[roller] (r2l) at (0.1,0.1) {};
    \node[roller] (r3l) at (0.1,-0.15) {};
    
    \node[roller] (r1r) at (7.9,0.35) {};
    \node[roller] (r2r) at (7.9,0.1) {};
    \node[roller] (r3r) at (7.9,-0.15) {};
    
    \draw[thick plate] 
      (0.2,0.08) 
      .. controls (2,0.35) and (3,0.5) .. (4,0.45)
      .. controls (5,0.4) and (6,0.3) .. (7.8,0.08)
      -- (7.8,-0.02)
      .. controls (6,0.2) and (5,0.3) .. (4,0.35)
      .. controls (3,0.4) and (2,0.25) .. (0.2,-0.02)
      -- cycle;
    
    \node[above,blue!60!black] at (4,0.55) {\small Elastic plate};
    
    \draw[spring,line width=0.7pt] (0.1,-0.4) -- (0.1,-1.1);
    \draw[line width=0.6pt] (-0.1,-1.1) -- (0.3,-1.1);
    \fill[ground] (-0.15,-1.15) rectangle (0.35,-1.35);
    
    \draw[spring,line width=0.7pt] (7.9,-0.4) -- (7.9,-1.1);
    \draw[line width=0.6pt] (7.7,-1.1) -- (8.1,-1.1);
    \fill[ground] (7.65,-1.15) rectangle (8.15,-1.35);
    
    \node[left,font=\footnotesize] at (-0.1,-1) {$k$};
    \node[right,font=\footnotesize] at (8.1,-1) {$k$};
    
    \draw[force arrow] (0.6,0.03) -- (0.6,0.6);
    \node[right,font=\footnotesize,red!70!black] at (0.65,0.35) {$u$};
    
    \draw[line width=0.5pt,densely dashed,green!50!black] (-0.3,0.03) -- (1.2,0.03);
    \node[below,font=\footnotesize,green!50!black] at (0.75,-0.50) {$\partial_n u = 0$};
    
    \draw[line width=0.5pt,densely dashed,green!50!black] (6.8,0.03) -- (8.3,0.03);
    \node[below,font=\footnotesize,green!50!black] at (7.25,-0.50) {$\partial_n u = 0$};
    
    \draw[force arrow] (0.1,-0.25) -- (0.1,-0.55);

    \node[left,font=\footnotesize,red!70!black,align=right] at (-0.25,-0.45) 
    {$\partial_n(\Delta u)$\\$= -\lambda u$};
    
    \node at (4,-0.4) {$\Omega$};
    
    \draw[-{Stealth},line width=0.5pt] (-0.8,-1.6) -- (-0.8,-1.0);
    \draw[-{Stealth},line width=0.5pt] (-0.8,-1.6) -- (-0.2,-1.6);
    \node[below,font=\tiny] at (-0.2,-1.6) {$x$};
    \node[left,font=\tiny] at (-0.8,-1.0) {$u$};
    
  \end{scope}
  
  \node[above,font=\small\bfseries] at (4,1.3) {Side View: Plate with Roller Guides and Elastic Supports};

  \begin{scope}[shift={(0,-4.5)},local bounding box=top view]
    
    \draw[thick plate] (0.2,0) rectangle (7.8,2.5);
    
    \node at (4,1.25) {$\Omega$};
    
    \node[below,font=\footnotesize] at (4,-0.15) {$\partial\Omega$};
    
    \foreach \x in {0.6,1.3,2.0,2.7,3.4,4.1,4.8,5.5,6.2,6.9,7.5} {
      \draw[fill=white,line width=0.4pt] (\x,0) circle (1.5pt);
      \draw[fill=white,line width=0.4pt] (\x,2.5) circle (1.5pt);
    }
    \foreach \y in {0.4,0.85,1.25,1.65,2.1} {
      \draw[fill=white,line width=0.4pt] (0.2,\y) circle (1.5pt);
      \draw[fill=white,line width=0.4pt] (7.8,\y) circle (1.5pt);
    }
    
    \foreach \x in {0.95,1.65,2.35,3.05,3.75,4.45,5.15,5.85,6.55,7.25} {
      \draw[spring,line width=0.4pt] (\x,0) -- (\x,-0.4);
      \draw[spring,line width=0.4pt] (\x,2.5) -- (\x,2.9);
    }
    \foreach \y in {0.2,0.65,1.05,1.45,1.85,2.3} {
      \draw[spring,line width=0.4pt] (0.2,\y) -- (-0.2,\y);
      \draw[spring,line width=0.4pt] (7.8,\y) -- (8.2,\y);
    }
    
    \draw[force arrow] (5.5,-0.05) -- (5.5,-0.7);
    \node[below,font=\footnotesize] at (5.5,-0.7) {$\mathbf{n}$};
    
  \end{scope}
  
  \node[above,font=\small\bfseries] at (4,-2.7) {Top View: Boundary Supports Distribution};

\end{tikzpicture}
  \end{center}

  \caption{Mechanical interpretation of the model
    problem~\eqref{eq:model1:A}--\eqref{eq:model1:C}.
    \textit{Top:} Side view of a thin elastic plate supported by
    roller guides (constraining the boundary slope to zero,
    $\partial_{\bn} u = 0$) and connected to elastic springs that
    provide a restoring force proportional to the vertical
    deflection.
    The balance between the internal shear force and the spring
    reaction yields the boundary condition
    $\partial_{\bn}(\Delta u) = -\lambda u$.
    \textit{Bottom:} Top view showing the distribution of roller
    supports and elastic springs along the entire boundary
    $\partial\Omega$.}
  \label{fig:physical:model}
\end{figure}


\subsection{Background}
Several numerical techniques apply to the approximation of the Steklov
eigenvalue problem. Among them, we emphasize the finite element
approximation for the fourth-order Steklov eigenvalue problem
\cite{bi2011conforming} and its virtual element approximation
\cite{monzon2019virtual}; the a priori, and a posteriori error
analysis for the non-self-adjoint Steklov eigenvalue problem
\cite{wang2022priori}; the non-conforming Crouzeix Raviart type
approximation~\cite{yang2020non}. Two grid discretization and
multigrid correction schemes are discussed in~\cite{zhang2020multigrid}.

The Virtual Element Method (VEM) extends the finite element method to
polygonal meshes inheriting the Galerkin orthogonality property.
Also, the VEM is connected to the mimetic finite difference
method~\cite{BasicVEM2013}.
In the VEM, we do not know the basis functions explicitly.
Instead, we use their projection onto a suitable polynomial space in
the discrete bilinear form of the variational formulation.
Extending the discrete scheme from $2D$ to $3D$ is more manageable
than other existing techniques, such as in the non-conforming finite
element methods~\cite{MR3460621}.

Conforming and non-conforming VEM have extensively been studied for
fourth-order problems such as plate bending problems
\cite{MR3002804,adak2022convergence,zhao2018morley,MR3529253,antonietti2018fully},
Cahn--Hilliard equation \cite{dedner2021higher,antonietti2016c},
vibration and buckling eigenvalue problems
\cite{adak2023c0,MR3875292,mora2020virtual}, Stommel--Munk models
\cite{adak2023morley}, Navier--Stokes, and Oseen equations in stream
function formulation \cite{adak2021virtual,mora2022virtual},
Singularly perturbed equation \cite{zhang2020nonconforming},
Boussinesq equation \cite{da2023fully}. In particular, non-conforming
VEM was developed for elliptic problems in \cite{de2016nonconforming}
and later studied for different model problems
\cite{cangiani2016nonconforming,liu2019nonconforming,zhang2021divergence}.

In Ref.~\cite{antonietti2018fully}, the fully
non-conforming VEM was developed  to approximate plate problems. Morley-type
non-conforming VEM was developed in \cite{zhao2018morley}, and
$C^0$-non-conforming VEM in \cite{zhao2016nonconforming}. Furthermore,
the works in Reference~\cite{zhao2019divergence,MR4019985}, studied the Stokes
complex relationship between the Morley and divergence-free
Crouzeix-Raviart space.

\subsection{Novelty}

Our study introduces a novel fourth-order Steklov eigenvalue problem
given by a biharmonic Steklov operator with boundary conditions
analogous to those considered in \cite{adak2023conforming} and aims to
devise $H^2$-conforming and $C^0$-non-conforming VEMs for its
approximation.
For a long time, the finite element method (FEM) for plate bending
problems has been an important area of research because of the
difficulty in constructing $H^2$-conforming elements. We refer the
readers to \cite{MR3002804,chinosi2016} and the references therein.
The primary motivation for using the $C^0$ non-conforming scheme is
that this approach relaxes the global continuity conditions of the
approximating functional space. Further, the boundary conditions
imposed to define local virtual element spaces significantly influence
the global regularity of the discrete functions. Unlike the
$H^2$-conforming case, we have considered different boundary
conditions to impose the $C^0$ continuity of the discrete basis
functions. Moreover, we are interested in developing conforming and
non-conforming spaces and carrying out a unified analysis, which
requires weaker regularity of the functions. Additionally, we have
highlighted the difficulty of extending our work to the fully
non-conforming space \cite{zhao2018morley},
in Conclusion. The proposed analysis can be extended
to any order of accuracy, cf. Remark~\ref{General_order}, where the
analytical solution has higher regularity. To the best of our
knowledge, the formulation is new, and the convergence analysis covers
both the conforming and non-conforming approximations with minimum
regularity of the exact solution.

The model problem deals with boundary integration on the right-hand
side. Consequently, we must have a well-defined solution operator in
$L^2(\partial\Omega)$ for both conforming and non-conforming analysis.
We emphasize that such a definition is not straightforward for the
non-conforming VEM since its space has less regularity. The fully
non-conforming virtual element space developed in \cite{zhao2018morley}
can not be used as approximation space since the virtual element
functions has only a global $L^2(\Omega)$ regularity. The loading term
involves boundary integration for which we need at least $H^1(\Omega)$
regularity. This issue could be avoided by properly choosing discrete
space such as $C^0$-non-conforming virtual element space.
Additionally, using the spectral theory for compact operators, we
prove optimal convergence order for the approximation of eigenvalues
and eigenfunctions on very general types of domains, such as
non-convex domains.
However, in the non-conforming VEM, extending the analysis to
non-convex domains is not immediate as in the conforming settings
because the discrete space is not a subspace of $H^2(\Omega)$.
The non-conforming nature of the method is reflected by an additional
error term due to the {\it variational crime}. By introducing an
enriching operator that maps non-conforming space to conforming space,
we derive the convergence analysis with optimal order in different
norms such as discrete $H^2$ seminorm and $H^1$, $L^2$ norms. We can
also derive the errors in different norms in the conforming
method.

\subsection{Outline}

The article is organized as follows.
In Section~\ref{SEC:STAT}, we present the fourth-order Steklov
eigenvalue problem and the associated variational formulations. We
define the source problems and solution operators relevant to the weak
formulations. Additionally, we study the regularity of the solution
related to the weak formulation.
In Section~\ref{Discrete:sp}, we derive the discrete formulation of
the problem and we introduce the lowest order $H^2$-conforming and
$C^0$ non-conforming virtual element spaces, the \emph{enriching} and
the discrete solution operators.
Then, in Section~\ref{Analysis:Source}, we carry out the convergence
analysis for the source problems with conforming method in different
norms such as $H^2$, $H^1$, and $L^2$ norms, and we extend the study
to the non-conforming method through the enriching technique.
In Section~\ref{SPC:CONVERGENCE}, we carry out the convergence analysis
for the eigenvalue problem.
In Section~\ref{Numer:Exp}, we investigate the behavior of the
proposed schemes through a set of numerical experiments that confirm
the theoretical expectations.
In Section~\ref{conclusion}, we offer our final remarks and discuss
possible future developments of our work.

\setcounter{equation}{0}
\section{The model eigenvalue problem}
\label{SEC:STAT}
Let $\O\subset\mathbb{R}^2$ be an open, bounded, polygonal domain with
boundary $\partial\Omega$.
We are interested in the following problem: find $\lambda\in\mathbb{R}$ such that there exists $u\neq 0$ such that the
following system holds
\begin{align}
  \Delta^2 u                &= 0 \phantom{-\lambda u}\quad\text{in~}\Omega,        \label{eq:model1:A}\\
  \bn\cdot\nabla{u}         &= 0 \phantom{-\lambda u}\quad\text{on~}\partial\Omega,\label{eq:model1:B}\\
  \bn\cdot\nabla(\Delta u)  &= -\lambda u\phantom{0} \quad\text{on~}\partial\Omega, \label{eq:model1:C}
\end{align}
where $\bn$ is the outward unit vector orthogonal
to $\partial\Omega$.

\subsection{Preliminaries and weak formulation of the problem }
\subsection{Notation}
Throughout this paper, we follow the convention of Sobolev spaces of
Ref.~\cite{adams2003sobolev}.
Accordingly, we denote the space of square-integrable functions
defined on any open, bounded, connected domain $\CD \subset
\mathbb{R}^2$ with boundary $\partial\CD$ by $L^2(\CD)$, and the
Hilbert space of functions in $L^2(\CD)$ with all partial derivatives
up to a positive integer $m$ also in $L^2(\CD)$ by $H^m(\CD)$.
We endow $H^m(\CD)$ with a norm and a seminorm that we denote as $ \|
\cdot \|_{m,\CD}$ and $\vert \,\cdot\,\vert_{m,\CD}$, respectively.
We denote the space of polynomials of degree up to a given integer
$l\geq0$ and defined on $\CD$ by $\mathbb{P}_{l}(\CD)$, and, for
$l=-1$, we conventionally assume that
$\mathbb{P}_{-1}(\CD)=\big\{0\big\}$.
We denote the unit vector that is orthogonal to $\partial\CD$ and
pointing out of $\CD$ by $\mathbf{n}_{\CD}=(n_1,n_2)^T$, and the unit
vector that is tangent to $\partial\CD$ by
$\mathbf{t}_{\CD}=(t_1,t_2)^T$.
The mutual orientation of vectors $\mathbf{n}_{\CD}$ and
$\mathbf{t}_{\CD}$ is such that $t_1=-n_2$ and $t_2=n_1$ when we
evaluate these vectors on the same edge of $\partial\CD$.
To avoid ambiguity, we express $\partial_{\bn}\phi=\bn\cdot\nabla\phi$
and $\partial_{\bt}\phi=\bt\cdot\nabla\phi$ to denote the normal and
tangential derivatives along an edge with unit normal and tangential
vectors $\bn$ and $\bt$, respectively.
The \emph{Hessian matrix of $\phi$} is given by
$D^2 \phi:=(\partial_{ij}\phi)_{1\leq i,j \leq2}$; the
\emph{gradient of $u$} is defined as the vector $\nabla u=(\partial_j
u)_{j,=1,2}$.
We denote the inner product of any pair of tensors
$\boldsymbol{\tau}=(\tau_{ij})_{i,j=1,2}$ and
$\boldsymbol{\sigma}=(\sigma_{ij})_{i,j=1,2}$ by
$\boldsymbol{\tau}:\boldsymbol{\sigma}=\sum_{i,j=1}^{2}\tau_{ij}\sigma_{ij}$.
Finally, we bring forth that the letter $C$, possibly with a subindex,
a superindex, or a modifier on top such as ``$\widetilde{C}$'' or
``$\widehat{C}$'', or ``$C_{\Omega}$'' denotes a positive constant
whose value can be different at any instance and that is independent
of $h$ but may depend on the other parameters of the problem and the
discretization that we will introduce in the next sections.
We also use the abbreviation "Ref." to indicate the word "Reference".

\subsection{The continuous spectral variational formulation} 

We present the variational formulation for
\eqref{eq:model1:A}-\eqref{eq:model1:C}.
We multiply \eqref{eq:model1:A} by suitable test functions, we
integrate by parts and use the boundary conditions \eqref{eq:model1:B}
and \eqref{eq:model1:C}.
The variational formulation reads:
\begin{problem}
  \label{pr:eigen1}
  Find $(u, \lambda) \in V \times \mathbb{R}$ with $u\neq 0$, such that
  \begin{align*}
    a(u,v)=\lambda b(u,v) \quad \forall v \in V,
  \end{align*}
  where
  \begin{align*}
    V = \Big\{
    v\in H^2(\O)~\textrm{such~that}~
    \partial_{\bn}v=0~\text{on}~\partial\Omega
    \Big\},
  \end{align*}
  and the bilinear forms
  $a:V\times V\rightarrow \mathbb{R}$ and $b:V\times
  V\rightarrow\mathbb{R}$ are defined by
  \begin{align*}
    a(u,v) := \int_{\Omega} D^2 u : D^2 v
    \quad\textrm{and}\quad
    b(u,v):= \int_{\partial\Omega} u v
  \end{align*}
  for all $u$ and $v\in V$.
\end{problem}

We reformulate Problem~\ref{pr:eigen1} as follows by introducing the coercive bilinear 
form $\widehat{a}(\cdot,\cdot)$.
\begin{problem}
  \label{pr:eigen1:mod}
  Find $(u, \lambda) \in V \times \mathbb{R}$ with $u\neq 0$, such that
  \begin{align*}
    \widehat{a}(u,v)=(\lambda+1) b(u,v) \quad \forall v \in V,
  \end{align*}
  where the bilinear form $\widehat{a}:V\times
  V\rightarrow \mathbb{R}$ reads
  \begin{align*}
    \widehat{a}(u,v):=\int_{\Omega} D^2 u : D^2 v + \int_{\partial \Omega} uv
  \end{align*}
  for all $u$ and $v\in V$.
\end{problem}
We observe that $\widehat{a}(\cdot,\cdot)$ and $b(\cdot,\cdot)$ are continuous
bilinear forms.
The bilinear form $\widehat{a}(\cdot,\cdot)$ is coercive in $V$, i.e.,
\begin{equation}
    \widehat{a}(u,u)\geq C \| u \|_{2,\Omega}^2, \quad \forall u \in V.
    \label{eq:coercivity}
\end{equation}
The proof of this property can be found in the final appendix.
Then, we consider the following source problem associated with Problem 2. 
\begin{problem}\label{weak:source}
  Given $f \in H^1(\Omega)$, find $\widetilde{u}\in V\subset
  H^2(\Omega)$ such that
  \begin{equation}
    \widehat{a}(\widetilde{u},v) =  b(f,v) \quad \forall v \in V.
  \end{equation}
\end{problem}
We note that $f|_{\partial \Omega} \in H^{1/2}(\partial \Omega)$ since
$f\in H^1(\Omega)$; the bilinear form $\widehat{a}(\cdot,\cdot)$ is
bounded and coercive.
An application of Lax-Milgram lemma implies the well-posedness of
Problem~\ref{weak:source}.
Therefore, we can introduce the \emph{solution operator} 
  $T:H^1(\O)\rightarrow H^1(\O)$ such
that $Tf=\widetilde{u}$, where $\widetilde{u}$ is the unique solution
to Problem~\ref{weak:source}.
We also observe that due to the symmetry of $\widehat{a}(\cdot,\cdot)$
and $b(\cdot,\cdot)$, the operator $T$ is self-adjoint, so that for any
$f,\,v\in V$ we state that
\begin{align*}
  \widehat{a}(T f, v)=b(f, v)=b(v, f)=\widehat{a}(Tv,f)=\widehat{a}(f,T v).
\end{align*}
Since $Tf=\widetilde{u}\in H^2(\O)\hookrightarrow H^1(\O)$, 
we deduce that $T$ is compact.
\noindent
Note that $(\lambda,\widetilde{u})\in\mathbb{R}\times V$,
$\widetilde{u}\neq 0$, is an eigenpair of Problem~\ref{pr:eigen1:mod}
if and only if $T \widetilde{u} = \mu \widetilde{u}$ with
$\mu:=\frac{1}{\lambda+1}\neq0$.
The spectral characterization of $T$ is as follows.
\begin{theorem}
  The spectrum of $T$, denoted by $\sp(T)$, decomposes as
  $\sp(T)=\{0\}\cup\{\mu_n\,:\, n\in\mathbb{N}\}$, where
  $\{\mu_n\}_{n\in\mathbb{N}}$ is a sequence of strictly positive
  eigenvalues with finite multiplicity tending to zero.
\end{theorem}
\begin{remark}
    We observe that $0$ is an eigenvalue of the operator $T$ with corresponding eigenspace equal to the kernel of the bilinear form $b$, which contains nonzero functions having vanishing trace on the boundary of $\Omega$. Moreover we emphasize that, since the bilinear form $\widehat{a}(\cdot,\cdot)$ is symmetric and positive definite, Problem~\ref{pr:eigen1:mod} admits a divergent sequence of strictly positive eigenvalues.  
\end{remark}
%

%
It is known that, if $f$ belongs to $H^1(\Omega)$, then the solution $Tf$ to the source Problem \ref{weak:source} belongs to 
$H^{2+s}(\Omega)$ for a certain $s\in (1/2,1]$. Moreover, the 
following stability estimate holds
  \begin{equation*}
      \| Tf\|_{2+s,\Omega} \leq C \| f\|_{1,\Omega}.
  \end{equation*}
  The regularity index $s$ of the solution $u=Tf$ of the source Problem \ref{weak:source} depends on the maximum re-entrant corner of the 
  computational domain $\Omega$.

\begin{remark}
  In Ref.~\cite{monzon2019virtual}, the fourth-order Steklov
  eigenvalue problem is approximated by employing $H^2$-conforming
  virtual element method.
  In that case, the weak formulation contains a boundary integration of the normal derivative of functions $u$ and $v$ as in
  $\int_{\partial\Omega}\displaystyle\partial_{\bn}u\,\partial_{\bn}v$.
  The boundary integral is thus well-defined, since the trace of $H^2$-conforming virtual element functions belongs (at least) to $H^{3/2}(\partial\Omega)$.
  In our analysis, beside a conforming virtual element approximation, we also consider a $C^0$ non-conforming space of functions with only $H^1$ global regularity, which is not enough to deal with the boundary integral presented in the formulation studied in~\cite{monzon2019virtual}. On the other hand, the boundary integral in our formulation 
  ($\int_{\partial\Omega}\displaystyle u\,v$)  
   is well-defined for functions in the $C^0$ non-conforming space.
\end{remark}

\section{Virtual Element Discretizations}
\label{Discrete:sp}
This section reviews the construction of the $H^2$-conforming and
$C^0$ non-conforming virtual element methods on general polygonal
meshes for the numerical approximation of Problem~\ref{pr:eigen1:mod}.
Conforming and non-conforming VEMs for biharmonic equations were first
developed in
Refs.~\cite{adak2023conforming,adak2022convergence,MR3002804}, and
\cite{zhao2018morley,zhao2016nonconforming}, respectively.
Herein, we consider the $H^2$-conforming and $C^0$ non-conforming
virtual element formulation presented in
References~\cite{zhao2016nonconforming,MR3002804}.

\medskip
\noindent
The virtual element formulation of Problem~\ref{pr:eigen1:mod} reads 
\begin{problem}\label{Disc_VEP_BEP}
  Find $(\lambda_h,u_h^{\dag})\in \mathbb{R}\times \Vh^{\dag}$, with
  $u_h^{\dag} \neq 0$ such that
  \begin{align*}
    \widehat{a}_h(u_h^{\dag},v_h)=(\lambda_h+1)  b(u_h^{\dag},v_h)\qquad \forall v_h\in \Vh^{\dag},
  \end{align*}
  we use the superscript $\dag\in\{c,nc\}$ to refer to the conforming
  and non-conforming discretizations.
\end{problem}
In Problem~\ref{Disc_VEP_BEP}, $(\lambda_h,u_h^{\dag})$ is the
discrete approximation of $(\lambda,u)$, and
$\widehat{a}_h(\cdot,\cdot)$ is the virtual element approximation of
$\widehat{a}(\cdot,\cdot)$.
In the rest of this section, we formulate the local, global,
conforming and non-conforming virtual element spaces and discuss the
construction of the polynomial projection operators that we use to
define the discrete bilinear form $\widehat{a}_h(\cdot,\cdot)$.
In the forthcoming sections, we will use the ``dagger'' notation for
generic definitions, properties, of theoretical results that are valid
simultaneously for the conforming and the non-conforming case.

\subsection{Mesh notation and regularity }
We denote a generic, star-shaped polygon, which can be nonconvex, by
$\E$ and its elemental diameter and boundary by $h_\E$ and
$\partial\E$.
We denote the length of a generic edge $e$ by $h_e$.
Let $\{\CTh\}_{h>0}$ be a sequence of decompositions of $\O$ composed
by non-overlapping elements $\E$.
Each decomposition $\CTh$, also called \emph{a mesh}, is labeled by
the mesh size parameter $h:=\max_{\E\in\CTh}h_\E$.  We denote the set
of mesh edges in $\CTh$ by $\EE_h:= \EEi \cup \EEbdry$, where $\EEi$
and $\EEbdry$ are the interior and boundary edge sets. Analogously, we
denote the set of mesh vertices by $\VV_h:=\VVi\cup\VVbdry$, where
$\VVi$ and $\VVbdry$ are the interior and boundary vertex
sets. Besides, we use the symbols $\bn_e$ and $\bt_{e}$ to denote the
unit normal and tangential vectors to edge $e\in\EE_h$.
Moreover, we define the space of discontinuous, piecewise polynomials
of degree $\ell$ on $\CTh$ by:
\begin{align*}
  \mathbb{P}_{\ell}(\CTh):= \big\{q\in L^2(\O): q|_{\E}\in\mathbb{P}_{\ell}(\E)\quad\forall\E\in\CTh\big\}.
\end{align*}
For all $m \in \N \cup \{0\}$, we consider the $L^2(\E)$-projection
operator $\Pi_{\E}^{m}$ onto the polynomial space
$\mathbb{P}_{m}(\E)$.
This operator is such that for each $\phi\in L^2(\E)$,
$\Pi_{\E}^{m}\phi$ is the unique polynomial of degree $m$ on $\E$
satisfying
\begin{equation}
  \left(\phi-  \Pi_{\E}^{m} \phi,q \right)_{0,\E}=0 
  \qquad\forall q\in\mathbb{P}_{m}(\E).
\end{equation}

\noindent
Next, for any integer number $t>0$, we introduce the broken Sobolev
space
\begin{align*}
  H^{t}(\CTh):= \big\{ \phi \in L^2(\O): \phi|_{\E} \in H^{t}(\E)  \quad \forall \E \in \CTh\big\},
\end{align*}
equipped with the broken seminorm and broken norm
\begin{equation}
  \label{seminorm}
  |\phi|_{t,h}:=\Bigg(  \: \sum_{\E\in\CTh}|\phi|_{t,\E}^{2} \Bigg)^{1/2};\qquad\parallel \phi\parallel_{t,h}:=\Bigg(  \: \sum_{\E\in\CTh} \parallel \phi \parallel_{t,\E}^{2} \Bigg)^{1/2}.
\end{equation}
According to~\cite{huang2021medius}, we define the jump of a function
$\phi \in H^2(\CTh)$ by $\jump{\phi}:=\phi^{+}-\phi^{-}$ on each
internal edge $e\in\EEi$, where $\phi^{\pm}$ is the trace
$\phi|_{\E^{\pm}}$, $e\subseteq\partial\E^{+}\cap\partial \E^{-}$, and
$\jump{\phi}:=\phi|_e$ on each boundary edge $e \in \EEbdry$.
Using this definition, we introduce the subspace of $H^2(\CTh)$ with weaker continuity:
\begin{multline*}
  \HncT 
  :=\bigg\{\phi_h\in H^2(\CTh)\cap H^1(\O) \,:\phi_h\text{ continuous at internal vertices, }\:\\
  \partial_{\bn}\phi_h(\vb_i)=0
  \quad\forall\vb_i\in\VVbdry,\quad\int_e\jump{\partial_{\bn_e}\phi_h}=0\quad\forall e\in\EEh\bigg\}.
\end{multline*}
For the theoretical analysis, we suppose that $\CTh$ satisfies the
following assumptions:
\begin{assumption}[Mesh Regularity]
  \label{Mesh:regularity}
  There exists a positive real number $\rho$ independent of $h$ such
  that for every $K\in\CTh$, it holds that
  \begin{itemize}
  \item[$\mathbf{(A1)}$]\textbf{star-shapedness}: $K$ is star-shaped
    with respect to an internal ball with a radius bigger than $\rho
    h_K$;
    
    \medskip
  \item[$\mathbf{(A2)}$]\textbf{uniform scaling}: the edge length
    $h_e$ for all $e\in\EE_h$ is bounded from below by $\rho h_K$,
    i.e., $h_e\geq\rho h_K$.
  \end{itemize}
\end{assumption}

\subsection{Local and global $H^2$-conforming virtual element spaces}
\label{subsec:conforming-VEM}
On each element $K \in \CTh$, we define the finite-dimensional space
\begin{multline*}
    \Vt_h^c(K) := \Big\{
    w_h\in H^2(K)\,:
    \Delta^2 w_h\in\mathbb{P}_2(K),\,
    w_h|_{\partial K}\in C^0(\partial K),\,
    w_h|_e\in\mathbb{P}_3(e)\,\,\forall e\in\partial K,\\[0.25em]
    \nabla w_h|_{\partial K}\in [C^0(\partial K)]^2, 
    \partial_{\bn_{e}} w_h|_e\in\mathbb{P}_1(e)\,\,\forall e\in\partial K 
    \Big\}.
\end{multline*}
The local space $\Vt_h^c(K)$ is associated with the
following set of linear functionals. For all $w_h \in
\Vt_h^c(K)$, we consider
\begin{itemize}
\item \DOFS{(H_1^c)}: pointwise evaluation $w_h(\vb_i)$ at the $N^K$
  vertices $\vb_i$ of the polygonal cell $\E$;
\item \DOFS{(H_2^c)}: pointwise evaluation
  $h_{\vb_i}\partial_xw_h(\vb_i)$, and $h_{\vb_i}\partial_yw_h(\vb_i)$ at
  the $N^K$ vertices $\vb_i$ of the polygonal cell $\E$,
\end{itemize}
where $h_{\vb_i}$ is a characteristic length associated with the
vertex $\vb_i$, for example, the maximum diameter of the elements
sharing such vertex.
The modified (or ``enhanced'') definition of the local space requires
the biharmonic projection operator
$\Pi^{\Delta,c}_K:\Vt_h^c(K) \rightarrow \mathbb{P}_2(K)$,
which, for all $w_h\in\Vt_h^c(K)$, is given by:
\begin{equation*}
  \begin{array}{rl}
    \displaystyle
    \int_K D^2(w_h-\Pi^{\Delta,c}_K w_h): D^2 q=0 & \forall q\in\mathbb{P}_2(K), \\[1.em]
    \displaystyle
    \int_{\partial K} (w_h-\Pi^{\Delta,c}_K w_h)q=0 & \forall q \in \mathbb{P}_1(\partial K).
  \end{array}
\end{equation*}
The operator $\Pi^{\Delta,c}_K$ is computable from the set of values
provided by the functionals \DOFS{(H_1^c)}-\DOFS{(H_2^c)}.
Using the biharmonic projection operator, we define the local virtual
elemental space
\begin{equation}
  V_h^c(K):= \bigg\{ w_h\in\Vt_h(K):\int_K(w_h-\Pi^{\Delta,c}_K w_h)q=0 \quad \forall q\in\mathbb{P}_2(K) \bigg\}.
\end{equation}
Since the space $ V_h^c(K)$ is a subspace of $\Vt_h(K)$, the
projection operator $\Pi^{\Delta,c}_K$ is well-defined and computable
using the values of the linear functional \DOFS{(H_1^c)} and
\DOFS{(H_2^c)}.
Moreover, we can choose the values of such linear operators as the
degrees of freedom that uniquely characterize the functions of
$V_h^c(K)$; see \cite[Lemma~2.3]{antonietti2016c} for further details.

Then, we introduce the global virtual element space by gluing the
local spaces $V_h^c(K)$ together in a conforming way:
\begin{equation}
  V_h^c:= \Big\{ w_h \in V : w_h|_{K} \in V_h^c(K)\,\,\forall K\in\CTh \Big\}.
  \label{Glb_sp}
\end{equation}
Such a definition incorporates the condition that the normal
derivative of the (globally defined) virtual element functions $w_h$
vanishes on the domain boundary, i.e., $\nabla w_h\cdot\mathbf{n}=0$
on $\partial\Omega$.
We obtain the degrees of freedom of the global virtual element space
$V_h^c$ by collecting all the local degrees of freedom.
\subsection{Local and global $C^0$ non-conforming virtual spaces}
Here, we describe the formulation of non-conforming virtual element
space for the model problem's discretization.
First, we define the local auxiliary space for each polygonal element
$K\in\CTh$ as
\begin{align*}
\Vt_h^{nc}(K)
:= \Big\{\phi_h\in \HdoK : \Delta^2\phi_h\in\mathbb{P}_{2}(\E),  \: \:\phi_h|_e\in\mathbb{P}_{2}(e),\: \: \Delta \phi_h|_e\in\mathbb{P}_{0}(e)\,\,\forall e\subseteq\partial\E\Big\},
\end{align*}
where $\Delta^2$ and $\Delta$ are the \emph{biharmonic} and the
\emph{Laplace} differential operators.
Then, we introduce three sets of bounded, linear functionals acting on
the virtual element functions $\phi_h\in\Vt_h^{nc}(K)$:
\begin{itemize}
\item $(\DOFS{F_v})_{v\in \VV_h^K}$: the pointwise evaluation
  $\phi_h(\vb_i)$ for every vertex $\vb_i$ of the polygon $\E$;
\item $(\DOFS{F_e^1})_{e \in \EE_h^K}$: the moments
  \begin{align*}
    \frac{1}{h_e}\int_e \phi_h  
    \qquad\forall\,\text{edge}\,\,e\in\EE_h^K;
  \end{align*}
\item $(\DOFS{F_e^2})_{e \in \EE_h^K}$: the moments 
  $$
  \quad  \int_e \partial_{\bn_{e}}\phi_h  
  \qquad \forall \, \text{edge} \, \, e \in \EE_h^K,
  $$	
\end{itemize}
where $\EE_h^{\E}$ is the set of edges forming the elemental boundary $\partial\E$. 
For every polygonal element $\E$, we define the elliptic projection
operator $\Pi^{\Delta,nc}_K:\Vt_h^{nc}(K)\longrightarrow
\mathbb{P}_2(\E)\subseteq\Vt_h^{nc}(K)$ as the solution to
the local variational problem:
\begin{equation*}
  \begin{array}{rl}
    \displaystyle
    \int_K D^2(w_h-\Pi^{\Delta,nc}_K w_h): D^2 q=0 & \forall q\in\mathbb{P}_2(K), \\[1.em]
    \displaystyle
    \int_{\partial K} (w_h-\Pi^{\Delta,nc}_K w_h)q=0 & \forall q \in \mathbb{P}_1(\partial K).
  \end{array}
\end{equation*}
The projection $\Pi^{\Delta,nc}_K \phi_h$ is computable from
the values of the functionals $(\DOFS{F_v})_{v\in \VV_h^K}$,
$(\DOFS{F_e^1})_{e \in \EE_h^K}$, and $(\DOFS{F_e^2})_{e \in\EE_h^K}$,
of $\phi_h \in \VtK$, cf.~\cite{MR3529253} for the proof.
\begin{lemma}\label{lemma-ProyDelta}
  The projection operator $\Pi^{\Delta,nc}_K :\Vt_h^{nc}(K)
  \longrightarrow \mathbb{P}_{2}(\E)$ is computable on the virtual
  element space $\Vt_h^{nc}(K)$, using only the information provided
  by the linear functionals $(\DOFS{F_v})_{v\in\VV_h^K}$,
  $(\DOFS{F_e^1})_{e\in\EE_h^K}$, and $(\DOFS{F_e^2})_{e\in\EE_h^K}$.
\end{lemma}
By employing the projection operator $\Pi^{\Delta,nc}_K$, we introduce
the enhanced, non-conforming virtual element space on each $\E\in\CTh$:
\begin{align}
  V_h^{nc}(K) := \bigg\{ v_h \in \Vt_h^{nc}(K) : 
  (v_h- \Pi^{\Delta,nc}_K v_h,q_2)_{0,\E} =0\quad\forall q_2\in \mathbb{P}_{2}(\E) \bigg\}.
  \label{local:space:nc}
\end{align}
We summarize some crucial properties that follow from the definition
of space $ V_h^{nc}(K)$.
\begin{itemize}
\item The sets of linear operators $(\DOFS{F_v})_{v\in\VV_h^K}$,
  $(\DOFS{F_e^1})_{e\in\EE_h^K}$ and $(\DOFS{F_e^2})_{e\in\EE_h^K}$
  constitutes a set of DoFs for $V_h^{nc}(K)$;
\item the operator $\Pi^{\Delta,nc}_K :V_h^{nc}(K)\longrightarrow
  \mathbb{P}_{2}(\E)$ is computable using the values provided by
  $(\DOFS{F_v})_{v\in \VV_h^K}$, $(\DOFS{F_e^1})_{e \in \EE_h^K}$ and
  $(\DOFS{F_e^2})_{e \in \EE_h^K}$;
\item $\mathbb{P}_2(\E) \subseteq V_h^{nc}(K)$;
\item For every $v_h \in V_h^{nc}(K)$, the $L^2$-orthogonal projector $\Pi^{0}_K$ onto
  $\mathbb{P}_2(\E)$ coincides with the projector
  $\Pi^{\Delta,nc}_K$, namely,
  \[
    \Pi^{0}_K v_h = \Pi^{\Delta,nc}_K v_h.
  \]
  This follows directly from the definition of the enhanced local space
  $V_h^{nc}(K)$ in \eqref{local:space:nc}.
\end{itemize}
Building on top of the local space definition, we introduce the global
non-conforming virtual space as
\begin{equation}
  \label{global:space:nc}
  V_h^{nc} = \Big\{
  v_h\in\HncT:v_h|_K\in V_h^{nc}(K)\quad\forall K\in\CTh
  \Big\}.
\end{equation}
Since $V_h^{nc}\not\subset V$, in the convergence analysis we will have to estimate also the 
\emph{non-consistency error}, as usual when dealing with non-conforming approximation.
We consider the broken semi-norm $|\cdot|_{2,h}$ defined in
\eqref{seminorm} for the convergence analysis.
%
\subsection{Construction of the discrete forms}
We discretize the forms and the load term presented in
Problem~\ref{pr:eigen1:mod} by using the projection operators defined
in the previous section.
Hereafter, we use the symbol $\dag$ to denote both labels "c" (for "conforming") and "nc" (for "non-conforming", i.e., $\dag\in\{c,nc\}$.
We outline that we can decompose the continuous form $a(\cdot,\cdot)$
as the sum of the elemental bilinear forms
$a^{\E}(\cdot,\cdot):H^2(K)\times H^2(K)\to\mathbb{R}$ such that 
\begin{align*}
  a(u,v)=\sum_{\E\in\CTh}a^\E(u,v) \; \;\forall u,v\in H^2(\Omega).
\end{align*}
We consider the following approximation of the bilinear form
$a^{\E}(\cdot,\cdot)$
\begin{equation}
  a_h^{\E}(\varphi_h,\phi_h)
  :=a^{\E}\big(\Pi^{\Delta,\dag}_K\varphi_h,\Pi^{\Delta,\dag}_K\phi_h\big)
  +\CS^{\E}\big((I-\Pi^{\Delta,\dag}_K)\varphi_h,(I-\Pi^{\Delta,\dag}_K)\phi_h\big),
  \label{eqn:bil:a:def}
\end{equation}
where $\Pi^{\Delta,\dag}_K$ is the projection operator associated with
the biharmonic bilinear form, i.e., $a^K(\cdot,\cdot)$.
The second term of the right-hand side of~\eqref{eqn:bil:a:def}, i.e.,
$\CS(\cdot,\cdot)$, can be any positive symmetric bilinear form that
satisfies
\begin{equation}
  C_{\ast} a^K(\varphi_h,\varphi_h) \leq \CS^K(\varphi_h,\varphi_h) \leq C^{\ast} a^K(\varphi_h,\varphi_h) \qquad \forall \varphi_h \in V_h^{\dag}(K), \text{with} \; \Pi^{\Delta,\dag}_K \varphi_h=0,
  \label{stab:cnc}
\end{equation}
where $C_{\ast}$, and $C^{\ast}$ are positive constants independent of
$h$ and $K$.
In our numerical experiments, we will consider the ``dofi-dofi"
formula,
\begin{equation*}
  \CS^K(\varphi_h,\nu_h):=h_k^{-2} \sum_{i,j=1}^{N^{\text{dof}}_K} \text{dof}_i(\varphi_h)\;\text{dof}_j(\nu_h),
\end{equation*}
for all $\varphi_h,\nu_h \in V_h^{\dag}(K)$, where
$N^{\text{dof}}_K$ is the total number of degrees of freedom of
$V_h^{\dag}(K)$ and $\text{dof}_i(\varphi_h)$ and
$\text{dof}_i(\nu_h)$ denote the $i$-th degrees of freedom of
$\varphi_h$ and $\nu_h$.
Additionally, we prove the polynomial consistency and stability
properties of the discrete bilinear from $a_h(\varphi_h,\phi_h)$ in
the following lemma.
\begin{lemma}
  \label{stability:bilinear}
  For each $K \in \CTh$, the bilinear form $a_h^K(\cdot,\cdot)$
  satisfies the following polynomial consistency property
  \begin{equation*}
    a_h^K(\varphi_h,p)=a^K(\varphi_h,p) \quad \forall \varphi_h \in V_h^{\dag}(K), \; p \in \mathbb{P}_2(K).
  \end{equation*}
  Moreover, $a_h^K(\cdot,\cdot)$ satisfies the stability properties,
  i.e., there exist two positive constants $C_{\alpha}$, and
  $C^{\alpha}$ such that
  \begin{equation*}
    C_{\alpha} a^K(\varphi_h,\varphi_h) \leq a^K_h(\varphi_h,\varphi_h) \leq C^{\alpha} a^K(\varphi_h,\varphi_h) \quad \forall \varphi_h \in V_h^{\dag}(K).
  \end{equation*}
\end{lemma}
By considering the local contribution, we define the global discrete
form as
\begin{equation*}
  a_h(\varphi_h,\phi_h):=\sum_{\E \in \Omega_h} a_h^{\E}(\varphi_h,\phi_h), \qquad	 \forall \varphi_h, \phi_h \in V_h^{\dag}(K).
\end{equation*}
Finally, we note that we can discretize the bilinear form on the
right-hand side as follows
\begin{equation*}
  b(u_h,v_h):=\sum_{e \subset \partial \Omega } \int_e u_h v_h,
\end{equation*}
hence without employing the projection operators since the edge trace
of a virtual element function is a polynomial.


\medskip

\begin{lemma}[Consistency error estimate]
  \label{lemma:consistency:error}
  Let Assumption~\ref{Mesh:regularity} hold.
  Then, for any $w_h,\,v_h \in V_h^{\dag}$ with piecewise polynomial
  approximations $w_{\pi},\,v_{\pi} \in \mathbb{P}_2(\CTh)$, the
  following estimate holds:
  \begin{equation}
    \label{eq:consistency:error}
    \left|
      \sum_{K \in \CTh}
      \Big(
      a_h^K(w_h - w_{\pi},\, v_h - v_{\pi})
      - a^K(w_h - w_{\pi},\, v_h - v_{\pi})
      \Big)
    \right|
    \leq C
    \left(
      \sum_{K \in \CTh} |w_h - w_{\pi}|_{2,K}^2
    \right)^{1/2}
    \left(
      \sum_{K \in \CTh} |v_h - v_{\pi}|_{2,K}^2
    \right)^{1/2},
  \end{equation}
  where $C$ is a positive constant depending only on the stability
  constants $C_{\alpha}$ and $C^{\alpha}$ of
  Lemma~\ref{stability:bilinear}.
\end{lemma}
\begin{proof}
  By the upper stability bound of $a_h^K(\cdot,\cdot)$
  (Lemma~\ref{stability:bilinear}) and the continuity of
  $a^K(\cdot,\cdot)$, together with the Cauchy--Schwarz inequality,
  we have, for each $K \in \CTh$,
  \begin{align*}
    \left| a_h^K(w_h - w_{\pi},\, v_h - v_{\pi}) \right|
    &\leq C^{\alpha}\, |w_h - w_{\pi}|_{2,K}\, |v_h - v_{\pi}|_{2,K},
    \\[0.5em]
    \left| a^K(w_h - w_{\pi},\, v_h - v_{\pi}) \right|
    &\leq |w_h - w_{\pi}|_{2,K}\, |v_h - v_{\pi}|_{2,K}.
  \end{align*}
  By the triangle inequality,
  \begin{align*}
    \left|
      a_h^K(w_h - w_{\pi},\, v_h - v_{\pi})
      - a^K(w_h - w_{\pi},\, v_h - v_{\pi})
    \right|
    &\leq (C^{\alpha} + 1)\,
      |w_h - w_{\pi}|_{2,K}\, |v_h - v_{\pi}|_{2,K}.
  \end{align*}
  Summing over all elements $K \in \CTh$ and applying the
  Cauchy--Schwarz inequality for sums
  yields~\eqref{eq:consistency:error}.
\end{proof}


\subsection{The discrete source problem}
Consider the discrete bilinear form
$\widehat{a}_h(\widetilde{u}_h,v_h)$ given by
\begin{equation*}
\widehat{a}_h(\widetilde{u}_h,v_h)=a_h(\widetilde{u}_h,v_h)+b(\widetilde{u}_h,v_h).
\end{equation*}
Applying the stability property of $a_h(\cdot,\cdot)$ from
Lemma~\ref{stability:bilinear} and the coercivity of $a(\cdot,\cdot)$,
we readily obtain the coercivity of
$\widehat{a}_h(\cdot,\cdot)$ on $V_h^c\times V_h^c$ with respect to the $\parallel \cdot \parallel_{2,h}$
norm:
\begin{equation*}
  \widehat{a}_h(\widetilde{u}_h,\widetilde{u}_h) =a_h(\widetilde{u}_h,\widetilde{u}_h)+b(\widetilde{u}_h,\widetilde{u}_h)
  \geq C |\widetilde{u}_h|_{2,h}^2+\int_{\partial \Omega} |\widetilde{u}_h|^2 \geq C \parallel \widetilde{u}_h\parallel_{2,h}^2 \quad\forall u_h\in V_h^{c}.
\end{equation*}
Further, we define for $v_h \in V_h^{nc}$
\begin{equation*}
|||v_h|||_{2,h}=| v_h|_{2,h}+\int_{\partial \Omega} |v_h|^2.
\end{equation*}
 Following Lemma~5.1 in \cite{zhao2018morley}, we can prove that $\|v_h\|_{0,h}+ |v_h|_{1,h}\leq C  ||v_h||_{2,h}$, where $C$ is a positive constant. Hence $|||\cdot|||_{2,h}$ is a norm on $V_h^{nc}$ and we can deduce the discrete coercivity as follows
\begin{equation*}
    \widehat{a}_h(\widetilde{u}_h,\widetilde{u}_h) \geq C  |||\widetilde{u}_h|||_{2,h} 
    \quad\forall u_h\in V_h^{nc}.
\end{equation*}
%
%
%
Then, we consider the discretization of Problem~\ref{weak:source}:
\begin{problem}
  \label{source:discrete}
  Given $f \in H^1(\Omega)$, Find $ \widetilde{u}_h^{\dag} \in  V_h^{\dag}$ such that
  \begin{equation}
    \label{discrete:source:issue}
    \widehat{a}_h(\widetilde{u}_h^{\dag},v_h)= b(f, v_h) \quad \forall v_h \in V_h^{\dag}.
  \end{equation}
\end{problem}
By employing the coercivity of $\widehat{a}_h(\cdot,\cdot)$, and the
Lax-Milgram theorem, we observe that the problem above is
well-posed.
In fact, given $f\in H^1(\Omega)$, we can prove the existence and
uniqueness of the approximate solution $\widetilde{u}_h^{\dag} \in
V_h^{\dag}$ to Problem~\ref{source:discrete}, which also satisfies the
stability inequality
\begin{align*}
  |\widetilde{u}_h^{\dag}|_{2,h} \leq C \| f \|_{H^1( \Omega)} \qquad \forall f\in H^1(\Omega),
\end{align*}
where the positive constant $C$ is independent of $h$.
Consequently, we can introduce the discrete solution operator as follows
\begin{align*}
  T_h: H^1(\O) &\to V^{\dag}_h \subset H^1(\O)\\[0.5em]
  f &\rightarrow T_h f:= \widetilde{u}_h^{\dag},
\end{align*}
where $\widetilde{u}_h^{\dag}$ is the unique solution of
problem~\eqref{discrete:source:issue}.
%
It is easy to check that
$(\lambda_h^{\dag}, u_h^{\dag}) \in \mathbb{R} \times V_h^{\dag} $ is
an eigenpair of Problem~\ref{Disc_VEP_BEP}
if and only if $(\mu_h^{\dag}, u_h^{\dag}) \in \mathbb{R} \times
V_h^{\dag}$ with $\mu_h^{\dag}:= \frac{1}{\lambda_h^{\dag}+1}$ is an
eigenpair of $T_h$. Moreover, from the definition of the discrete
bilinear forms $\widehat{a}_h(\cdot,\cdot)$, and $b(\cdot,\cdot)$, we
can prove that $T_h$ is self-adjoint. We conclude the section with the
spectral characterization of the operator $T_h$.
%
\begin{theorem}\label{teo:spectral_characterization}
  Let $m_h$ be the dimension of the discrete space $Z_h:=\{ u_h \in
  V_h^{\dag}: b(u_h,v_h)=0 \quad \forall v_h \in V_h^{\dag} \}$. Then
  the following result holds
  \begin{enumerate}[i)]
      \item The spectrum of $T_h$ consists of $N^\text{dof}-m_h$ real and strictly positive eigenvalues repeated according to their multiplicity. 
      \item The spectrum of $T_h$ contains the $0$ eigenvalue with multiplicity $m_h$.
  \end{enumerate}
\end{theorem}

\section{Convergence analysis of the source problem}
\label{Analysis:Source}
In this section, we introduce the theoretical framework that will be
used to carry out the convergence analysis. We begin by defining the
enriching operator, denoted as $E_h$, which maps elements from the
$C^0$ non-conforming virtual element space to the corresponding
$H^2$-conforming virtual element space. This concept is central to our study as we aim to develop a convergence analysis of the source problem  \eqref{discrete:source:issue} with a minimal regularity solution, as in the case of a non-convex domain. A concise overview of the enriching operator and its approximation properties
will be provided herein.


\subsection{Enriching Operator}
\label{enriching:operator}
In this section, we recall the definition of the \emph{enriching operator}
$E_h:V_h^{nc}\rightarrow V_h^{c}$ and its main property,
cf. References~\cite{huang2021medius} (VEM) and
\cite{brenner2013morley} (FEM).
We denote the \emph{patch of the vertex $\chi$}, i.e., the union of
all the elements of $\CTh$ sharing $\chi$ by $\mathcal{W}(\chi)$, and
the total number of elements of $\mathcal{W}(\chi)$ by 
$\mathcal{Z}(\chi)$.
We note that $E_h(\phi_h)$ belongs to the conforming space $V_h^c$ for
every $\phi_h\in V_h^{nc}$.
In view of the definitions of
$\DOFS{H_1^c}$ and $\DOFS{H_2^c}$, cf.
Subsection~\ref{subsec:conforming-VEM}, we uniquely characterize $E_h(\phi_h)$ by specifying:
\begin{itemize}
\item $\DOFS{H_1^c}(E_h(\phi_h))$: the function pointwise value
  $\frac{1}{\mathcal{Z}(\chi)} \sum_{\widehat{K} \in
  \mathcal{W}(\chi)} \Pi^0 \phi_h|_{\widehat{K}}(\chi)$;
\item $\DOFS{H_2^c}(E_h(\phi_h))$: the gradient pointwise values
  $\frac{1}{\mathcal{Z}(\chi)} \sum_{\widehat{K} \in
  \mathcal{W}(\chi)} h_{\chi} \nabla(\Pi^0
  \phi_h|_{\widehat{K}})(\chi)$.
\end{itemize}
In the definitions above, $h_{\chi}$ is a characteristic length
associated with the vertex $\chi$ and $\Pi^0$ is the piecewise 
$L^2$-orthogonal projector discussed in
Subsection~\ref{subsec:conforming-VEM}.
Then, we consider an arbitrary numbering of these degrees af freedom,
which we refer to as $\mbox{\textbf{dof}}_{i}(E_h(\phi_h))$ for
$i=1,2,\ldots,N_C^{\text{dof}}$, $N_C^{\text{dof}}$ being the
dimension of $V_h^c$.
Finally, writing $E_h(\phi_h)(x)$ in terms of the canonical basis functions $\psi_i$
corresponding to $\mbox{\textbf{dof}}_{i}(E_h(\phi_h))$ yields:
\begin{align*}
  E_h(\phi_h)(x)
  := \sum_{i=1}^{N_C^{\text{dof}}} \mbox{\textbf{dof}}_{i}\big(E_h(\phi_h)\big)\psi_i(x)\quad\forall\phi_h\in V_h^{nc}.
\end{align*}
We recall the approximation properties of $E_h$ in the following lemma (see,~\cite[Lemma 4.9-4.10]{adak2023morley}). 
\begin{lemma}
  \label{Eh:ADEh:ADI}~
  
  \begin{itemize}
  \item for all $ v_h \in \Vh^{nc}$, there exists $C>0$ independent of
    $h$, such that
    \begin{equation*}
      \| v_h -E_h v_h \|_{0,\O}+h| v_h -E_h v_h |_{1,\Omega}+ h^2|E_h v_h |_{2,\O} \leq Ch^2| v_h |_{2,h};
    \end{equation*}	
  \item let $w\in H^{2+s}(\Omega)$, with $s\in [0,1]$. Then, for all
    $v_h\in \Vh^{nc}$ we have
    \begin{equation*}
      a(w,v_h-E_hv_h)\leq Ch^s||w||_{2+s,\Omega}|v_h|_{2,h};
    \end{equation*}
  \end{itemize}
\end{lemma}
The following results summarize some fundamental approximation properties
helpful in deriving the a priori error estimates. 
\begin{lemma}
For $\phi \in H^{2+s}(\Omega)$, and $\chi \in H^{2+s}(\Omega)$
with $s \in [0,1]$, it holds
\begin{equation*}
    a(\phi, \chi-\chi_I) \leq C h^{2s}~|\phi|_{2+s,\Omega} |\chi|_{2+s,\Omega},
\end{equation*}
where $\chi_I$ is the interpolant of $\chi$ in the virtual element space $V_h^{nc}$.
\end{lemma}
We refer to Lemma 4.11 in~\cite{adak2023morley} for the details of the proof. 
\begin{lemma}
  \label{Int_poly}
  Let Assumption~\ref{Mesh:regularity} hold. Then, 
  \begin{itemize}
  \item for every $\phi\in H^{2+t}(K)$ with $t\in[0,1]$, there exists
    $\phi_{\pi}\in\mathbb{P}_2(K)$ and generic constant $C>0$
    (independent of $h$) such that:
    \begin{equation*}
      \| \phi-\phi_{\pi} \|_{l,K}\leq C h_K^{2+t-l} |\phi|_{2+t,K}, \quad l=0,1,2
    \end{equation*}
    for all $\E\in\CTh$;
  \item for every $\phi\in H^{2+t}(\Omega)$ with $t\in[0,1]$, there exists $\phi_I\in V_h^{\dag}$, and $C>0$ (independent of $h$) such that
    \begin{equation*}
      \| \phi-\phi_{I} \|_{l,K}\leq C h_K^{2+t-l} |\phi|_{2+t,K}, \quad l=0,1,2
    \end{equation*}
    for all $\E\in\CTh$;
  \end{itemize}
\end{lemma}

\subsection{A priori error estimates for the source problem in conforming method}
In this section, we carry out the a priori analysis of the conforming
VEM for the source problem.
\begin{theorem}
  \label{Con_Anal_H2}
  Under mesh assumptions $({\bf A_1})-({\bf A_2})$, let $\widetilde{u}$ and $\widetilde{u}_{h}$ be the unique solutions to Problems~\ref{weak:source} and \ref{source:discrete}, respectively.
  Then, there exists a positive constant $C$, independent of $h$, such that

  \begin{equation}
     \|\widetilde{u}-\widetilde{u}_h\|_{0,\Omega}
    + |\widetilde{u}-\widetilde{u}_h|_{1,\Omega}
    + h^s\parallel \widetilde{u}-\widetilde{u}_h\parallel_{2,\Omega}
    \leq C h^{2s} |\widetilde{u}|_{2+s,\Omega}. 
  \end{equation}

\end{theorem}

\begin{proof}~\\[-\baselineskip]

  \noindent
  $\bullet$ \textit{\bf $H^2$ estimate}.
  Adding and subtracting the interpolation $\widetilde{u}_I$, we split
  the approximation error $\widetilde{u}-\widetilde{u}_h$ as follows:
  $\widetilde{u}-\widetilde{u}_h=\widetilde{u}-\widetilde{u}_I+\widetilde{u}_I-\widetilde{u}_h$.
  We estimate the first term by applying Lemma~\ref{Int_poly}.
  We estimate the second term by using the same argument of
  \cite{adak2023conforming}.
  
  By using the coercivity of the bilinear form
  $\widehat{a}(\cdot,\cdot)$, we derive that
  \begin{equation}
  \label{main:H2:conf}
      \begin{split}
          C \parallel \widetilde{u}_h-\widetilde{u}_I
            \parallel_{2,\Omega}^2
          &\leq \widehat{a}(\widetilde{u}_h-\widetilde{u}_I,
                            \widetilde{u}_h-\widetilde{u}_I) \\
          &= \widehat{a}(\widetilde{u}_h,
                         \widetilde{u}_h-\widetilde{u}_I)
            -\widehat{a}(\widetilde{u}_I,
                         \widetilde{u}_h-\widetilde{u}_I).
      \end{split}
  \end{equation}
  We expand $\widehat{a} = a + b$ in both terms on the right-hand side
  of~\eqref{main:H2:conf} and set
  $\eta := \widetilde{u}_h - \widetilde{u}_I$ for brevity.
  For the first term, we write
  $\widehat{a}(\widetilde{u}_h, \eta)
  = a(\widetilde{u}_h, \eta) + b(\widetilde{u}_h, \eta)$.
  Adding and subtracting $a_h(\widetilde{u}_h, \eta)$ and using the
  discrete source
  equation~\eqref{discrete:source:issue},
  $a_h(\widetilde{u}_h, \eta) + b(\widetilde{u}_h, \eta) = b(f, \eta)$,
  we obtain
  \begin{equation*}
    \widehat{a}(\widetilde{u}_h, \eta)
    = b(f, \eta)
      + \big[ a(\widetilde{u}_h, \eta) - a_h(\widetilde{u}_h, \eta) \big].
  \end{equation*}
  Since $\eta \in V_h^c \subset V$, the continuous source equation
  gives $b(f, \eta)
  = \widehat{a}(\widetilde{u}, \eta)
  = a(\widetilde{u}, \eta) + b(\widetilde{u}, \eta)$.
  For the second term in~\eqref{main:H2:conf}, we similarly expand
  $\widehat{a}(\widetilde{u}_I, \eta)
  = a(\widetilde{u}_I, \eta) + b(\widetilde{u}_I, \eta)$.
  Substituting and rearranging, we find
  \begin{equation*}
    \widehat{a}(\widetilde{u}_h, \eta)
    - \widehat{a}(\widetilde{u}_I, \eta)
    = b(\widetilde{u} - \widetilde{u}_I, \eta)
      + a(\widetilde{u} - \widetilde{u}_I, \eta)
      + \big[ a(\widetilde{u}_h, \eta) - a_h(\widetilde{u}_h, \eta) \big].
  \end{equation*}
  We rewrite the sum of the last two terms element by element.
  By the polynomial consistency of $a_h^K$ (Lemma~\ref{stability:bilinear}),
  we have
  $a^K(\widetilde{u}_h, \eta) - a_h^K(\widetilde{u}_h, \eta)
  = -\big[a_h^K(\widetilde{u}_h - \widetilde{u}_{\pi}, \eta)
  - a^K(\widetilde{u}_h - \widetilde{u}_{\pi}, \eta)\big]$.
  Adding and subtracting
  $a^K(\widetilde{u}_I - \widetilde{u}_{\pi}, \eta)$
  inside the sum over elements, and recombining with
  $a(\widetilde{u} - \widetilde{u}_I, \eta)$, we obtain
  \begin{equation}
  \label{main:H2:conf:decomp}
      \begin{split}
          C \parallel \widetilde{u}_h-\widetilde{u}_I
            \parallel_{2,\Omega}^2
          &\leq
            b(\widetilde{u}-\widetilde{u}_I,\,
              \widetilde{u}_h-\widetilde{u}_I)
          \\
          &\quad
          - \sum_{K\in \CTh}
            \Big(
              a_h^K(\widetilde{u}_I-\widetilde{u}_{\pi},\,
                     \widetilde{u}_h-\widetilde{u}_I)
              - a^K(\widetilde{u}_I-\widetilde{u}_{\pi},\,
                     \widetilde{u}_h-\widetilde{u}_I)
            \Big)
          \\
          &\quad
          - \sum_{K\in \CTh}
            a^K(\widetilde{u}_I-\widetilde{u},\,
                \widetilde{u}_h-\widetilde{u}_I).
      \end{split}
  \end{equation}
  By using the trace inequality and Lemma~\ref{Int_poly}, we derive that
  \begin{equation}
  \label{main:trace:conf}
      \begin{split}
          \vert b(\widetilde{u}-\widetilde{u}_I,\,
                   \widetilde{u}_h-\widetilde{u}_I) \vert
          &\leq C \,
            |\widetilde{u}-\widetilde{u}_I|_{1,\Omega}\,
            |\widetilde{u}_h-\widetilde{u}_I|_{1,\Omega}
          \\
          &\leq C \, h^{1+s}\,
            |\widetilde{u}|_{2+s,\Omega}\,
            \|\widetilde{u}_h-\widetilde{u}_I\|_{2,h}.
      \end{split}
  \end{equation}
  For the second term
  in~\eqref{main:H2:conf:decomp}, we apply
  Lemma~\ref{lemma:consistency:error} with
  $w_h = \widetilde{u}_I$,
  $w_{\pi} = \widetilde{u}_{\pi}$,
  $v_h = \widetilde{u}_h - \widetilde{u}_I$,
  and $v_{\pi} = 0$
  (noting that $0 \in \mathbb{P}_2(K)$).
  This yields
  \begin{equation}
  \label{H2:consistency}
  \begin{split}
    \left|
      \sum_{K \in \CTh}
      \Big(
        a_h^K(\widetilde{u}_I-\widetilde{u}_{\pi},\,
               \widetilde{u}_h-\widetilde{u}_I)
        - a^K(\widetilde{u}_I-\widetilde{u}_{\pi},\,
               \widetilde{u}_h-\widetilde{u}_I)
      \Big)
    \right|
    &\leq C \,
      |\widetilde{u}_I-\widetilde{u}_{\pi}|_{2,h}\,
      |\widetilde{u}_h-\widetilde{u}_I|_{2,h}
    \\
    &\leq C \, h^s \,
      |\widetilde{u}|_{2+s,\Omega}\,
      \|\widetilde{u}_h-\widetilde{u}_I\|_{2,h},
  \end{split}
  \end{equation}
  where in the last step we used the triangle inequality
  $|\widetilde{u}_I-\widetilde{u}_{\pi}|_{2,h}
  \leq |\widetilde{u}_I-\widetilde{u}|_{2,h}
  + |\widetilde{u}-\widetilde{u}_{\pi}|_{2,h}$
  together with Lemma~\ref{Int_poly}.
  For the third term in~\eqref{main:H2:conf:decomp},
  the Cauchy--Schwarz inequality and Lemma~\ref{Int_poly} give
  \begin{equation}
  \label{H2:continuous}
      \left|
        \sum_{K \in \CTh}
        a^K(\widetilde{u}_I-\widetilde{u},\,
             \widetilde{u}_h-\widetilde{u}_I)
      \right|
      \leq
        |\widetilde{u}_I-\widetilde{u}|_{2,h}\,
        |\widetilde{u}_h-\widetilde{u}_I|_{2,h}
      \leq C \, h^s \,
        |\widetilde{u}|_{2+s,\Omega}\,
        \|\widetilde{u}_h-\widetilde{u}_I\|_{2,h}.
  \end{equation}
  Inserting~\eqref{main:trace:conf}, \eqref{H2:consistency},
  and~\eqref{H2:continuous}
  into~\eqref{main:H2:conf:decomp}, dividing both sides by
  $\|\widetilde{u}_h-\widetilde{u}_I\|_{2,h}$, and combining
  with the interpolation estimate for
  $\|\widetilde{u}-\widetilde{u}_I\|_{2,h}$ from
  Lemma~\ref{Int_poly}, we derive that
  \begin{equation}
  \label{H2:final:conf}
      \parallel \widetilde{u}-\widetilde{u}_h
        \parallel_{2,\Omega}
      \leq C \, h^s \,
        \vert \widetilde{u} \vert_{2+s,\Omega}.
  \end{equation}

  $\bullet$
  \textit{\bf $H^1$ estimates}.
  Using the $H^2$ estimate, we obtain an upper bound for the approximation error in the $H^1$ norm. In this direction, we first define the auxiliary variational problem: Find $\phi \in V$ such that
  \begin{equation}\label{label:H1}
    \widehat{a}(\phi,\omega)=\int_{\Omega} \nabla (\widetilde{u}-\widetilde{u}_h) \cdot \nabla \omega  \quad \forall \omega \in V.
  \end{equation}
  Since the bilinear form $\widehat{a}(\cdot,\cdot)$ is coercive, problem \eqref{label:H1} is well-posed. Classical regularity result states that
  \begin{equation}
    \|\phi\|_{2+s, \Omega} \leq C \| \nabla (\widetilde{u}-\widetilde{u}_h) \|_{0,\Omega}.
    \label{eq:Th41:regularity}
  \end{equation}
  Since $\widetilde{u}_h \in V_h^c \subset V$, we choose
  $\omega=\widetilde{u}-\widetilde{u}_h$ in \eqref{label:H1}, we add and
  subtract $\phi_I$, and we obtain
  \begin{align}
   |\widetilde{u}-\widetilde{u}_h|_{1,\Omega}^2
      = \widehat{a}(\phi,\widetilde{u}-\widetilde{u}_h)
    = \widehat{a}(\phi-\phi_I,\widetilde{u}-\widetilde{u}_h) + \widehat{a}(\phi_I,\widetilde{u}-\widetilde{u}_h).
    \label{first_H1_conf}
  \end{align}
  We estimate separately both terms of \eqref{first_H1_conf} as follows
  \begin{equation*}
    \widehat{a}(\phi-\phi_I,\widetilde{u}-\widetilde{u}_h)=a(\phi-\phi_I,\widetilde{u}-\widetilde{u}_h)+b(\phi-\phi_I,\widetilde{u}-\widetilde{u}_h).
  \end{equation*}
  By employing the approximation properties of the interpolation
  operator, cf. Lemma~\ref{Int_poly}, and the regularity
  inequality~\eqref{eq:Th41:regularity}, and the estimate $\| \widetilde{u}-\widetilde{u}_h \|_{2,\Omega}$, we obtain
  \begin{equation}
  \label{H1:conf:a}
  \begin{split}
    a(\phi-\phi_I,\widetilde{u}-\widetilde{u}_h) &\leq c |\phi-\phi_I|_{2,\Omega} |\widetilde{u}-\widetilde{u}_h|_{2,\Omega} \\
    & \leq C h^s |\phi|_{2+s,\Omega}~ h^s |\widetilde{u}|_{2+s,\Omega} \\
    & \leq C h^{2s} |\widetilde{u}-\widetilde{u}_h|_{1,\Omega} |\widetilde{u}|_{2+s,\Omega}.
  \end{split}
  \end{equation}
  Further, an application of the trace inequality,
  Lemma~\ref{Int_poly}, and $H^2-$ estimate, we derive that
  \begin{equation}
    \label{H1:conf:b}
    \begin{split}
      b(\phi-\phi_I,\widetilde{u}-\widetilde{u}_h) &\leq C |\phi-\phi_I|_{1,\Omega} |\widetilde{u}-\widetilde{u}_h|_{1,\Omega} \\
      & \leq C h^{1+s} |\phi|_{2+s,\Omega} \|\widetilde{u}-\widetilde{u}_h\|_{2,\Omega}  \\
      & \leq C h^{1+2s} |\widetilde{u}-\widetilde{u}_h|_{1,\Omega} |\widetilde{u}|_{2+s,\Omega}.
    \end{split}
  \end{equation}
  
  To estimate the last term of~\eqref{first_H1_conf}, we use the
  continuous and discrete source equations.
  Since
  $\widehat{a}(\widetilde{u}, \phi_I) = b(f, \phi_I)$
  and
  $\widehat{a}_h(\widetilde{u}_h, \phi_I) = b(f, \phi_I)$,
  we obtain
  \begin{equation}
  \label{H1:conf:consistency}
    \begin{split}
      \widehat{a}(\phi_I,\widetilde{u}-\widetilde{u}_h)
      &= \widehat{a}(\widetilde{u},\phi_I)
         - \widehat{a}(\widetilde{u}_h,\phi_I)
       = \widehat{a}_h(\widetilde{u}_h,\phi_I)
         - \widehat{a}(\widetilde{u}_h,\phi_I)
      \\
      &= a_h(\widetilde{u}_h,\phi_I)
         - a(\widetilde{u}_h,\phi_I)
      \\
      &= \sum_{K \in \CTh}
         \Big(
           a_h^K(\widetilde{u}_h-\widetilde{u}_{\pi},\,
                  \phi_I-\phi_{\pi})
           - a^K(\widetilde{u}_h-\widetilde{u}_{\pi},\,
                  \phi_I-\phi_{\pi})
         \Big),
    \end{split}
  \end{equation}
  where in the last step we used the polynomial consistency of
  $a_h^K(\cdot,\cdot)$ (Lemma~\ref{stability:bilinear}).
  An application of Lemma~\ref{lemma:consistency:error} with
  $w_h = \widetilde{u}_h$,
  $w_{\pi} = \widetilde{u}_{\pi}$,
  $v_h = \phi_I$,
  $v_{\pi} = \phi_{\pi}$
  yields
  \begin{equation*}
    \left|
      \widehat{a}(\phi_I,\widetilde{u}-\widetilde{u}_h)
    \right|
    \leq C \,
      |\widetilde{u}_h-\widetilde{u}_{\pi}|_{2,h}\,
      |\phi_I-\phi_{\pi}|_{2,h}.
  \end{equation*}
  By the triangle inequality and the $H^2$ estimate~\eqref{H2:final:conf},
  \begin{equation*}
    |\widetilde{u}_h-\widetilde{u}_{\pi}|_{2,h}
    \leq |\widetilde{u}_h-\widetilde{u}|_{2,h}
         + |\widetilde{u}-\widetilde{u}_{\pi}|_{2,h}
    \leq C \, h^s \,
      |\widetilde{u}|_{2+s,\Omega},
  \end{equation*}
  and, by Lemma~\ref{Int_poly},
  $|\phi_I-\phi_{\pi}|_{2,h}
  \leq C \, h^s \, |\phi|_{2+s,\Omega}$.
  Combining these bounds with the regularity
  estimate~\eqref{eq:Th41:regularity}, we conclude that
  \begin{equation*}
    \left|
      \widehat{a}(\phi_I,\widetilde{u}-\widetilde{u}_h)
    \right|
    \leq C \, h^{2s} \,
      |\widetilde{u}|_{2+s,\Omega}\,
      |\widetilde{u}-\widetilde{u}_h|_{1,\Omega}.
  \end{equation*}
  
  By inserting \eqref{H1:conf:a}, \eqref{H1:conf:b}, and
  \eqref{H1:conf:consistency} into \eqref{first_H1_conf}, and assuming
  $|\widetilde{u}-\widetilde{u}_h|_{1,\Omega} \neq 0$, we derive the
  estimate.

  
  $\bullet$ \textit{\bf $L^2$ estimates}.
  To prove the error estimate in $L^2$ norm, we first define the auxiliary problem as: find $\phi \in V$ such that
  \begin{equation}
  \label{aux:L2:con}
      \widehat{a}(\phi,\omega)=\int_{\Omega} (\widetilde{u}-\widetilde{u}_h) \omega  \quad \forall \omega \in V.
  \end{equation}

The regularity result for the biharmonic problem states the following inequality
\begin{equation}
\label{Regularity:L2}
    \|\phi\|_{2+s,\Omega} \leq C \|\widetilde{u}-\widetilde{u}_h \|_{0,\Omega}.
\end{equation}
Now, by considering $\omega=\widetilde{u}-\widetilde{u}_h$ in \eqref{aux:L2:con}, and an application of analogous arguments and the regularity mentioned above in \eqref{Regularity:L2}, we derive the estimate in the $L^2$-norm as
\begin{equation*}
    \|\widetilde{u}-\widetilde{u}_h \|_{0,\Omega} \leq C h^{2s} |\widetilde{u}|_{2+s,\Omega}. 
\end{equation*}
\end{proof}
Now, we prove the error estimates for the non-conforming scheme.  
%
%
\subsection{A priori error estimates for the source problem approximated by the non-conforming method}
\label{Apriori_Error}
In this subsection, we present the convergence analysis of the
non-conforming approximation of Problem~\ref{weak:source}.
To this end, we employ the \emph{enriching operator} $E_h$ defined in Section~\ref{enriching:operator}.
Since the discrete space is non-conforming, a variational crime occurs
and we need to estimate a non-conformity error, which is defined as
follows:
\begin{equation}
  \mathcal{N}_h(\widetilde{u},v_h)
  := \sum_{K\in\CTh} \widehat{a}^K(\widetilde{u},v_h)-b(f,v_h)\qquad \qquad \forall v_h \in V_h^{nc}.
  \label{consistency:error}
\end{equation} 
We bound the non-conformity error by employing the enriching operator
as in the lemma below.
\begin{lemma}
  Let $f$ belong to $H^1(\Omega) $
  and
  $\widetilde{u} \in H^{2+s}(\Omega)$ be the solution of
    Problem~\ref{weak:source}.
  Then there exists a constant $C>0$,
  independent of $h$, such that
  \begin{align*}
    \mathcal{N}_h(\widetilde{u},v_h) \leq 
     Ch^{\min(1,s)} \Big( \|\widetilde{u}\|_{2+s,\O} + \|f\|_{1,\O} \Big) |v_h|_{2,h}\quad\forall v_h\in\Vh^{nc},
  \end{align*}
  where $\mathcal{N}_h(\widetilde{u}, \cdot)$ is the consistency error defined in~\eqref{consistency:error}. 
\end{lemma}
\begin{proof}
  Upon employing the enriching operator $E_h$, we rewrite the term
  $\mathcal{N}_h(\widetilde{u},v_h)$, as follows
  \begin{align*}
    \mathcal{N}_h(\widetilde{u},v_h)
    &= \sum_{K\in\CTh} \widehat{a}^K(\widetilde{u},v_h-E_h v_h)-b(f,v_h-E_h v_h) \\
    &= \sum_{K\in\CTh} a^K(\widetilde{u},v_h-E_h v_h) + b(\widetilde{u},v_h-E_h v_h) - b(f,v_h-E_h v_h) \\
    &\leq C \Big( h^{s}\|\widetilde{u}\|_{2+s,\Omega}\,|v_h|_{2,h} + \|\widetilde{u}\|_{\half,\partial\Omega} \|v_h-E_h v_h\|_{\half,\partial\Omega} + \|f\|_{\half,\partial\Omega} \|v_h-E_h v_h\|_{\half,\partial\Omega} \Big) \\
    &\leq C \Big( h^{s}\|\widetilde{u}\|_{2+s,\Omega}\,|v_h|_{2,h} + \|\widetilde{u}\|_{1,\Omega} \|v_h-E_h v_h\|_{1,\Omega} + \|f\|_{1,\Omega} \|v_h-E_h v_h\|_{1,\Omega} \Big) \\
    &\leq C \Big( h^{s}\|\widetilde{u}\|_{2+s,\Omega}\,|v_h|_{2,h} + h \|\widetilde{u}\|_{1,\Omega} |v_h|_{2,h} + h\|f\|_{1,\Omega} |v_h|_{2,h} \Big) \\
    &  \leq C \Big( h^s \|\widetilde{u}\|_{2+s,\Omega} + h\|\widetilde{u}\|_{1,\Omega}  + h\|f\|_{1,\Omega} \Big) |v_h|_{2,h}.
  \end{align*}
  %
\end{proof}
\begin{remark}
  \label{General_order}
  The $C^0$ non-conforming VEM for the plate bending problem was first
  proposed in Ref.~\cite{zhao2016nonconforming}, where the authors
  carried out the a~prior analysis for any order of accuracy on a
  convex domain.
  In particular, for the lowest-order virtual element space, i.e., for
  $k = 2$, the authors of Ref.~\cite{zhao2016nonconforming} assumed
  that the analytical solution $u$ belongs to $H^3(\Omega)$.
  In our paper, we assume that the computational domain may have a
  nonconvex shape, so that 
  the analytical solution $u$ belongs to
  $H^{2+s}(\Omega)$, with $s \in (1/2,1]$.
  To prove the optimal order of convergence, i.e., $O(h^s)$, we have
  employed different arguments using an enriching operator from the
  $C^0$ non-conforming space to its $H^2$-conforming counterpart.
\end{remark}
In what follows, we will prove some preliminary results to establish
that the operator $T_{h} $ converges to $T $ when $h$ goes to
zero.
First, we have the Strang-type result of Theorem~\ref{Noncon_anal},
where we can identify $\widetilde{u}_I$ with the interpolation of $u$
in the approximation space $\Vh^{nc}$ as defined in
Lemma~\ref{Int_poly}.
\begin{theorem}
  \label{Noncon_anal}
  Under Mesh Assumptions $\mathbf{(A_1)}-\mathbf{(A_2)}$, we let $\widetilde{u}$
  and $\widetilde{u}_{h}$ be the unique solutions to
  Problems~\ref{weak:source} and \ref{source:discrete}, respectively.
  Then, there exists a positive constant $C$, independent of $h$, such
  that
  \begin{equation}
    \label{strang:type}
    \|\widetilde{u}-\widetilde{u}_h\|_{2,h}
    \leq C\Big(
    |\widetilde{u}-\widetilde{u}_I|_{1,\Omega} +
    |\widetilde{u}-\widetilde{u}_I|_{2,h} +
    |\widetilde{u}-\widetilde{u}_{\pi}|_{2,h} +
    \underset{\underset{\phi_h \neq 0}{\phi_h \in V_h^{nc}} }{\sup}\frac{\mathcal{N}(\widetilde{u},\phi_h)}{\|\phi_h\|_{2,h}}
    \Big), 	
  \end{equation}
  for any approximation $\widetilde{u}_I$ in $\Vh^{nc}$ and
  $\widetilde{u}_{\pi}$ in $\mathbb{P}_2(\CTh)$, of $\widetilde{u}$. Further, by employing the approximation property of $u_I$, $u_{\pi}$, we derive that
  \begin{align*}
    \|\widetilde{u}-\widetilde{u}_h\|_{2,h} \leq C h^{s}
    \|\widetilde{u}\|_{2+s,\Omega}.
  \end{align*}
\end{theorem}
%
\begin{proof}
  We add and subtract the interpolated field $\widetilde{u}_I$ and we
  split the error as follows:
  $\widetilde{u}-\widetilde{u}_h:=\widetilde{u}-\widetilde{u}_I+\widetilde{u}_I-\widetilde{u}_h$.
  We bound the term $\widetilde{u}-\widetilde{u}_I$ by employing the
  approximation properties of $\widetilde{u}_I$, cf.
  Lemma~\ref{Int_poly}.
  Hence, we focus on the term $\eta:=\widetilde{u}_h-\widetilde{u}_I$.
  Using the coercivity of the bilinear form
  $\widehat{a}(\cdot,\cdot)$, we obtain that
%
  \begin{align*}
    \hspace{-5cm}
    &
    C(C_\alpha) \|\widetilde{u}_h-\widetilde{u}_I\|^2_{2,h} \leq \widehat{a}_h(\widetilde{u}_h - \widetilde{u}_I, \eta) = \widehat{a}_h(\widetilde{u}_h, \eta)-\widehat{a}_h(\widetilde{u}_I, \eta)\\
    &\quad = b(f,\eta)-\sum_{K\in \Omega_h} a_h^K(\widetilde{u}_I, \eta)-b(\widetilde{u}_I,\eta) \\
    &\quad = b(f,\eta)-\sum_{K\in \Omega_h} \Big( a_h^K(\widetilde{u}_I-\widetilde{u}_{\pi}, \eta) + a_h^K(\widetilde{u}_{\pi}, \eta) \Big) -b(\widetilde{u}_I,\eta) \\
    &\quad = b(f,\eta)-\sum_{K\in \Omega_h} \Big( a_h^K(\widetilde{u}_I-\widetilde{u}_{\pi}, \eta) + a^K(\widetilde{u}_{\pi}, \eta) \Big) -b(\widetilde{u}_I,\eta) \\
    &\quad = -b(\widetilde{u}_{I},\eta)-\sum_{K\in \Omega_h} \Big( a_h^K(\widetilde{u}_I-\widetilde{u}_{\pi}, \eta) + a^K(\widetilde{u}_{\pi}-\widetilde{u}, \eta) \Big) - \bigg(  \sum_{K\in \Omega_h} a^K(\widetilde{u}, \eta) - b(f,\eta) \bigg) \\
    &\quad = -b(\widetilde{u}_{I},\eta)-\sum_{K\in \Omega_h} \Big( a_h^K(\widetilde{u}_I-\widetilde{u}_{\pi}, \eta) + a^K(\widetilde{u}_{\pi}-\widetilde{u}, \eta) \Big) - \bigg(  \sum_{K\in \Omega_h} \widehat{a}^K(\widetilde{u}, \eta) -b(\widetilde{u},\eta) - b(f,\eta) \bigg) \\
    &\quad =  b(\widetilde{u},\eta)-b(\widetilde{u}_{I},\eta)-\sum_{K\in \Omega_h} \Big( a_h^K(\widetilde{u}_I-\widetilde{u}_{\pi}, \eta) + a^K(\widetilde{u}_{\pi}-\widetilde{u}, \eta) \Big) - \bigg(  \sum_{K\in \Omega_h} \widehat{a}^K(\widetilde{u}, \eta) - b(f,\eta) \bigg) \\
    &\quad = b(\widetilde{u}-\widetilde{u}_I, \eta)-\sum_{K\in \Omega_h} \Big( a_h^K(\widetilde{u}_I-\widetilde{u}_{\pi}, \eta) + a^K(\widetilde{u}_{\pi}-\widetilde{u},\eta) \Big) - \mathcal{N}(\widetilde{u},\eta)\\
    &\quad \leq \Big( |\widetilde{u}-\widetilde{u}_I|_{1,\Omega}+|\widetilde{u}_I-\widetilde{u}_{\pi}|_{2,h}+|\widetilde{u}_{\pi}-\widetilde{u}|_{2,h} \Big)|\eta|_{2,h} - \mathcal{N}(\widetilde{u},\eta)\\
    &\quad \leq \Big(|\widetilde{u}-\widetilde{u}_I|_{1,\Omega} + |\widetilde{u}_I-\widetilde{u}|_{2,h} + 2|\widetilde{u}-\widetilde{u}_{\pi}|_{2,h}  \Big)\,|\eta|_{2,h} - \mathcal{N}(\widetilde{u},\eta).
  \end{align*}
  We note that
  \begin{align*}
    \mathcal{N}(\widetilde{u},\eta)
    \leq C h^s \|\widetilde{u}\|_{2+s,\Omega} |\eta|_{2,\Omega}.
  \end{align*}
  By inserting the approximation properties of the interpolation
  operator $\widetilde{u}_I$ and exploiting the polynomial
  approximation property of $\widetilde{u}_{\pi}$ (Lemma~\ref{Int_poly}), we derive the order
  of convergence in the $\vert \cdot \vert_{2,\Omega}$ norm as follows
  \begin{align}
    |\widetilde{u}-\widetilde{u}_h|_{2,h} \leq C h^{s}
    \|\widetilde{u}\|_{2+s,\Omega}.
  \end{align}
\end{proof}
$\bullet$ \textit{\bf Convergence of the source problem in $H^1$ norm}.
%
Now, we conduct the convergence analysis in the $H^1$
norm. To perform the convergence analysis, we employ the enrichment operator $E_h$. 
\begin{theorem}
\label{H1:Nonconf:Th}
  Let Assumption~\ref{Mesh:regularity} be satisfied. Further, let
  $\widetilde{u}$, and $\widetilde{u}_h$ be the unique solutions to
  Problems~\ref{weak:source}, and \ref{Disc_VEP_BEP}
  respectively. Then,
  there exists a positive constant $C$, independent of $h$ such that
  \begin{equation}
    |\widetilde{u}-\widetilde{u}_h|_{1,\Omega} \leq C h^{2s}|\widetilde{u}|_{2+s,\Omega}.
  \end{equation}
\end{theorem}

\begin{proof}
By adding and subtracting the interpolated field $\widetilde{u}_{I}$,
we split the error as follows
\begin{equation}
  \widetilde{u}_h-\widetilde{u}=\widetilde{u}_h-\widetilde{u}_{I}+\widetilde{u}_I-\widetilde{u}=\widetilde{u}_I-\widetilde{u}+\delta_h-E_h \delta_h+E_h \delta_h,
\end{equation} 
$\delta_h=\widetilde{u}_h-\widetilde{u}_{I} $.
By employing the triangle inequality and Lemma~\ref{Eh:ADEh:ADI}, we
obtain that
\begin{equation}
  \label{Err:H1}
  \begin{split}
    |\widetilde{u}-\widetilde{u}_h|_{1,\Omega} & \leq |\widetilde{u}_I-\widetilde{u}|_{1,\Omega}+|\delta_h- E_h \delta_h|_{1,\Omega}+|E_h \delta_h|_{1,\Omega} \\
    & \leq C h^{1+s} |\widetilde{u}|_{2+s,\Omega}+ |E_h \delta_h|_{1,\Omega}.
  \end{split}
\end{equation}
To derive an upper bound for the term $|E_h \delta_h|_{1,\Omega}$, we
consider the following auxiliary problems: Find $\phi \in V$ such that
\begin{equation}
  \label{aux:h1}
  \widehat{a}(\phi,\omega)=\int_{\Omega} \nabla E_h \delta_h \cdot \nabla \omega \qquad \forall \omega \in V.
\end{equation}
Since the problem \eqref{aux:h1} is well-posed, by employing the
regularity result, we have
\begin{equation}
\label{reg:non_con:H1}
  \|\phi\|_{2+s,\Omega} \leq C \| \nabla E_h \delta_h \|_{0,\Omega}.
\end{equation}
By considering $\omega=\nabla E_h \delta_h \in V_h^C \subset
H^2(\Omega)$ in \eqref{aux:h1}, we obtain that
\begin{equation}
\label{Bound:EH}
\|\nabla E_h \delta_h \|_{0,\Omega}^2=\widehat{a}(\phi, E_h \delta_h )=\widehat{a}(\phi, E_h \delta_h-\delta_h)+\widehat{a}(\phi, \delta_h):=T_1+T_2.
\end{equation}
By applying Lemma~\ref{Eh:ADEh:ADI}, we bound the term $T_1$ as
\begin{equation}
\label{Bound:T1}
\begin{split}
|T_1|=|\widehat{a}(\phi, E_h \delta_h-\delta_h)|&\leq |a(\phi, E_h \delta_h-\delta_h)|+|b(\phi, E_h \delta_h-\delta_h)| \\
& \leq C h^s |\phi|_{2+s,\Omega} |\delta_h|_{2,h}+C |\phi|_{1,\Omega}|E_h \delta_h-\delta_h|_{1,\Omega} \\
& \leq Ch^{2s} |\phi|_{2+s,\Omega} |\widetilde{u}|_{2+s,\Omega}+C h^{1+s} |\phi|_{1,\Omega} |\widetilde{u}|_{2+s,\Omega} \\
& \leq C h^{2s} |\phi|_{2+s,\Omega} |\widetilde{u}|_{2+s,\Omega} \leq  C h^{2s} \|\nabla E_h \delta_h \|_{0,\Omega} |\widetilde{u}|_{2+s,\Omega}.
\end{split}
\end{equation} 
To derive an upper bound of the term $T_2$, we first note that
\begin{equation}
\label{Bound:T2}
T_2=\widehat{a}(\phi, \delta_h)=\widehat{a}(\widetilde{u}-\widetilde{u}_I, \phi)+\widehat{a}(\widetilde{u}_h-\widetilde{u}, \phi-\phi_I)+\widehat{a}(\widetilde{u}_h-\widetilde{u},\phi_I).
\end{equation}
Then, an application of Lemma~\ref{Eh:ADEh:ADI} yields the following upper bound for the first term:
\begin{equation}
\label{bound:A1}
\begin{split}
\widehat{a}(\widetilde{u}-\widetilde{u}_I, \phi) &=a (\widetilde{u}-\widetilde{u}_I, \phi)+ b(\widetilde{u}-\widetilde{u}_I, \phi) \\
& \leq C h^{2s} |\widetilde{u}|_{2+s,\Omega} \|\phi\|_{2+s,\Omega}+C |\widetilde{u}-\widetilde{u}_I|_{1,\Omega} |\phi|_{1,\Omega} \\
& \leq C h^{2s} |\widetilde{u}|_{2+s,\Omega} \|\phi\|_{2+s,\Omega} \leq C h^{2s}~\|\nabla E_h \delta_h \|_{0,\Omega} |\widetilde{u}|_{2+s,\Omega}.
\end{split}
\end{equation}
Further, by employing Theorem~\ref{Noncon_anal}, and approximation property of interpolant $u_I$, we derive that
\begin{equation}
\label{bound:A2}
\widehat{a}(\widetilde{u}_h-\widetilde{u}, \phi-\phi_I) \leq C h^{2s} (|\widetilde{u}|_{2+s,\Omega} +\|f\|_{1,\Omega})|\phi|_{2+s,\Omega} \leq C h^{2s} (|\widetilde{u}|_{2+s,\Omega} +\|f\|_{1,\Omega}) \|\nabla E_h \delta_h \|_{0,\Omega}.
\end{equation} 
Further, we rewrite the term $\widehat{a}(\widetilde{u}_h-\widetilde{u},\phi_I)$ as follows
\begin{equation*}
\begin{split}
\widehat{a}(\widetilde{u}_h-\widetilde{u},\phi_I)&=\widehat{a}(\widetilde{u}_h,\phi_I)-\widehat{a}(\widetilde{u},\phi_I) \\
& =\widehat{a}(\widetilde{u}_h,\phi_I)-\widehat{a}_h(\widetilde{u}_h,\phi_I)+b(f,\phi_I-\phi)+\widehat{a}(\widetilde{u},\phi-\phi_I).
\end{split}
\end{equation*}
Further, upon using properties of interpolation operator $\phi_I$, trace inequality, and regularity estimate~\eqref{reg:non_con:H1}, we derive that
\begin{equation}
\label{bound:F}
b(f, \phi_I-\phi) \leq C h^{2s} \|f\|_{0, \partial \Omega}~\|\phi\|_{2+s,\Omega} \leq C h^{2s} \|f\|_{1, \Omega} \|\nabla E_h \delta_h \|_{0,\Omega}.
\end{equation}
By using Lemma~\ref{Eh:ADEh:ADI}, we bound the term as follows
\begin{equation}
\label{bound:F2}
\begin{split}
\widehat{a}(\widetilde{u},\phi-\phi_I)&=a(\widetilde{u},\phi-\phi_I)+b(\widetilde{u},\phi-\phi_I) \\
& \leq C h^{2s} \|\widetilde{u}\|_{2,\Omega} \|\phi\|_{2+s,\Omega} \leq C h^{2s} \|\widetilde{u}\|_{2+s,\Omega} \|\nabla E_h \delta_h \|_{0,\Omega}.
\end{split}
\end{equation}
By using the polynomial consistency of $\widehat{a}_h(\cdot,\cdot)$, the polynomial approximation property of $u_{\pi}$, and Theorem~\ref{Noncon_anal}, we rewrite the difference as follows
\begin{equation}
\label{bound:diff}
\begin{split}
\widehat{a}(\widetilde{u}_h,\phi_I)-\widehat{a}_h(\widetilde{u}_h,\phi_I)&=\widehat{a}(\widetilde{u}_h-\widetilde{u}_{\pi},\phi_I-\phi_{\pi})-\widehat{a}_h(\widetilde{u}_h-\widetilde{u}_{\pi},\phi_I-\phi_{\pi}) \\
& \leq C h^{2s} (|\widetilde{u}|_{2+s,\Omega} +\|f\|_{1,\Omega}) \|\phi\|_{2+s,\Omega}  \\
&\leq C h^{2s} (|\widetilde{u}|_{2+s,\Omega} +\|f\|_{1,\Omega}) \|\nabla E_h \delta_h \|_{0,\Omega}.
\end{split}
\end{equation}
By inserting \eqref{Bound:T1}-\eqref{bound:diff} into \eqref{Bound:EH}, we 
estimate
\begin{equation}
\|\nabla E_h \delta_h \|_{0,\Omega} \leq C h^{2s}(\|\widetilde{u}\|_{2+s,\Omega}+\|f\|_{1,\Omega}).
\end{equation}
Upon inserting the estimate of $\|\nabla E_h \delta_h \|_{0,\Omega}$  into \eqref{Err:H1}, we derive the result
\begin{equation}
|\widetilde{u}-\widetilde{u}_h|_{1,\Omega} \leq C h^{2s} (\|\widetilde{u}\|_{2+s,\Omega}+ \|f\|_{1,\Omega}).
\end{equation}

\end{proof}
%
$\bullet$ \textit{\bf Convergence of the source problem in $L^2$ norm}.
In addition to the previous estimate, we would like to derive the error estimate of  Problem~\ref{source:discrete} in the $L^2$-norm. 
Following \cite[Lemma~4.3]{adak2023conforming} and  \cite[Theorem~5.1]{chinosi2016virtual}, we estimate $\|\widetilde{u}-\widetilde{u}_h\|_{0,\Omega}$. Since most of the arguments for proving the  $L^2$ estimate are used in Theorem~\ref{H1:Nonconf:Th}, we briefly state the proof.
\begin{theorem}
Let the mesh regularity stated in Assumption~\ref{Mesh:regularity} be satisfied. Let $\widetilde{u}$ and $\widetilde{u}_h$ be the unique solutions of Problem \ref{pr:eigen1} and \ref{Disc_VEP_BEP}, respectively. Then, there exists a positive constant $C$, independent of $h$ such that
\begin{equation}
\|\widetilde{u}-\widetilde{u}_h\|_{0,\Omega} \leq C h^{2s} (\|\widetilde{u}\|_{2+s,\Omega}+ \|f\|_{1,\Omega}).
\end{equation}
\end{theorem}
\begin{proof}
    By adding and subtracting the interpolated field $\widetilde{u}_{I}$,
we split the error as follows
\begin{equation}
  \widetilde{u}_h-\widetilde{u}=\widetilde{u}_h-\widetilde{u}_{I}+\widetilde{u}_I-\widetilde{u}=\widetilde{u}_I-\widetilde{u}+\delta_h,
\end{equation} 
$\delta_h=\widetilde{u}_h-\widetilde{u}_{I} $.
By employing the triangle inequality and Lemma~\ref{Int_poly}, we
obtain that
\begin{equation}
  \label{Err:L2}
  \begin{split}
    \|\widetilde{u}-\widetilde{u}_h\|_{0,\Omega} & \leq \|\widetilde{u}_I-\widetilde{u}\|_{0,\Omega}+\|\delta_h\|_{0,\Omega} \\
    & \leq C h^{1+s} \|\widetilde{u}\|_{2+s,\Omega}+ \|\delta_h\|_{0,\Omega}.
  \end{split}
\end{equation}
To derive the estimate for $\|\delta_h\|_{0,\Omega}$, we define the following auxiliary problems. Find $\phi \in V$ such that
\begin{equation}
  \label{aux:L2}
  \widehat{a}(\phi,\omega)=\int_{\Omega} \delta_h  \omega \qquad \forall \omega \in V.
\end{equation}
Since the problem \eqref{aux:L2} is well-posed, by employing the
regularity result, we have
\begin{equation}
\label{reg:non_con:L1}
  \|\phi\|_{2+s,\Omega} \leq C \| \delta_h \|_{0,\Omega}.
\end{equation}
By considering $\omega=\delta_h$ in \eqref{aux:L2}, we obtain that
\begin{equation}
\label{Bound:EH:L2}
\| \delta_h \|_{0,\Omega}^2=\widehat{a}(\phi, \delta_h ).
\end{equation}
Following analogous arguments as in Theorem~\ref{H1:Nonconf:Th}, we derive the estimate
\begin{equation}
\label{estimate:del_h}
    \| \delta_h \|_{0,\Omega}\leq C h^{2s}~(\|\widetilde{u}\|_{2+s,\Omega}+ \|f\|_{1,\Omega}). 
\end{equation}
Inserting \eqref{estimate:del_h} into \eqref{Err:L2}, we derive the bound of $\widetilde{u}-\widetilde{u}_h$ in $L^2$-norm.
\end{proof}
%
\section{Spectral Convergence}
\label{SPC:CONVERGENCE}
This section will prove the convergence and establish error estimates
for the proposed VEM discretizations of the Steklov eigenvalue
problem.
To this end, we will prove that $T_h$ provides a correct spectral
approximation of $T$ using the classical theory for compact operators,
see, e.g.,~\cite{BO}.
Next, an immediate consequence of Theorem~\ref{Con_Anal_H2}, and
Theorem~\ref{H1:Nonconf:Th} is that both the conforming and
non-conforming methods approximate isolated parts of $sp(T)$ by
isolated parts of $sp(T_h)$.
Consequently, if $\mu$ is a nonzero eigenvalue of $T$ with algebraic
multiplicity $m$, there exist $m$ discrete eigenvalues
$\mu_h^{(1)},\ldots,\mu_h^{(m)}$ of $T_h$ (repeated according to their
respective multiplicities) that will converge to $\mu$ as $h$ goes to
zero.
Let $\mathcal{E}$ and $\mathcal{E}_h$ denote the eigenspaces
associated with the eigenvalue $\mu$ and the span of the eigenspaces
associated with $\mu_h^{(1)},\ldots,\mu_h^{(m)}$, respectively.
We recall that the \textit{gap} $\hat{\delta}$ between two closed
subspaces $\mathcal{X}$ and $\mathcal{Y}$ of $H^1(\Omega)$ is
\begin{align*}
  \hat{\delta}(\mathcal{X},\mathcal{Y})
  :=\max\left\{\delta(\mathcal{X},\mathcal{Y}),\delta(\mathcal{Y},\mathcal{X})\right\},
\end{align*}
where
\begin{align*}
  \delta(\mathcal{X},\mathcal{Y})
  :=\sup_{\mathbf{x}\in\mathcal{X}:\ 
\left\|x\right\|_{H^1( \Omega)}=1}\delta(x,\mathcal{Y}),
  \quad\text{with }\delta(x,\mathcal{Y}):=
  \inf_{y\in\mathcal{Y}}\|x-y\|_{H^1(\Omega)}.
\end{align*}
We also define
\begin{align*}
  \gamma_h:=\sup\limits_{v\in \mathcal{E}:\Vert
    v\Vert_{H^1( \Omega)}=1} \Vert(T-T_h)v\Vert_{H^1(
    \Omega)}
\end{align*}
and 
\begin{align*}
  \beta_h:=\sup\limits_{v\in \mathcal{E}:\Vert
    v\Vert_{H^2( \Omega)}=1} \Vert(T-T_h)v\Vert_{H^2(
    \Omega)}.
\end{align*}
The following error estimates for the approximation of eigenvalues and
eigenfunctions hold true.
The result can be obtained from the application of \cite[Theorems~7.1
  and~7.3]{BO}.

\begin{theorem}
  \label{gap}
  There exists a strictly positive constant $C$ such that
  \begin{align}
    \hat{\delta}(\mathcal{E},\mathcal{E}_h) 
    & \leq C \gamma_h,\nonumber\\
    \left|\mu-\mu_h^{(j)}\right|
    & \le C \beta_h^2 \qquad \forall j=1,\ldots,m.\nonumber
  \end{align}
\end{theorem}
Moreover, employing the additional regularity of the eigenfunctions,
we immediately obtain the following bound. 
\begin{theorem}\label{gapr}
  There exist $s> 1/2$ and $C>0$ independent of $h$ such that
  \begin{align}
    &\Vert(T-T_h)f\Vert_{H^1( \Omega)} =\Vert \widetilde{u}-\widetilde{u}_h \Vert_{H^1( \Omega)} 
    \le C h^{2s} \Big (|\widetilde{u}|_{2+s,\Omega}+||f||_{1,\Omega} \Big),\label{bou_gamma_h_1} \\
      &\Vert(T-T_h)f\Vert_{H^2( \Omega)} =\Vert \widetilde{u}-\widetilde{u}_h \Vert_{H^2( \Omega)} 
    \le C h^{s} \Big (|\widetilde{u}|_{2+s,\Omega}+||f||_{1,\Omega} \Big),\label{bou_gamma_h_2}
  \end{align}
  and as a consequence,
  \begin{align}
    &\gamma_h \leq C h^{2s} \text{ and } \beta_h \leq C h^{s} . \label{bound1r}
  \end{align}
\end{theorem}
\begin{proof}
  The inequalities~\eqref{bou_gamma_h_1} and~\eqref{bou_gamma_h_2}  are obtained repeating the proof
  of Theorem~\ref{Con_Anal_H2}, and Theorem~\ref{H1:Nonconf:Th}.
  Estimates \eqref{bound1r} follow from the definition of $\gamma_h$ and $\beta_h$
  and~\eqref{bou_gamma_h_1} and~~\eqref{bou_gamma_h_2}.
\end{proof}
%

\section{Numerical Results}
\label{Numer:Exp}
This section reports some numerical experiments that support the
theoretical estimates.
We show the results for the lowest-order $H^2$-conforming and
$C^0$ non-conforming VEM space, i.e., for $k=2$.
To assess these methods' behavior, we consider four different mesh
families: triangle, square, non-convex, and Voronoi meshes.
The model problem \eqref{eq:model1:A}-\eqref{eq:model1:C} deals with
two distinct types of boundary conditions: homogeneous Dirichlet
and Neumann boundary conditions.
As usual, we impose strongly the homogeneous Dirichlet boundary condition on the discrete space in which the numerical approximation is searched.
Generally, by employing the standard basis of the virtual element
space, Problem~\ref{Disc_VEP_BEP} leads to the generalized matrix
eigenvalue problem:
\begin{equation}
  \mathbf{A} \mathbf{U}=(\lambda_h +1)\mathbf{B} \mathbf{U},
\end{equation}
where $\mathbf{U}:\big[\text{dof}_j(u_h)\big]_{j=1}^{N^{\text{dof}}}$
is the column vector collecting the degrees of freedom of the virtual
element eigenvector.
Matrix $\mathbf{A}$ is symmetric and positive definite, whereas
$\mathbf{B}$ is symmetric and semi-positive definite.
We use the MATLAB command \texttt{eigs} to solve the
equivalent problem
\begin{equation}
  \mathbf{B}\mathbf{U}=\frac{1}{\lambda_h+1} \mathbf{A} \mathbf{U}.
\end{equation}
Let $N$ denote the number of elements in the mesh. To analyze the convergence rate, we use the relation $h \approx N^{-1/2}$. Some meshes were generated with Gmsh mesh generator \cite{geuzaine2009gmsh}, and Voronoi meshes were created using Polymesher \cite{polymeshertalischi2012}.

To the best of our knowledge, the proposed model problem associated
with these boundary conditions is new in the literature on the Steklov
eigenvalue problems.
Since the analytical eigenvalues are not known, we compare the
solutions computed with the $H^2$-conforming and the
$C^0$ non-conforming VEM.
In both cases, we obtain the same numerical approximations, which we
display in
Tables~\ref{table:square2D-VEM-Th0}-\ref{table:square2D-VEM-Th3} 
of Test~\ref{subsec:square2D}.
For the conforming method, we have come across the same behavior of
the spectrum, and to avoid repetition, we do not show the
corresponding results for the eigenfunctions and eigenvalues in
Test~\ref{subsec:lshape2D} - Test~\ref{subsec:circle-holes}. For completeness, we plot the gradients of the eigenfunctions, which are reconstructed via polynomial projections from the virtual element degrees of freedom to assess the enforcement of the zero normal derivative condition.

\subsection{Square domain with homogeneous boundary condition}\label{subsec:square2D}
This test assesses the accuracy of the proposed conforming and nonconforming schemes on a convex polygonal domain. We consider $\Omega := (0,1)^2$ and compute the first ten positive eigenvalues for several mesh refinement levels on each mesh family $\mathcal{T}_h^i$, $i=0,1,2,3$, shown in Figure~\ref{fig:mesh_PolyGon}.
\begin{figure}[t!]\centering
  \begin{minipage}{0.23\linewidth}\centering
    \includegraphics[scale=0.09, trim=17cm 2cm 17cm 2cm,clip]{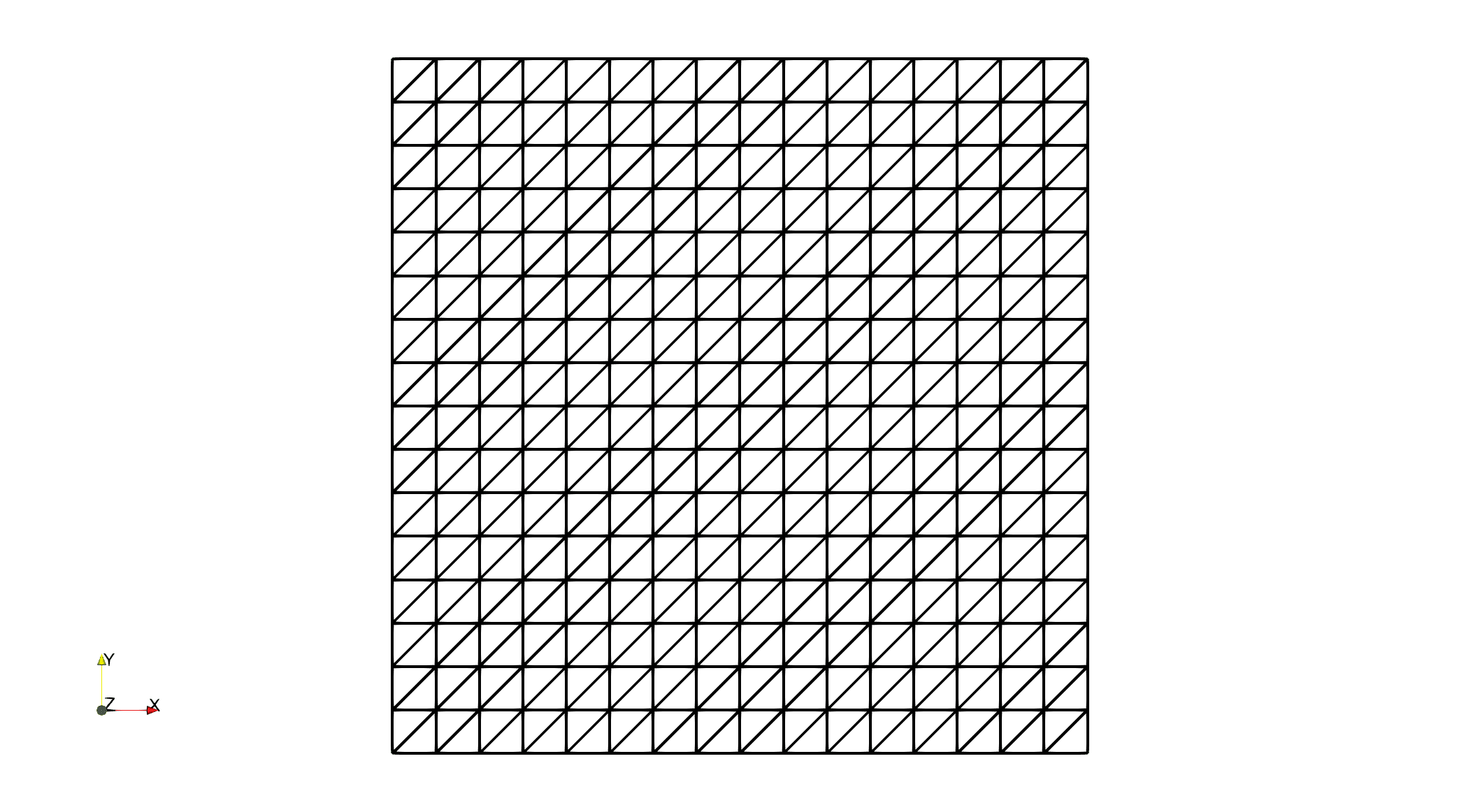}\\
	            {$\CT^0_h$}
  \end{minipage}
  \begin{minipage}{0.23\linewidth}\centering
    \includegraphics[scale=0.09, trim=17cm 2cm 17cm 2cm,clip]{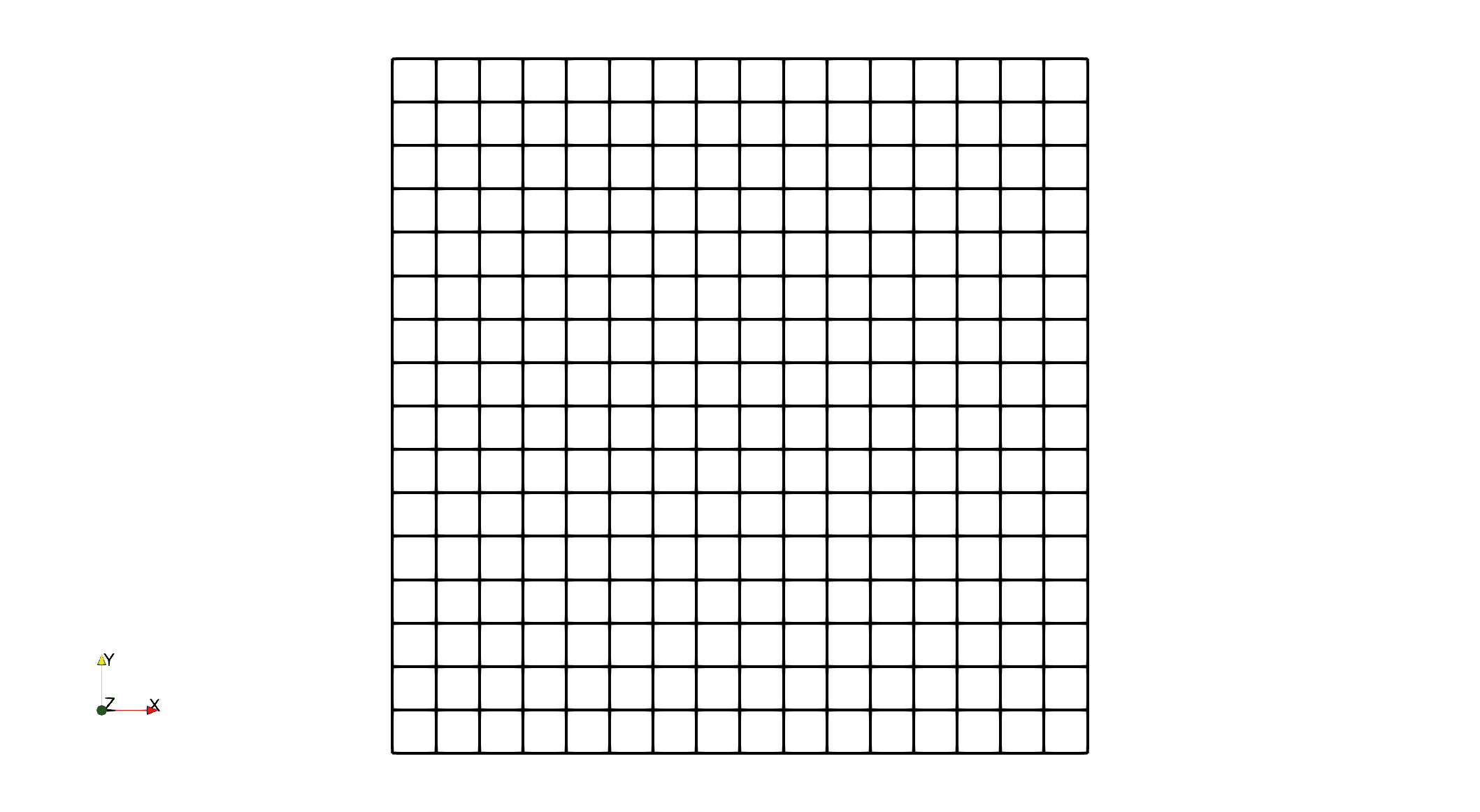}\\
                    {$\CT^1_h$}
  \end{minipage}
  \begin{minipage}{0.23\linewidth}\centering
    \includegraphics[scale=0.09, trim=17cm 2cm 17cm 2cm,clip]{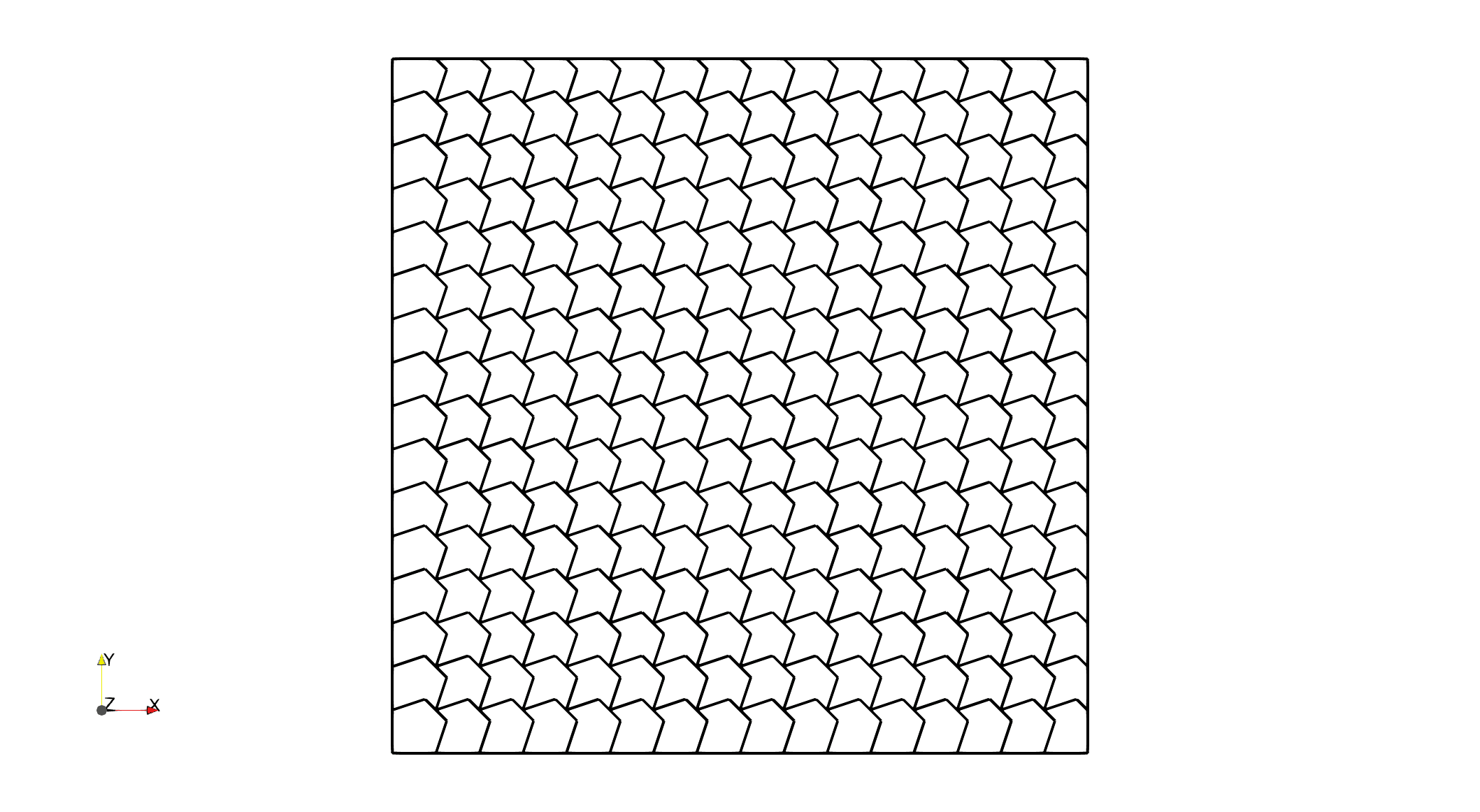}\\
	            {$\CT^2_h$}
  \end{minipage}
  \begin{minipage}{0.23\linewidth}\centering
    \includegraphics[scale=0.09, trim=17cm 2cm 17cm 2cm,clip]{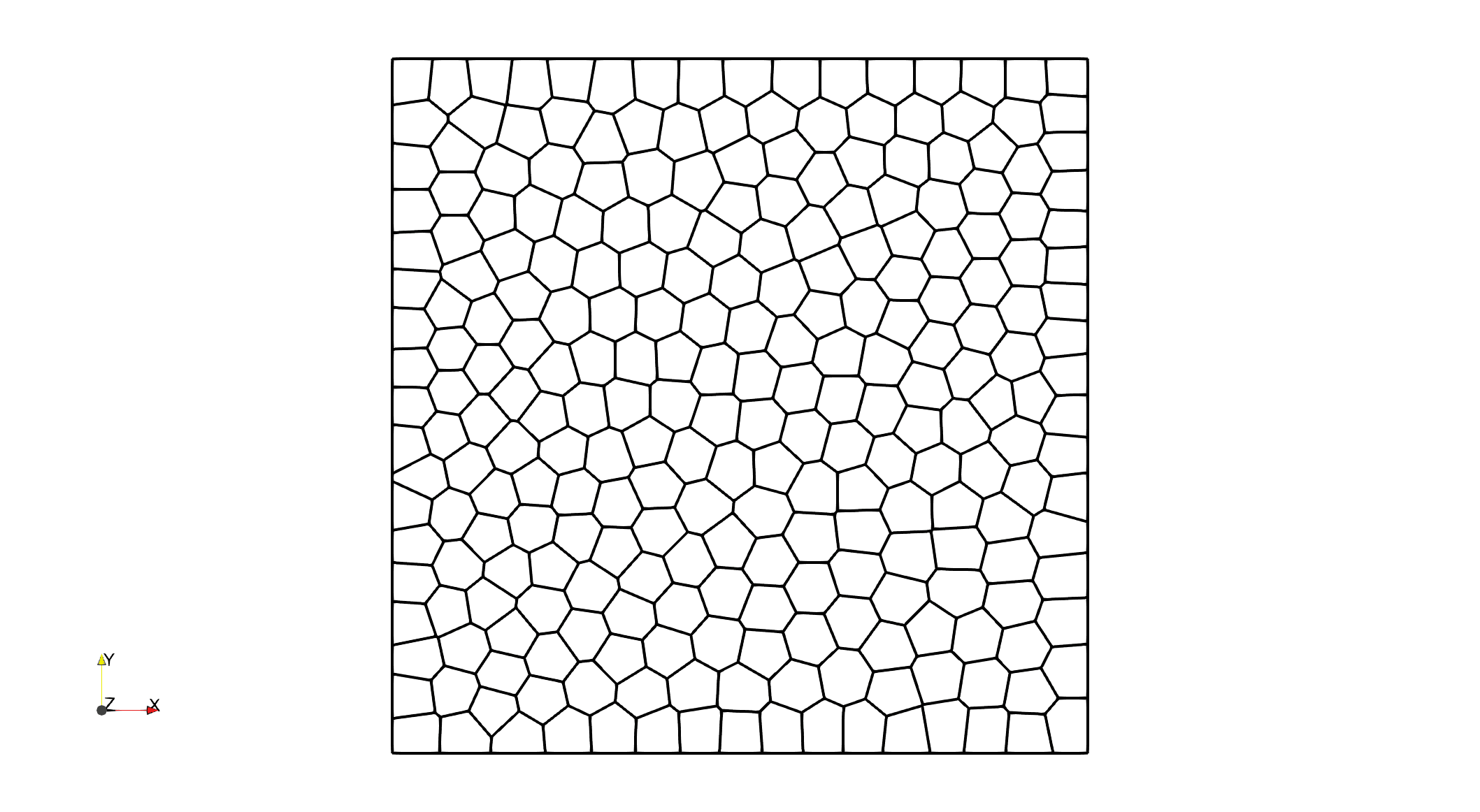}\\
	            {$\CT^3_h$}
  \end{minipage}
  \caption{Test~\ref{subsec:square2D}. Examples of the mesh families used in the	eigenvalue calculations.}
  \label{fig:mesh_PolyGon} 
\end{figure}
Tables~\ref{table:square2D-VEM-Th0}--\ref{table:square2D-VEM-Th3} report the convergence histories obtained with both conforming and nonconforming schemes (denoted by $C^1$ cVEM and $C^0$ ncVEM, respectively) on the four mesh families. Since the square is a convex polygonal domain, the numerical results exhibit the expected second-order convergence behavior, namely $\mathcal{O}(h^2)$, or equivalently $\mathcal{O}(N^{-1})$ in two dimensions, in agreement with Theorem~\ref{gapr}. The extrapolated eigenvalues obtained from the different mesh families are very close, indicating that the observed convergence behavior is essentially independent of the particular polygonal mesh sequence. We also observe the presence of repeated, or nearly repeated, eigenvalues. This is consistent with the symmetry of the square domain and should be interpreted in terms of convergence of eigenspaces rather than solely pointwise convergence of individual eigenvalue branches. For illustration, Figure~\ref{fig:square2D-uh-graduh} shows selected computed eigenfunctions and the corresponding gradient fields obtained with the nonconforming scheme on the Voronoi mesh. The conforming results are omitted, as they exhibit qualitatively similar behavior. The plots are consistent with the homogeneous boundary condition $\partial_{\mathbf n}u=0$; in particular, the gradient field is tangential to the boundary, except at corner points where the normal and tangential directions are not uniquely defined.%
\begin{table}[t!]
  \centering 
  \footnotesize
  \begin{center}
    \caption{Example \ref{subsec:square2D}. Convergence history of the ten lowest computed eigenvalues on the square domain and $\CT_h^0$. }
    \label{table:square2D-VEM-Th0}
    \begin{tabular}{|c| c c c c |c| c|}
      \hline
      \hline
      &$N=512$             &  $N=2048$         &   $N=8192$         & $N=32768$ & Order & $\lambda_{\text{extr}}$ \\ 
      \hline 
      \hline
      \multirow{9}{0.11\linewidth}{$C^0$ ncVEM}
      &  15.8025  &    15.9265  &    15.9568  &    15.9640  & 2.04 &    15.9664  \\
      &  15.8733  &    15.9433  &    15.9608  &    15.9651  & 2.01 &    15.9665  \\
      & 46.8996   &    47.3745  &    47.4893  &    47.5168  & 2.05 &    47.5257  \\
      & 238.0188  &   245.2412  &   246.9778  &   247.3980  & 2.05 &   247.5341  \\
      & 312.3578  &   324.7498  &   327.7016  &   328.3955  & 2.07 &   328.6204  \\
      & 315.9303  &   325.5833  &   327.8902  &   328.4491  & 2.06 &   328.6236  \\
      & 350.0269  &   363.7253  &   367.0108  &   367.7827  & 2.06 &   368.0405  \\
      &1007.5503  &  1089.4047  &  1109.2059  &  1113.9331  & 2.05 &  1115.4730  \\
      &1105.4769  &  1211.6456  &  1237.5230  &  1243.5748  & 2.05 &  1245.5808  \\
      &1125.5417  &  1216.7030  &  1238.6360  &  1243.9016  & 2.06 &  1245.5283  \\
      \hline
      \multirow{9}{0.11\linewidth}{$C^1$ cVEM}
      &  15.9584  &    15.9645  &    15.9660  &    15.9664  & 2.00 &    15.9665  \\
      &  15.9971  &    15.9741  &    15.9684  &    15.9670  & 2.00 &    15.9665  \\
      &  47.5525  &    47.5327  &    47.5278  &    47.5266  & 2.03 &    47.5262  \\
      & 248.7586  &   247.8404  &   247.6134  &   247.5569  & 2.01 &   247.5380  \\
      & 328.8417  &   328.6830  &   328.6469  &   328.6383  & 2.12 &   328.6358  \\
      & 331.0476  &   329.2350  &   328.7850  &   328.6728  & 2.01 &   328.6362  \\
      & 369.6616  &   368.4465  &   368.1490  &   368.0753  & 2.03 &   368.0522  \\
      &1127.9642  &  1118.5973  &  1116.2848  &  1115.7102  & 2.02 &  1115.5275  \\
      &1253.4774  &  1247.5562  &  1246.1325  &  1245.7838  & 2.05 &  1245.6747  \\
      &1265.5007  &  1250.5615  &  1246.8835  &  1245.9715  & 2.02 &  1245.6755  \\
      \hline
      \hline             
    \end{tabular}
  \end{center}
\end{table}
\begin{table}[t!]
  \centering 
  \footnotesize
  \begin{center}
    \caption{Example \ref{subsec:square2D}. Convergence history of the
      ten lowest computed eigenvalues on the square domain and
      $\CT_h^1$. }
    \label{table:square2D-VEM-Th1}
    \begin{tabular}{|c| c c c c |c| c|}
      \hline
      \hline
      &$N=256$             &  $N=1024$         &   $N=4096$         & $N=16384$ & Order & $\lambda_{\text{extr}}$ \\ 
      \hline
      \hline
      \multirow{10}{0.11\linewidth}{$C^0$ ncVEM}
      &  15.7390  &    15.9108  &    15.9528  &    15.9631  & 2.03 &    15.9664  \\
      &  15.7390  &    15.9108  &    15.9528  &    15.9631  & 2.03 &    15.9664  \\
      &  46.5976  &    47.3081  &    47.4736  &    47.5133  & 2.10 &    47.5245  \\
      & 229.8467  &   243.3007  &   246.5188  &   247.2891  & 2.06 &   247.5369  \\
      & 303.5487  &   322.8091  &   327.2580  &   328.3020  & 2.11 &   328.6100  \\
      & 303.5487  &   322.8091  &   327.2580  &   328.3020  & 2.11 &   328.6100  \\
      & 337.0410  &   360.8788  &   366.3599  &   367.6420  & 2.12 &   368.0073  \\
      & 918.4974  &  1068.0040  &  1104.3836  &  1112.8524  & 2.05 &  1115.7204  \\
      &1010.7716  &  1189.5470  &  1232.6177  &  1242.5558  & 2.06 &  1245.9622  \\
      &1010.7716  &  1189.5470  &  1232.6177  &  1242.5558  & 2.06 &  1245.9622  \\
      \hline
      \multirow{10}{0.11\linewidth}{$C^1$ cVEM} &15.9605  &    15.9650  &    15.9661  &    15.9664  & 2.00 &    15.9665  \\
      &  15.9605  &    15.9650  &    15.9661  &    15.9664  & 2.00 &    15.9665  \\
      &  47.5209  &    47.5249  &    47.5259  &    47.5261  & 1.98 &    47.5262  \\
      & 247.4412  &   247.5137  &   247.5320  &   247.5365  & 1.99 &   247.5381  \\
      & 329.7480  &   328.9155  &   328.7056  &   328.6530  & 1.99 &   328.6352  \\
      & 329.7480  &   328.9155  &   328.7056  &   328.6530  & 1.99 &   328.6352  \\
      & 368.3098  &   368.1143  &   368.0666  &   368.0548  & 2.03 &   368.0510  \\
      &1123.6236  &  1117.6004  &  1116.0428  &  1115.6502  & 1.96 &  1115.5112  \\
      &1250.6352  &  1246.9187  &  1245.9813  &  1245.7467  & 1.99 &  1245.6673  \\
      &1250.6352  &  1246.9187  &  1245.9813  &  1245.7467  & 1.99 &  1245.6673  \\
      \hline
      \hline             
    \end{tabular}
  \end{center}
\end{table}

\begin{table}[t!]
  \centering 
  \footnotesize
  \begin{center}
    \caption{Example \ref{subsec:square2D}. Convergence history of the
      ten lowest computed eigenvalues on the square domain and
      $\CT_h^2$. }
    \label{table:square2D-VEM-Th2}
    \begin{tabular}{|c|c c c c |c| c|}
      \hline
      \hline
      &$N=256$             &  $N=1024$         &   $N=4096$         & $N=16384$ & Order & $\lambda_{\text{extr}}$ \\ 
      \hline
      \hline
      \multirow{10}{0.11\linewidth}{$C^0$ ncVEM}
      &  15.7626  &    15.9153  &    15.9537  &    15.9633  & 1.96 &    15.9667  \\
      &  15.8270  &    15.9327  &    15.9582  &    15.9644  & 2.02 &    15.9664  \\
      &  46.6744  &    47.3151  &    47.4737  &    47.5131  & 1.98 &    47.5266  \\
      & 234.5368  &   244.3737  &   246.7649  &   247.3474  & 2.01 &   247.5379  \\
      & 307.3526  &   323.4127  &   327.3474  &   328.3160  & 1.99 &   328.6522  \\
      & 310.0740  &   324.1780  &   327.5542  &   328.3697  & 2.03 &   328.6293  \\
      & 343.6970  &   362.0684  &   366.5804  &   367.6868  & 1.99 &   368.0702  \\
      & 963.7848  &  1078.7960  &  1106.6477  &  1113.3487  & 2.01 &  1115.6610  \\
      &1055.2214  &  1199.1905  &  1234.3319  &  1242.8766  & 2.00 &  1245.8098  \\
      &1076.3021  &  1203.6800  &  1235.5504  &  1243.1994  & 1.97 &  1246.0956  \\
      \hline
      \multirow{10}{0.11\linewidth}{$C^1$ cVEM} & 
      15.2877  &    15.8238  &    15.9374  &    15.9606  & 2.15 &    15.9681  \\
      &  15.2882  &    15.8239  &    15.9374  &    15.9606  & 2.15 &    15.9681  \\
      &  46.1238  &    47.2275  &    47.4649  &    47.5137  & 2.13 &    47.5301  \\
      & 232.5623  &   244.2909  &   246.8725  &   247.4034  & 2.10 &   247.5954  \\
      & 313.6371  &   325.2247  &   327.9286  &   328.4928  & 2.03 &   328.7196  \\
      & 314.0362  &   325.3969  &   327.9478  &   328.4940  & 2.07 &   328.6882  \\
      & 354.6668  &   364.9039  &   367.3815  &   367.9141  & 1.98 &   368.1399  \\
      &1049.1072  &  1100.4036  &  1112.3625  &  1114.8565  & 2.03 &  1115.8649  \\
      &1187.1871  &  1231.5915  &  1242.6972  &  1245.0393  & 1.95 &  1246.1043  \\
      &1189.5286  &  1232.7252  &  1242.7613  &  1245.0511  & 2.02 &  1245.8473  \\
      \hline
      \hline             
    \end{tabular}
  \end{center}
\end{table}

\begin{table}[t!]
  \centering 
  \footnotesize
  \begin{center}
    \caption{Example \ref{subsec:square2D}. Convergence history of the
      ten lowest computed eigenvalues on the square domain and
      $\CT_h^3$. }
    \label{table:square2D-VEM-Th3}
    \begin{tabular}{|c|c c c c |c| c|}
      \hline
      \hline
      &$N=256$             &  $N=1024$         &   $N=4096$         & $N=16384$ & Order & $\lambda_{\text{extr}}$ \\ 
      \hline
      \hline
      \multirow{10}{0.11\linewidth}{$C^0$ ncVEM}
      &15.8173  &    15.9299  &    15.9576  &    15.9643  & 2.03 &    15.9665  \\
   &15.8203  &    15.9311  &    15.9579  &    15.9643  & 2.05 &    15.9664  \\
   &46.6316  &    47.3105  &    47.4741  &    47.5131  & 2.05 &    47.5261  \\
  &234.8461  &   244.6719  &   246.8824  &   247.3790  & 2.15 &   247.5261  \\
  &306.8576  &   323.8409  &   327.5166  &   328.3608  & 2.20 &   328.5652  \\
  &308.2492  &   323.8819  &   327.5294  &   328.3641  & 2.10 &   328.6320  \\
  &341.7519  &   362.0988  &   366.6280  &   367.7045  & 2.15 &   367.9905  \\
  &954.2283  &  1078.5636  &  1107.3549  &  1113.5837  & 2.12 &  1115.7190  \\
 &1043.5687  &  1199.9067  &  1235.3608  &  1243.2162  & 2.15 &  1245.5253  \\
 &1049.2813  &  1200.4177  &  1235.5341  &  1243.2530  & 2.12 &  1245.7372  \\
      \hline
      \multirow{10}{0.11\linewidth}{$C^1$ cVEM}
      &  15.9208  &    15.9552  &    15.9637  &    15.9658  & 1.99 &    15.9665  \\
      &  15.8741  &    15.9443  &    15.9611  &    15.9652  & 2.03 &    15.9664  \\
      &  46.8996  &    47.3734  &    47.4885  &    47.5169  & 2.00 &    47.5265  \\
      & 238.4931  &   245.3784  &   247.0140  &   247.4092  & 2.04 &   247.5316  \\
      & 313.1013  &   324.9101  &   327.7261  &   328.4109  & 2.03 &   328.6292  \\
      & 315.2244  &   325.4938  &   327.8821  &   328.4514  & 2.06 &   328.6314  \\
      & 350.1731  &   363.7780  &   367.0135  &   367.7956  & 2.03 &   368.0529  \\
      &1004.6596  &  1089.6798  &  1109.3682  &  1114.0237  & 2.07 &  1115.4647  \\
      &1103.9321  &  1212.3897  &  1237.6819  &  1243.7159  & 2.06 &  1245.5905  \\
      &1119.5983  &  1215.8226  &  1238.6044  &  1243.9587  & 2.05 &  1245.6679  \\
      \hline
      \hline             
    \end{tabular}
  \end{center}
\end{table}

\begin{figure}[t!]
  \centering
  \begin{minipage}{0.32\linewidth}\centering
    \includegraphics[scale=0.062, trim=37cm 5cm 37cm 5cm,clip]{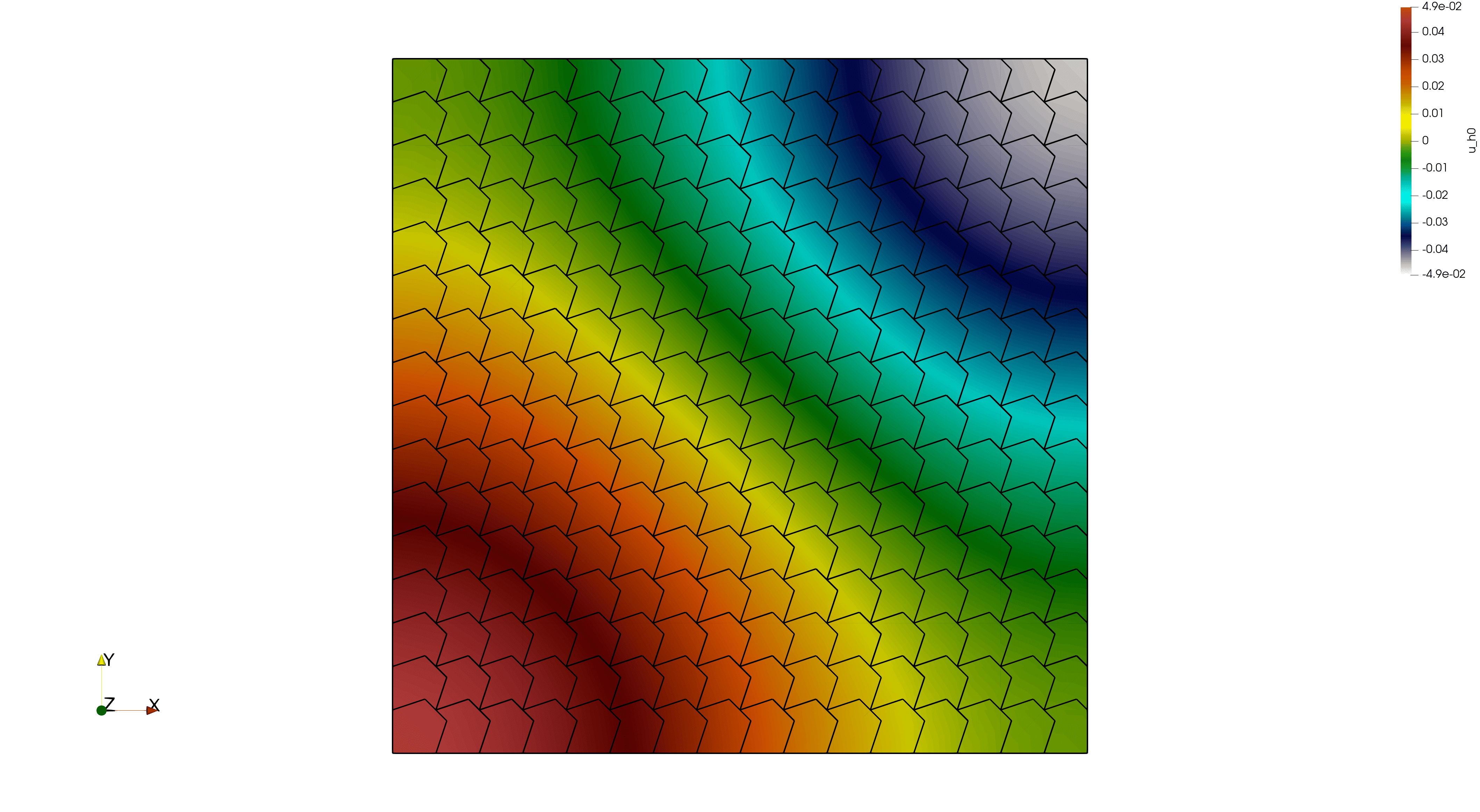}\\
		    {\footnotesize $u_{1,h}$}
  \end{minipage}
  \begin{minipage}{0.32\linewidth}\centering
    \includegraphics[scale=0.062, trim=37cm 5cm 37cm 5cm,clip]{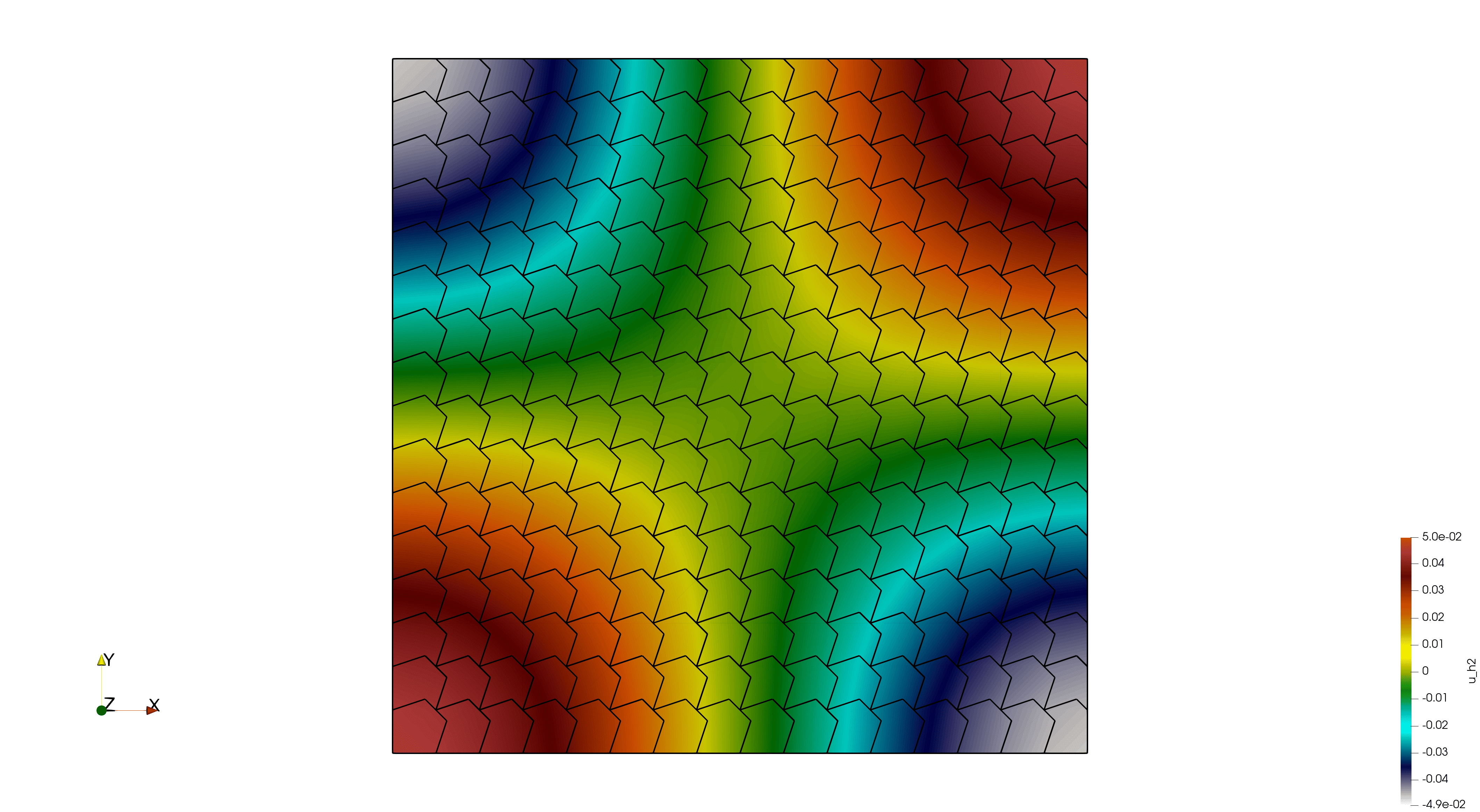}\\
		    {\footnotesize $u_{3,h}$}
  \end{minipage}
  \begin{minipage}{0.32\linewidth}\centering
    \includegraphics[scale=0.062, trim=37cm 5cm 37cm 5cm,clip]{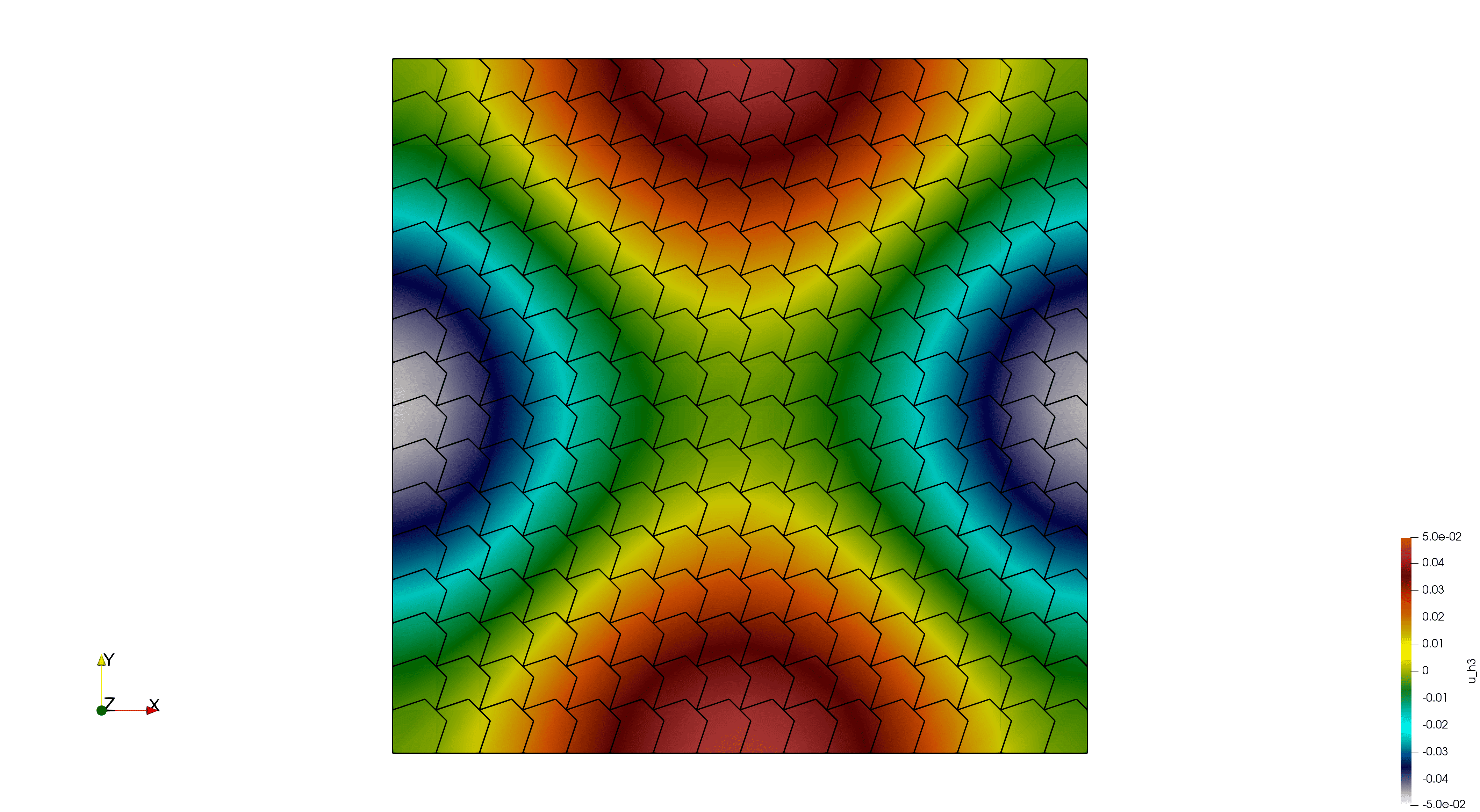}\\
	            {\footnotesize $u_{4,h}$}
  \end{minipage}
  \hspace*{0.1cm}\begin{minipage}{0.32\linewidth}\centering
    \includegraphics[scale=0.125, trim=53cm 2cm 53cm 2cm,clip]{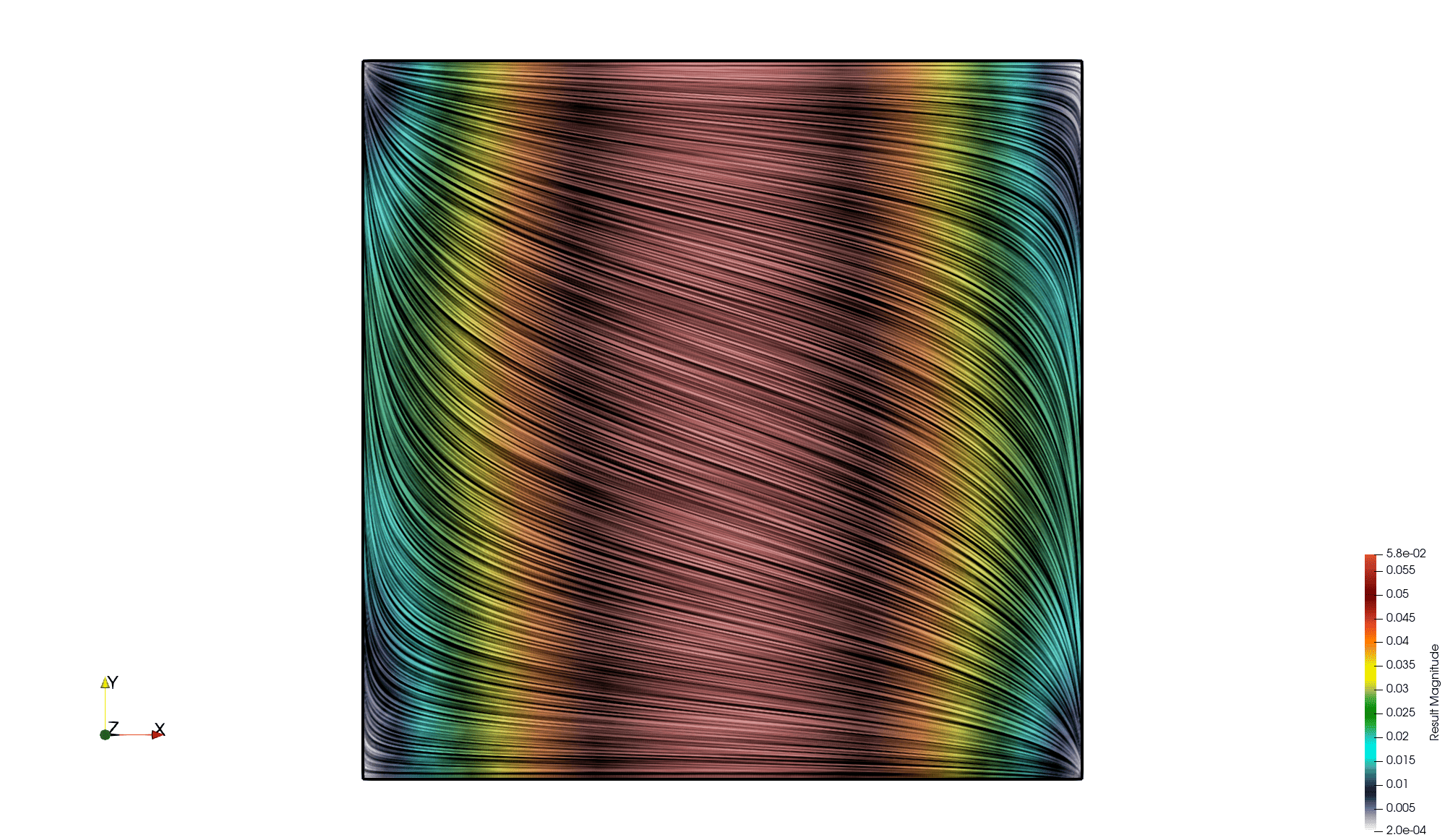}\\
		    {\footnotesize $\nabla u_{1,h}$}
  \end{minipage}
  \begin{minipage}{0.32\linewidth}\centering
    \includegraphics[scale=0.125, trim=53cm 2cm 53cm 2cm,clip]{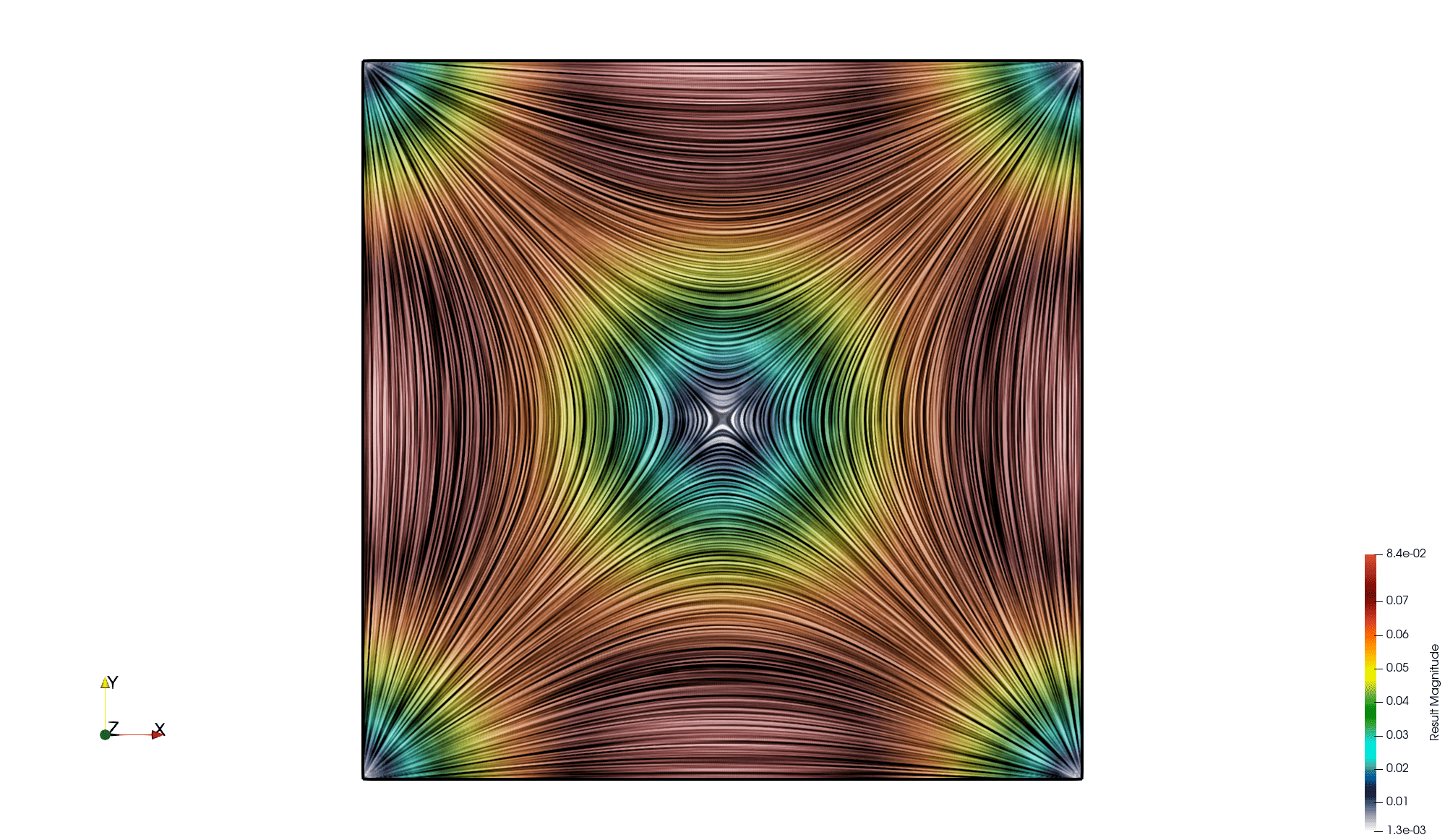}\\
		    {\footnotesize $\nabla u_{3,h}$}
	\end{minipage}
  \begin{minipage}{0.32\linewidth}\centering
    \includegraphics[scale=0.125, trim=53cm 2cm 53cm 2cm,clip]{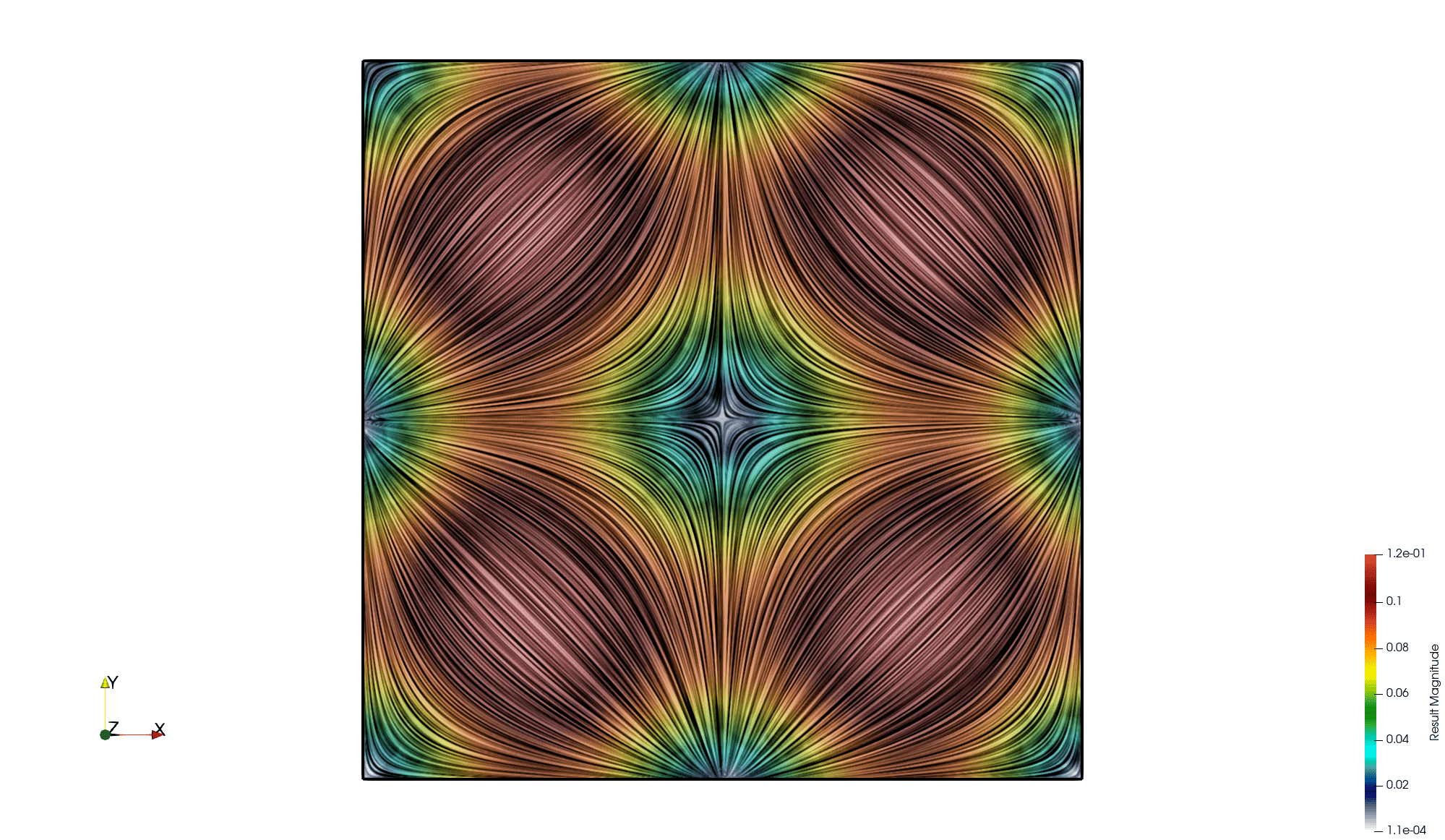}\\
		    {\footnotesize $\nabla u_{4,h}$}
  \end{minipage}
  \caption{Example \ref{subsec:square2D}. Surface plot and gradients
    surface line integral convolution of the first, third, and fourth
    lowest computed eigenvalues using non-conforming scheme on the mesh
    $\CT^3_h$ for $N=32$.}
  \label{fig:square2D-uh-graduh}
\end{figure}

\subsection{L-Shaped Domain}\label{subsec:lshape2D}
The purpose of this experiment is to assess the performance of the method on a domain exhibiting reduced regularity. We consider the L-shaped domain
\begin{align*}
  \Omega:= (-1,1)^2\backslash\left(-1,0\right)^2,
\end{align*}
which features a reentrant corner at $(0,0)$ and therefore induces singular behavior in the eigenfunctions.
We consider meshes composed of triangles and Voronoi polygons, as illustrated in Figure~\ref{fig:mesh_PolyGon-lshape}.

\begin{figure}[t!]\centering
  \begin{minipage}{0.23\linewidth}\centering
    \includegraphics[scale=0.09, trim=17cm 2cm 17cm 2cm,clip]{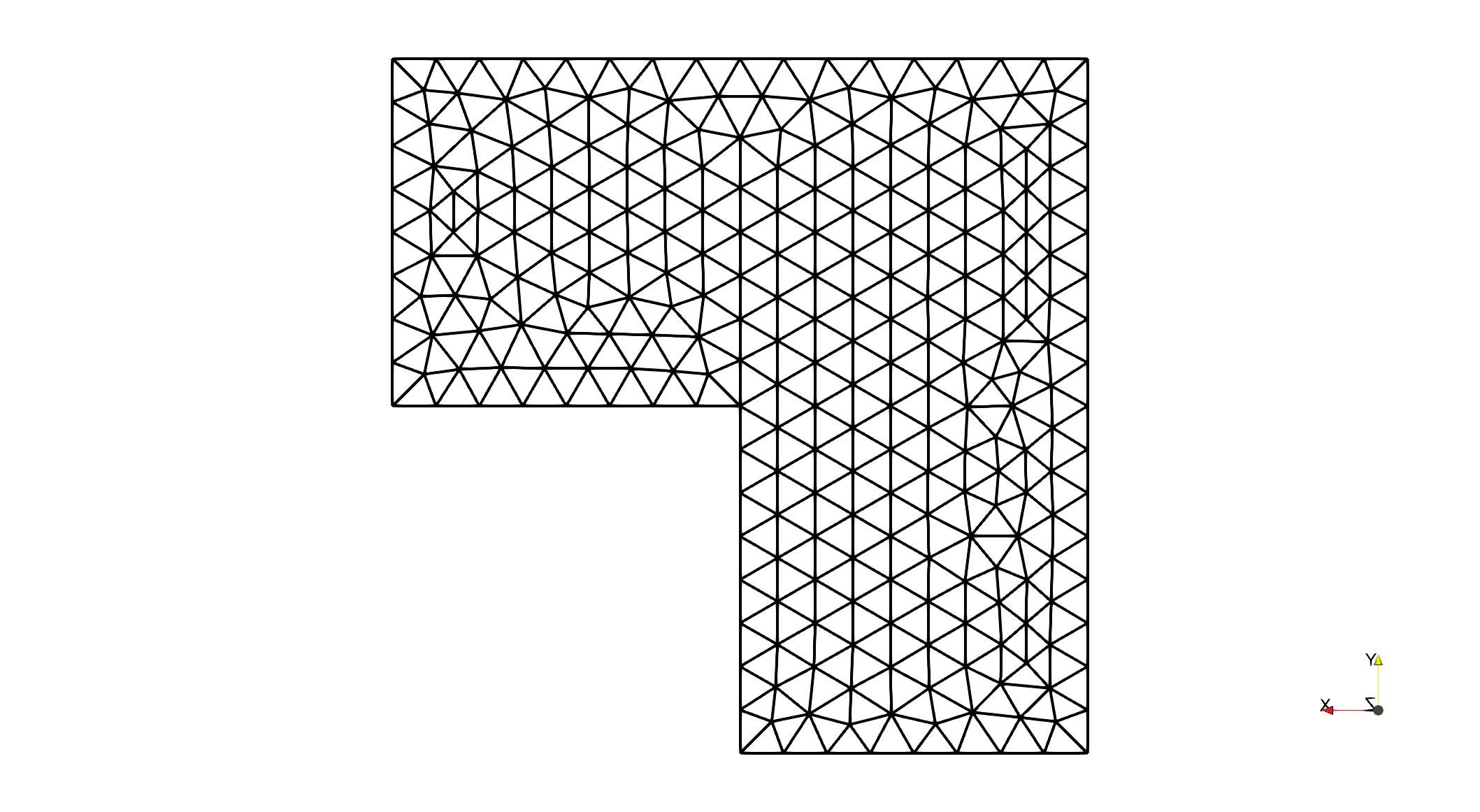}\\
		    {$\CT^4_h$}
  \end{minipage}
  \begin{minipage}{0.23\linewidth}\centering
    \includegraphics[scale=0.09, trim=17cm 2cm 17cm 2cm,clip]{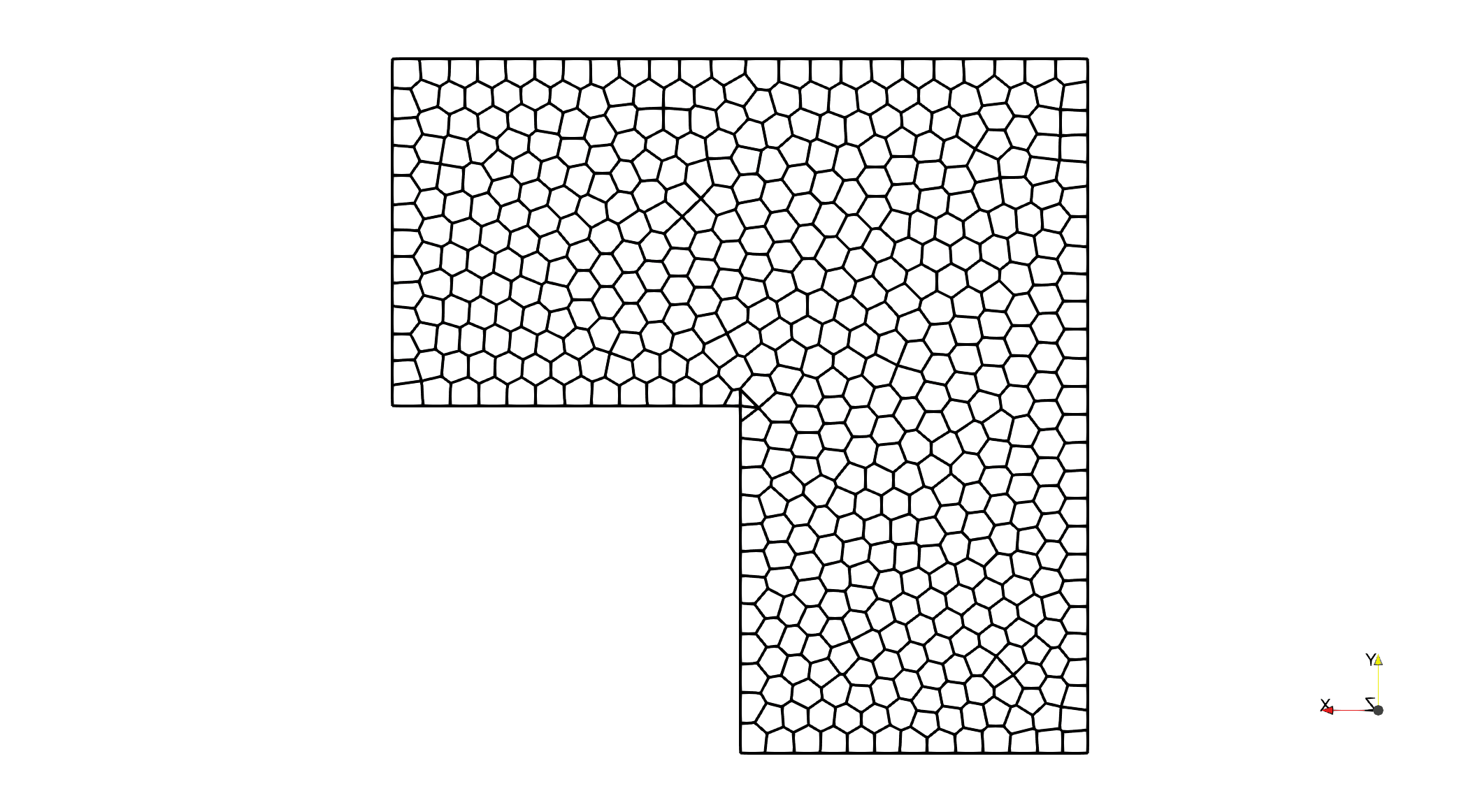}\\
		    {$\CT^5_h$}
  \end{minipage}
  \caption{Test \eqref{subsec:lshape2D}. Representative meshes used in  the eigenvalue computations.}
  \label{fig:mesh_PolyGon-lshape} 
\end{figure}
We report in Table~\ref{table:lshape2D-FEM} the convergence history of the ten lowest eigenvalues computed with the non-conforming method. It can be observed that, as expected, the rate of convergence is suboptimal for eigenvalues with associated singular eigenfunctions.  
For instance, the
first two eigenvalues are affected by the singularity, with an error behaving like $\mathcal{O}(h^{r})$, $1.2\leq r\leq 1.43$. In contrast, higher eigenvalues display convergence rates close to the optimal $\mathcal{O}(h^2)$ behavior, indicating that their associated eigenfunctions are less affected by the singularity.

Figure~\ref{fig:lshape2D-uh-graduh} shows selected eigenfunctions and their corresponding gradient fields computed with the nonconforming scheme. The first two eigenfunctions exhibit strong gradients in the vicinity of the reentrant corner, which is consistent with the expected singular behavior. The reconstructed gradient fields also confirm that the normal derivative vanishes along the boundary, as prescribed by the model.
\begin{table}[t!]
  \centering 
  \footnotesize
  \begin{center}
    \caption{Example \ref{subsec:lshape2D}. Convergence history of the
      ten lowest computed eigenvalues on the L-shaped domain in
      different meshes. }
    \label{table:lshape2D-FEM}
    \begin{tabular}{|c c c c |c| c|}
      \hline
      \multicolumn{6}{|c|}{$\CT^4_h$}\\
      \hline
      \hline
      $N=126$             &  $N=480$         &   $N=1818$         & $N=7178$ & Order & $\lambda_{\text{extr}}$ \\ 
      \hline
        1.8454  &     2.2456  &     2.4245  &     2.5422  & 1.42 &     2.5868  \\
        3.1214  &     3.2005  &     3.2338  &     3.2528  & 1.23 &     3.2637  \\
       20.6871  &    21.1653  &    21.2955  &    21.3341  & 1.98 &    21.3445  \\
       20.7923  &    21.5447  &    21.7254  &    21.7818  & 2.15 &    21.7897  \\
       32.5976  &    33.3889  &    33.5864  &    33.6350  & 2.14 &    33.6478  \\
       35.2018  &    38.5684  &    39.3554  &    39.6577  & 2.21 &    39.6706  \\
       73.9685  &    77.4825  &    78.3661  &    78.6065  & 2.14 &    78.6590  \\
      127.0634  &   138.2728  &   141.3847  &   142.5152  & 2.00 &   142.7097  \\
      146.2965  &   159.1158  &   162.1393  &   162.8566  & 2.32 &   162.9865  \\
      143.9508  &   177.2920  &   187.7400  &   192.9284  & 1.96 &   193.4054  \\
      \hline
      \hline
      \multicolumn{6}{|c|}{$\CT^5_h$}\\
      \hline
      \hline
      $N=259$             &  $N=515$         &   $N=1027$         & $N=2051$ & Order & $\lambda_{\text{extr}}$ \\ 
      \hline
        2.1642  &     2.2964  &     2.3342  &     2.4218  & 1.30 &     2.5099  \\
        3.0236  &     3.1050  &     3.1478  &     3.1850  & 1.43 &     3.2332  \\
       19.1866  &    20.2113  &    20.7780  &    21.0344  & 2.09 &    21.3098  \\
       19.5528  &    20.5675  &    21.1742  &    21.4604  & 1.91 &    21.8137  \\
       21.4139  &    26.6081  &    29.8566  &    31.6832  & 2.30 &    33.2864  \\
       24.4630  &    30.0922  &    33.9049  &    36.3514  & 2.01 &    39.0219  \\
       70.3257  &    74.5462  &    76.5876  &    77.6256  & 2.29 &    78.4458  \\
      129.3497  &   136.1305  &   139.1755  &   140.9823  & 2.28 &   142.2166  \\
      132.9624  &   147.6219  &   155.9331  &   159.6513  & 2.24 &   163.3600  \\
      165.2225  &   178.2810  &   183.4868  &   188.0302  & 2.21 &   190.4697  \\
      \hline
      \hline             
    \end{tabular}
  \end{center}
\end{table}

\begin{figure}[t!]
  \centering
  \begin{minipage}{0.32\linewidth}\centering
    \includegraphics[scale=0.064, trim=34cm 4cm 34cm 4cm,clip]{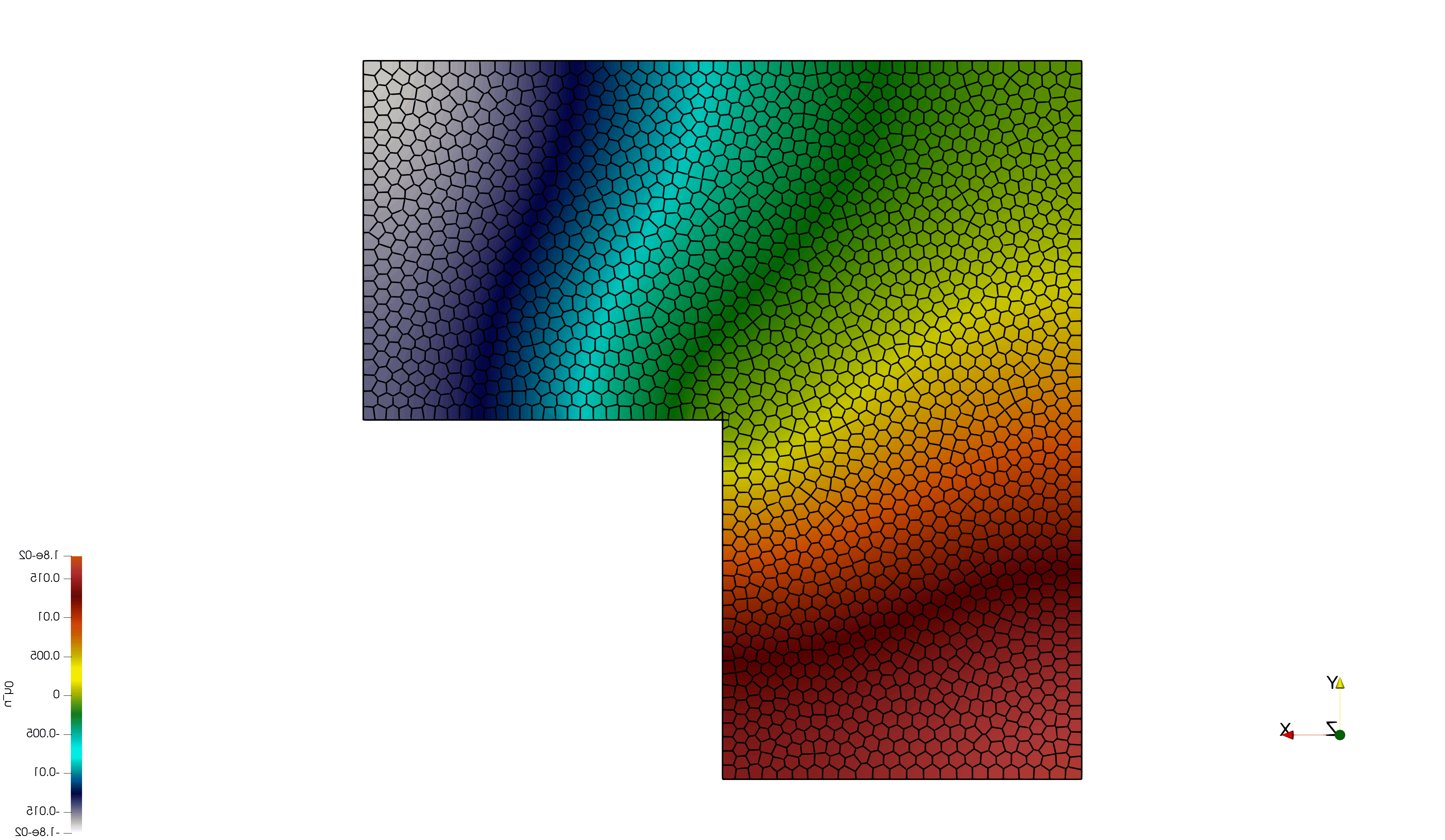}\\
		    {\footnotesize $u_{1,h}$}
  \end{minipage}
  \begin{minipage}{0.32\linewidth}\centering
    \includegraphics[scale=0.064, trim=34cm 4cm 34cm 4cm,clip]{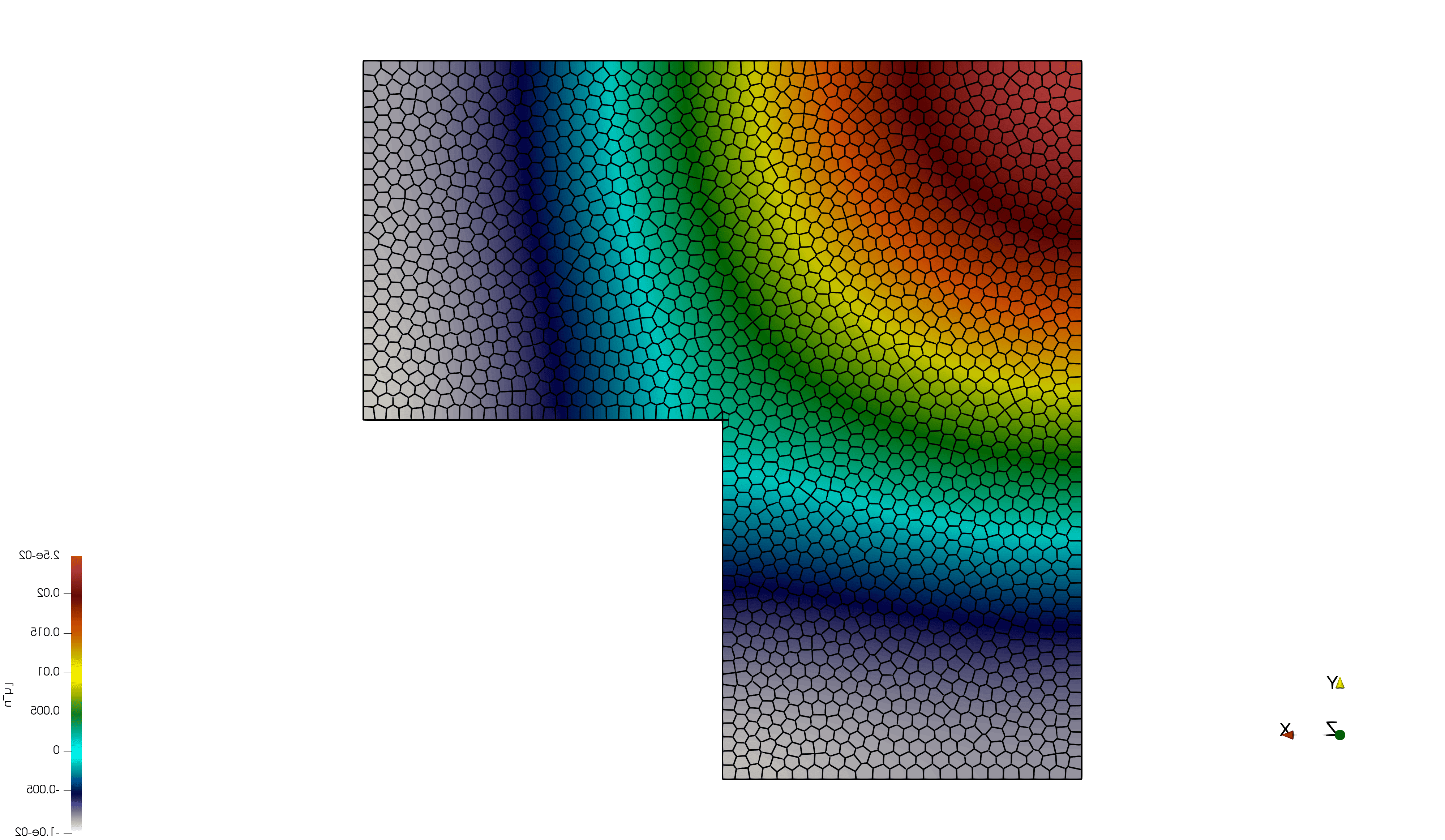}\\
		    {\footnotesize $u_{2,h}$}
  \end{minipage}
  \begin{minipage}{0.32\linewidth}\centering
    \includegraphics[scale=0.064, trim=34cm 4cm 34cm 4cm,clip]{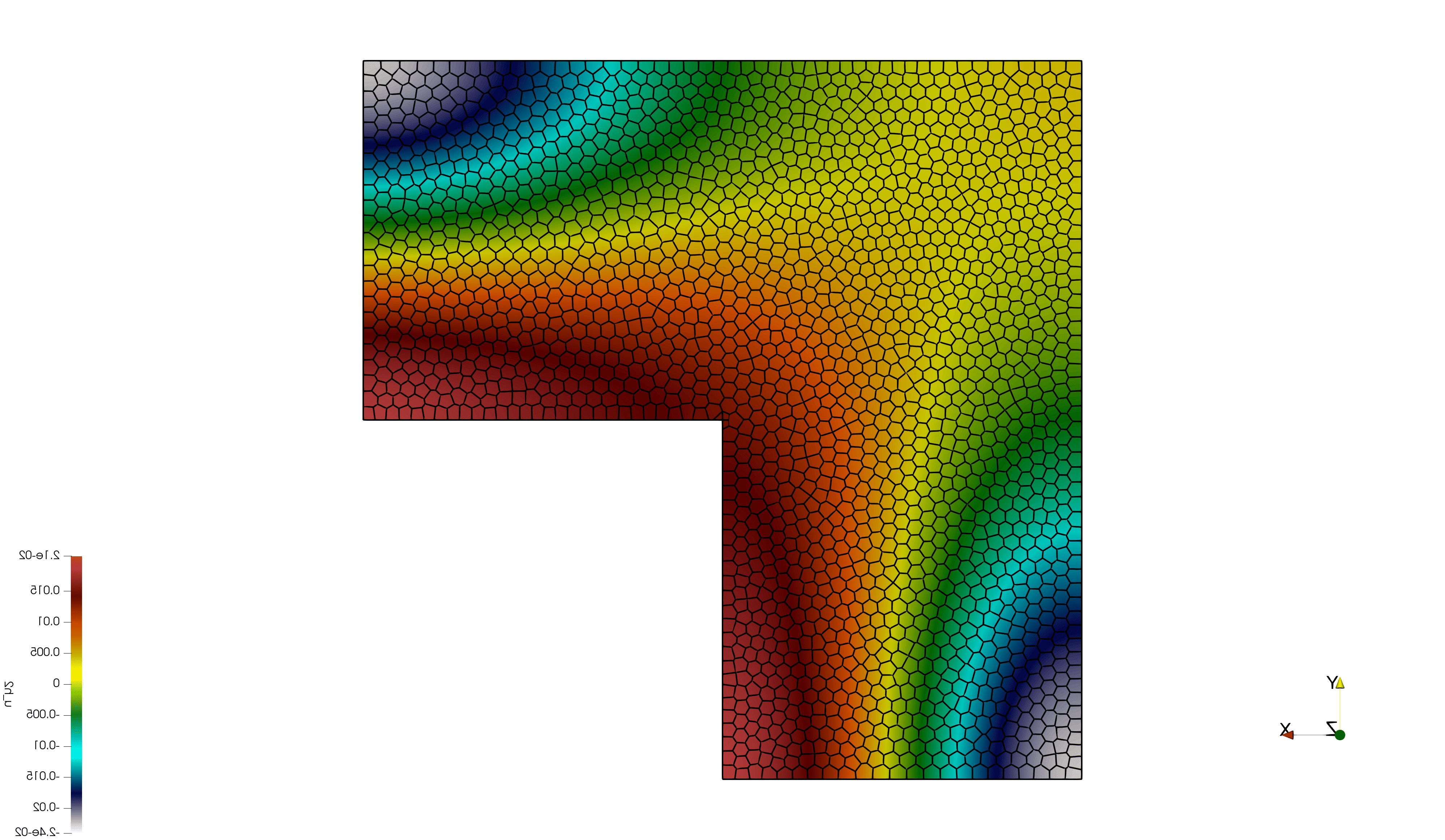}\\
		    {\footnotesize $u_{3,h}$}
  \end{minipage}
  \hspace*{0.1cm}\begin{minipage}{0.32\linewidth}\centering
    \includegraphics[scale=0.125, trim=53cm 2cm 53cm 2cm,clip]{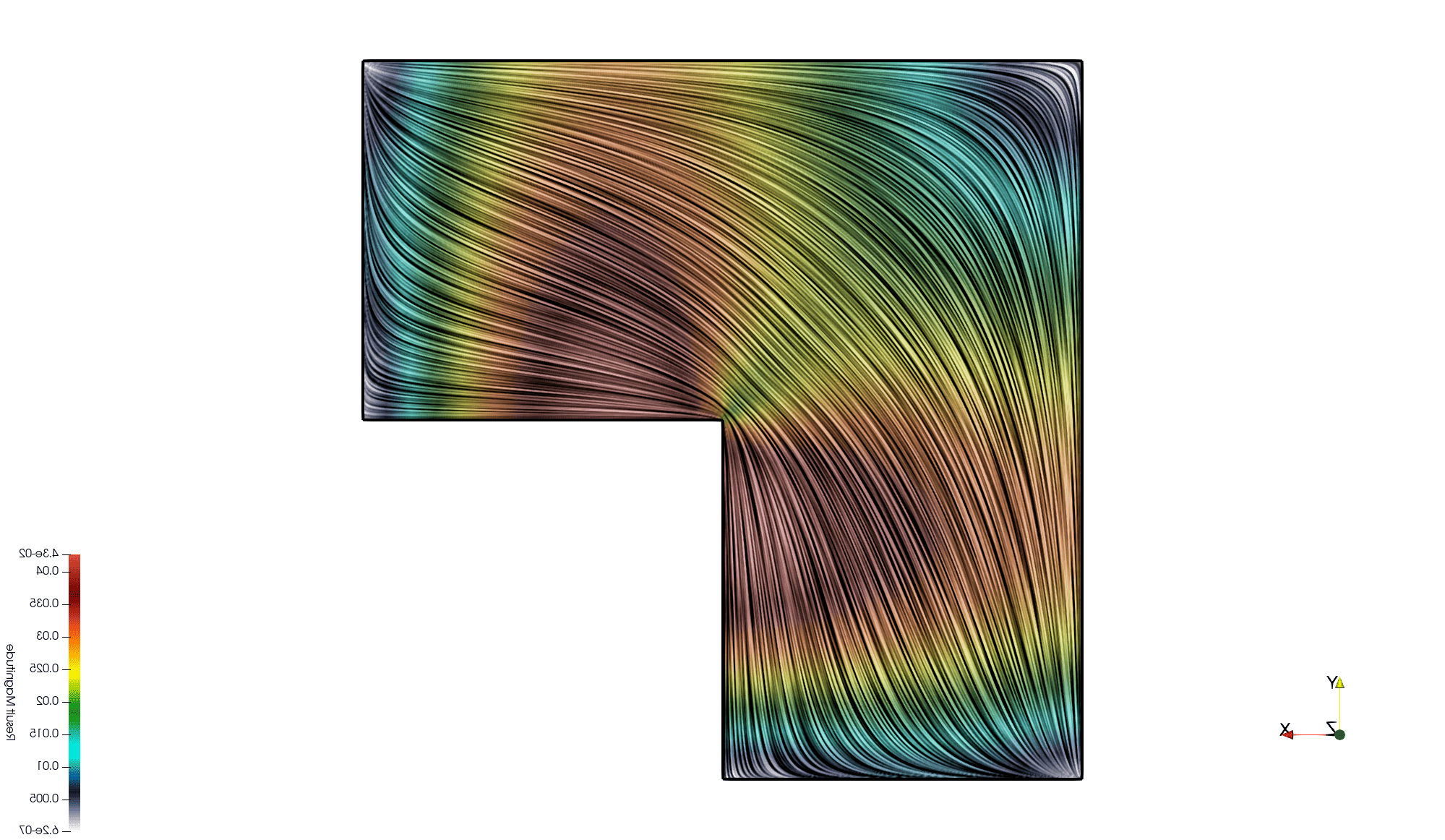}\\
		    {\footnotesize $\nabla u_{1,h}$}
  \end{minipage}
  \begin{minipage}{0.32\linewidth}\centering
    \includegraphics[scale=0.125, trim=53cm 2cm 53cm 2cm,clip]{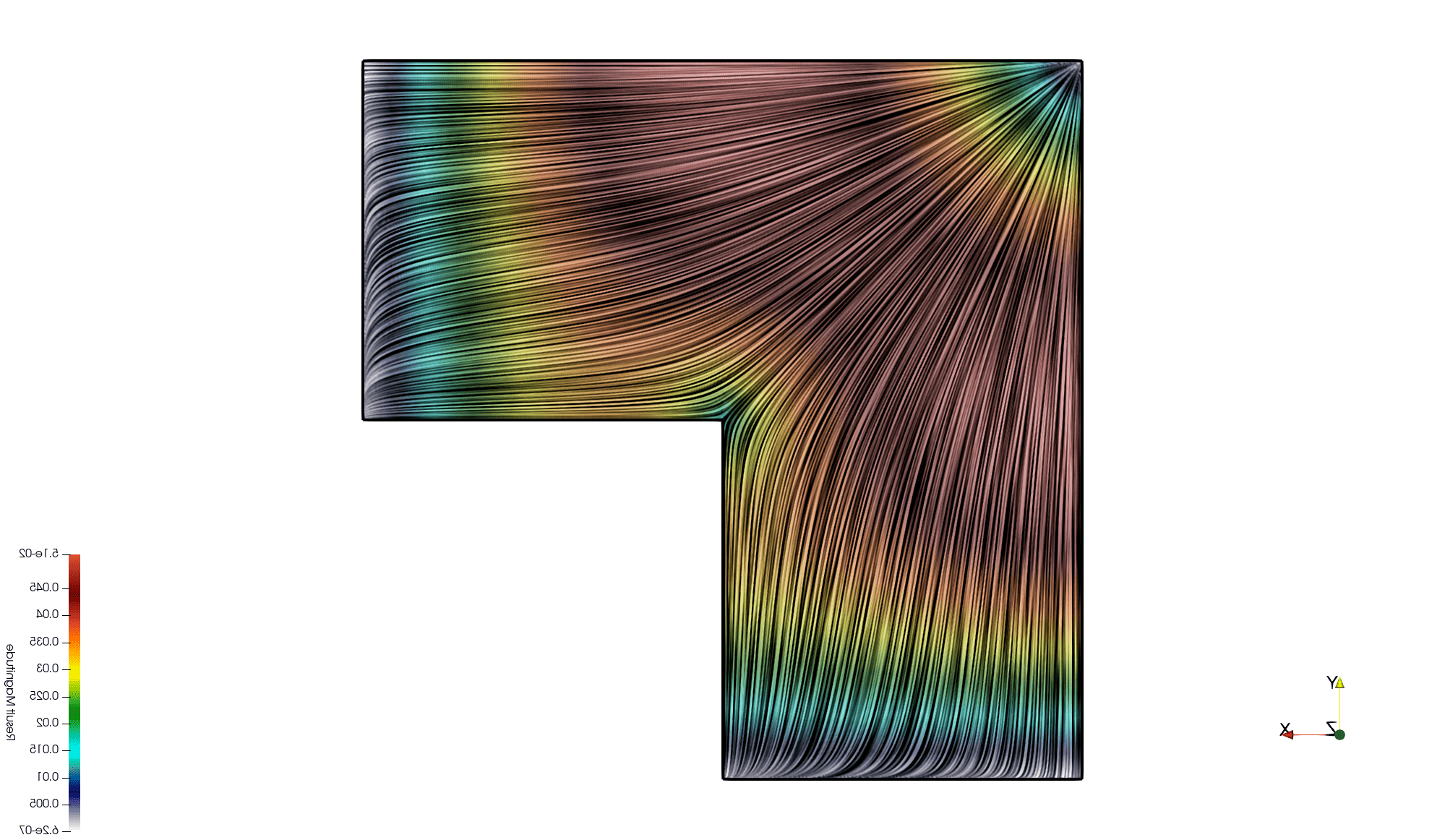}\\
		    {\footnotesize $\nabla u_{2,h}$}
  \end{minipage}
  \begin{minipage}{0.32\linewidth}\centering
    \includegraphics[scale=0.125, trim=53cm 2cm 53cm 2cm,clip]{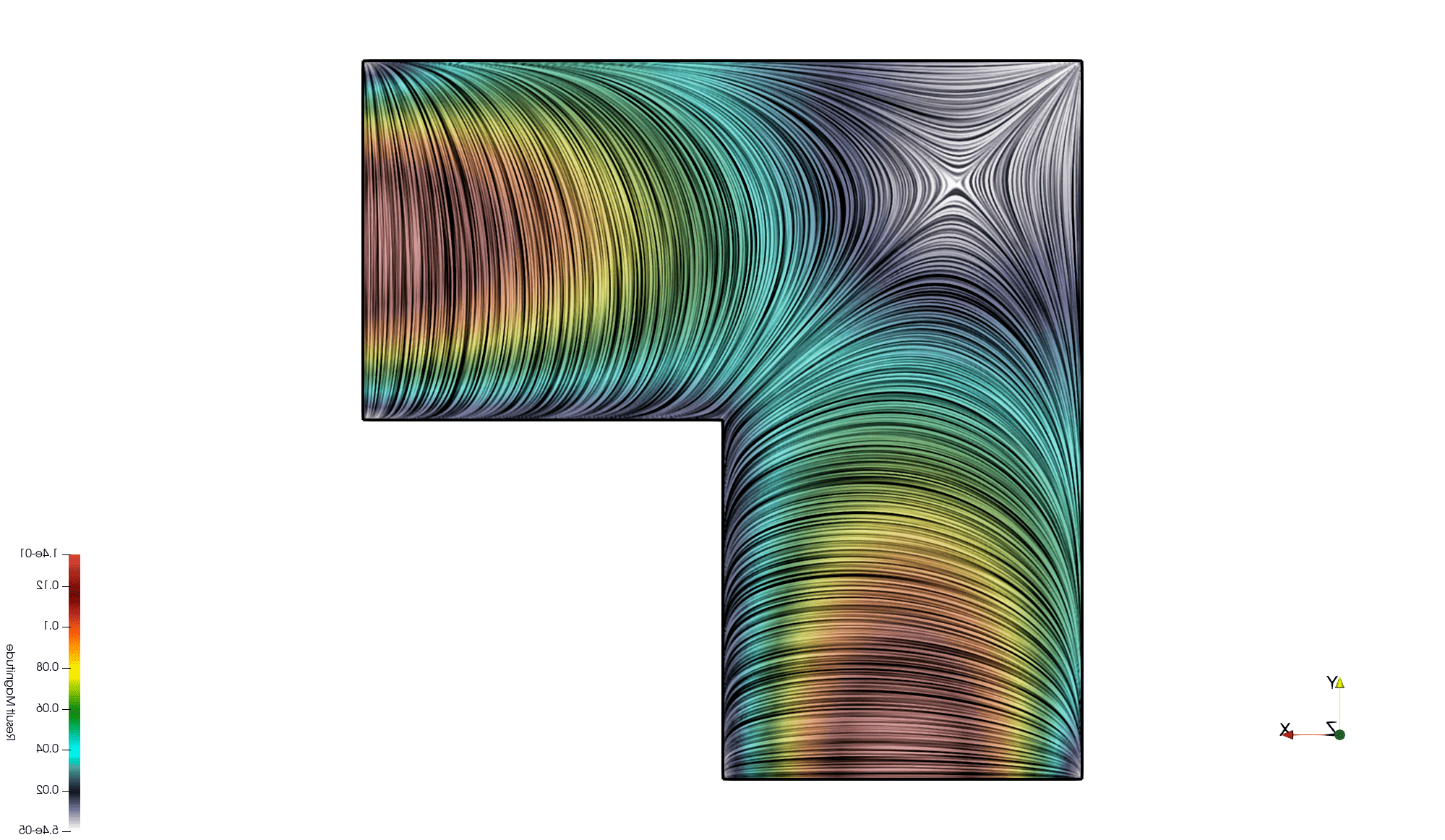}\\
		    {\footnotesize $\nabla u_{3,h}$}
  \end{minipage}
  \caption{Test~\ref{subsec:lshape2D}. Surface plot and line integral convolution plots of the gradient fields of the first, second, and third
    lowest computed eigenvalues on the mesh $\CT^5_h$ for $N=22$ while
    a non-conforming approach is chosen.}
  \label{fig:lshape2D-uh-graduh}
\end{figure}
\subsection{Circular domain}\label{subsec:circle2D}
We now consider a smooth, non-polygonal domain. Specifically, we take
$$
\Omega:=\left\{(x,y)\in\mathbb{R}^2\,:\, x^2 + y^2 \leq 1/4\right\}.
$$ Since the domain is approximated by polygonal meshes, the method incurs a geometric variational crime. Consequently, in this setting, the error in the approximation of the eigenvalues is 
$\mathcal{O}(h^2)$,  or equivalentely $\mathcal{O}(N^{-1})$. 
Figure~\ref{fig:mesh_PolyGon-circle} shows representative meshes used in the computations.
\begin{figure}[t!]\centering
  \begin{minipage}{0.23\linewidth}\centering
    \includegraphics[scale=0.09, trim=17cm 2cm 17cm 2cm,clip]{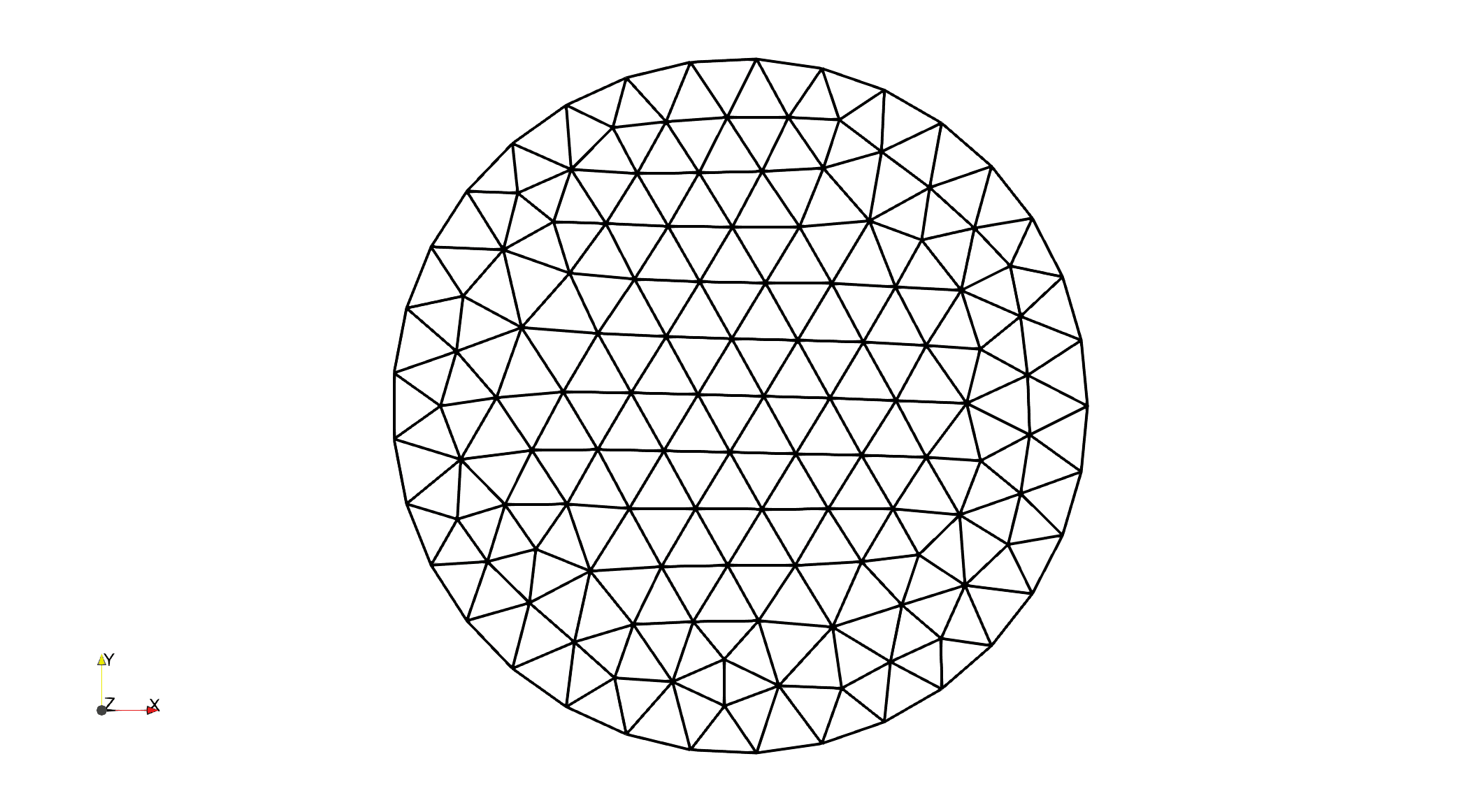}\\
		    {$\CT^6_h$}
  \end{minipage}
  \begin{minipage}{0.23\linewidth}\centering
    \includegraphics[scale=0.09, trim=17cm 2cm 17cm 2cm,clip]{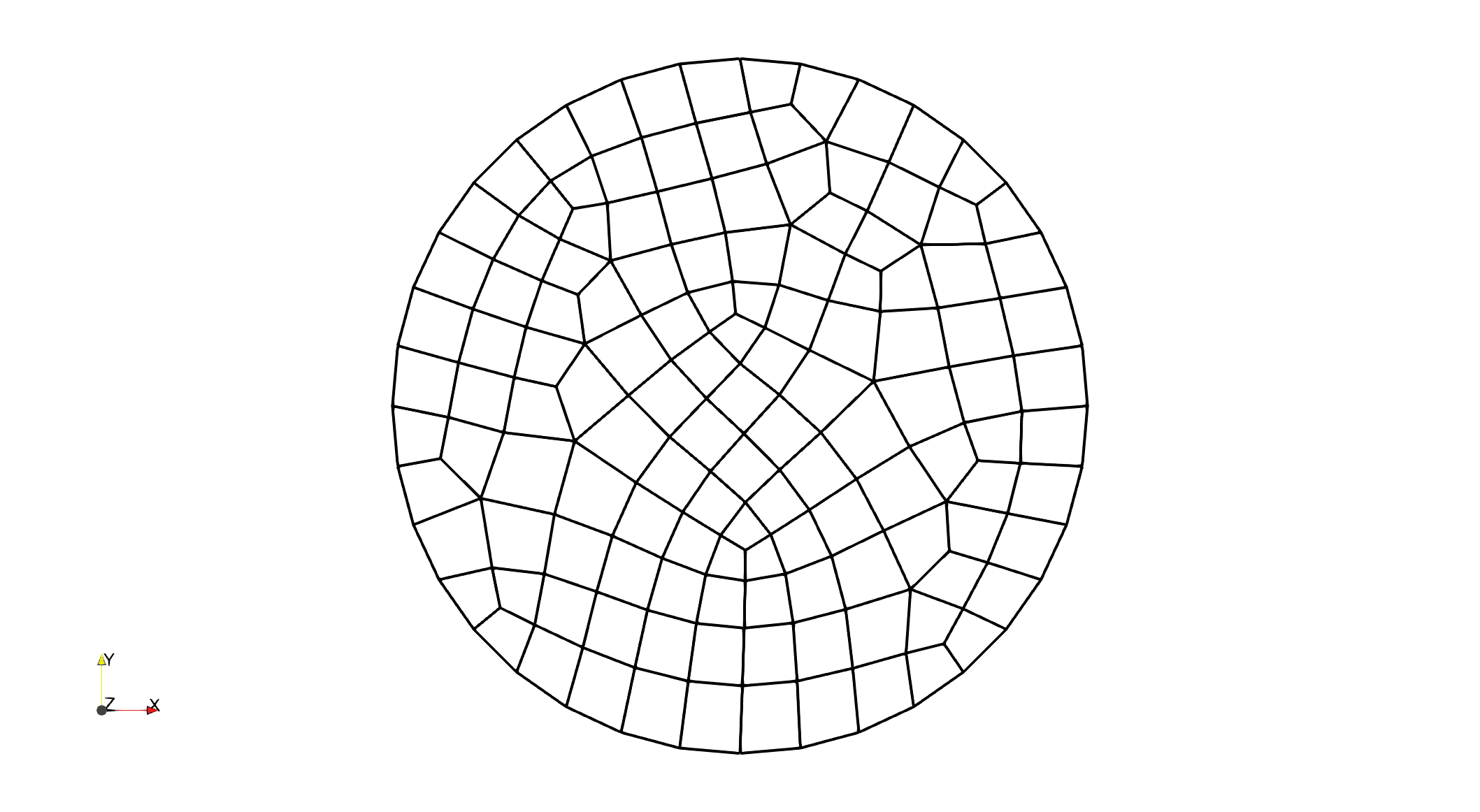}\\
		    {$\CT^7_h$}
  \end{minipage}
  \begin{minipage}{0.23\linewidth}\centering
    \includegraphics[scale=0.09, trim=17cm 2cm 17cm 2cm,clip]{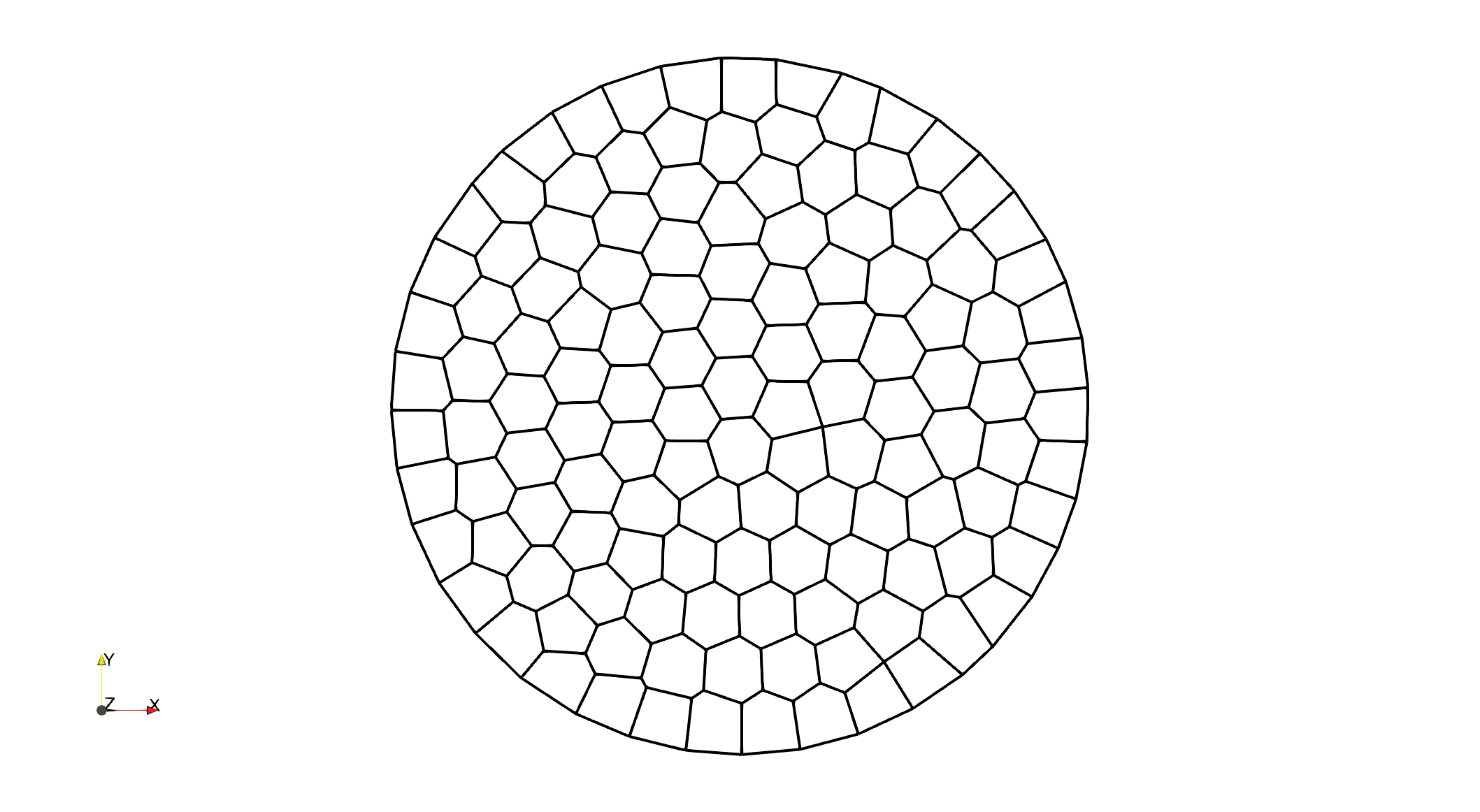}\\
		    {$\CT^8_h$}
  \end{minipage}
  \caption{Test~\ref{subsec:circle2D}. Representative mesh families used in the eigenvalue computations.}
  \label{fig:mesh_PolyGon-circle} 
\end{figure}
The convergence history of the first ten positive eigenvalues computed with the nonconforming scheme is reported in Table~\ref{table:circle2D-FEM}. A quadratic convergence rate is observed across all mesh families, which is consistent with the expected behavior under geometric approximation of the domain. The extrapolated eigenvalues were obtained using highly refined meshes combined with a least-squares fit. 

Figure~\ref{fig:circle2D-uh-graduh} displays selected eigenfunctions and the corresponding reconstructed gradient fields computed on Voronoi meshes. In contrast to the polygonal cases, no singular behavior is present, and the eigenfunctions exhibit smooth distributions over the domain. The gradient fields are tangential along the boundary, confirming the enforcement of the homogeneous Neumann condition $\partial_{\mathbf n}u=0$.
\begin{table}[t!]
  \centering 
  \footnotesize
  \begin{center}
    \caption{Test~\ref{subsec:circle2D}. Convergence history of the
      ten lowest eigenvalues computed with the non-conforming approach
      on the circle domain in different meshes. }
    \label{table:circle2D-FEM}
    \begin{tabular}{|c c c c |c| c|}
      \hline
      \hline
      \multicolumn{6}{|c|}{$\CT^6_h$}\\
      \hline
      \hline
      $N=225$             &  $N=761$         &   $N=2962$         & $N=11778$ & Order & $\lambda_{\text{extr}}$ \\ 
      \hline
      23.7479  &    23.9326  &    23.9843  &    23.9843  & 2.23 &    23.9992  \\
        23.7502  &    23.9334  &    23.9843  &    23.9843  & 2.04 &    24.0006  \\
       149.6535  &   157.2891  &   159.3755  &   159.3755  & 2.31 &   159.9572  \\
       149.6909  &   157.3259  &   159.3809  &   159.3809  & 2.07 &   160.0227  \\
       427.9333  &   483.8660  &   499.4299  &   499.4299  & 2.35 &   503.6499  \\
       428.9644  &   484.5253  &   499.6030  &   499.6030  & 1.99 &   504.4781  \\
       849.1243  &  1069.3200  &  1133.5190  &  1133.5190  & 2.15 &  1152.1133  \\
       857.4570  &  1070.6139  &  1133.6290  &  1133.6290  & 1.99 &  1153.6119  \\
      1357.7182  &  1950.6772  &  2143.9016  &  2143.9016  & 1.89 &  2208.0568  \\
      1376.8635  &  1954.3054  &  2144.3432  &  2144.3432  & 1.99 &  2203.9002  \\
      \hline\hline
      \multicolumn{6}{|c|}{$\CT^7_h$}\\
      \hline\hline
      $N=129$             &  $N=414$         &   $N=1476$         & $N=5847$ & Order & $\lambda_{\text{extr}}$ \\ 
      \hline
        23.6785  &    23.8935  &    23.9704  &    23.9704  & 2.11 &    23.9992  \\
        23.6915  &    23.8975  &    23.9715  &    23.9715  & 2.01 &    24.0006  \\
       149.6325  &   156.7202  &   159.0968  &   159.0968  & 2.19 &   159.9572  \\
       150.0780  &   156.7259  &   159.1471  &   159.1471  & 2.01 &   160.0227  \\
       433.0186  &   481.3161  &   497.9471  &   497.9471  & 2.20 &   503.6499  \\
       433.8642  &   482.0060  &   498.2037  &   498.2037  & 1.94 &   504.4781  \\
       867.3765  &  1061.5059  &  1128.3734  &  1128.3734  & 2.08 &  1152.1133  \\
       882.4124  &  1063.1044  &  1128.7506  &  1128.7506  & 1.94 &  1153.6119  \\
      1412.3223  &  1928.0070  &  2130.4845  &  2130.4845  & 1.86 &  2208.0568  \\
      1424.8487  &  1940.3125  &  2130.9424  &  2130.9424  & 1.95 &  2203.9002  \\
      \hline\hline
      \multicolumn{6}{|c|}{$\CT^8_h$}\\
      \hline\hline
      $N=128$             &  $N=384$         &   $N=768$         & $N=1536$ & Order & $\lambda_{\text{extr}}$ \\ 
      \hline
        23.7266  &    23.9182  &    23.9619  &    23.9808  & 2.13 &    23.9992  \\
        23.7777  &    23.9294  &    23.9657  &    23.9832  & 2.00 &    24.0006  \\
       153.6024  &   158.1142  &   159.0829  &   159.5611  & 2.16 &   159.9572  \\
       154.1278  &   158.1753  &   159.1487  &   159.5762  & 2.04 &   160.0227  \\
       460.3176  &   491.3718  &   497.7227  &   501.1154  & 2.19 &   503.6499  \\
       464.9252  &   491.6933  &   498.5945  &   501.2138  & 2.00 &   504.4781  \\
       983.5687  &  1100.8885  &  1128.0101  &  1140.2852  & 2.09 &  1152.1133  \\
       989.9151  &  1102.3202  &  1129.4662  &  1141.1000  & 2.04 &  1153.6119  \\
      1708.8633  &  2043.8460  &  2129.1779  &  2165.7591  & 1.96 &  2208.0568  \\
      1723.7599  &  2051.8494  &  2131.5298  &  2166.3778  & 2.02 &  2203.9002  \\
      \hline
      \hline             
    \end{tabular}
  \end{center}
\end{table}

\begin{figure}[!h]
  \centering
  \begin{minipage}{0.32\linewidth}\centering
    \includegraphics[scale=0.064, trim=34cm 4cm 34cm 4cm,clip]{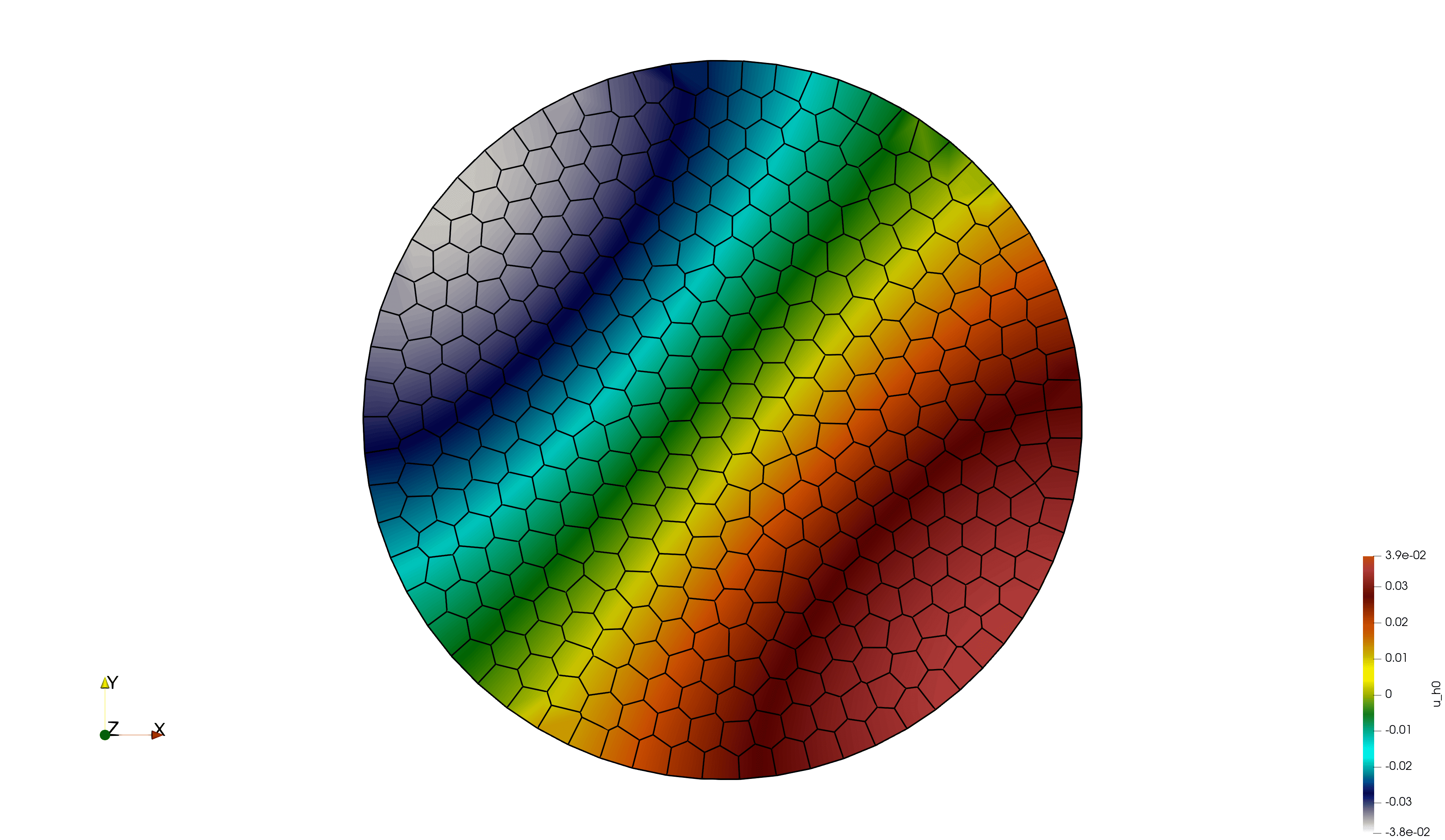}\\
		    {\footnotesize $ u_{1,h}$}
  \end{minipage}
  \begin{minipage}{0.32\linewidth}\centering
    \includegraphics[scale=0.064, trim=34cm 4cm 34cm 4cm,clip]{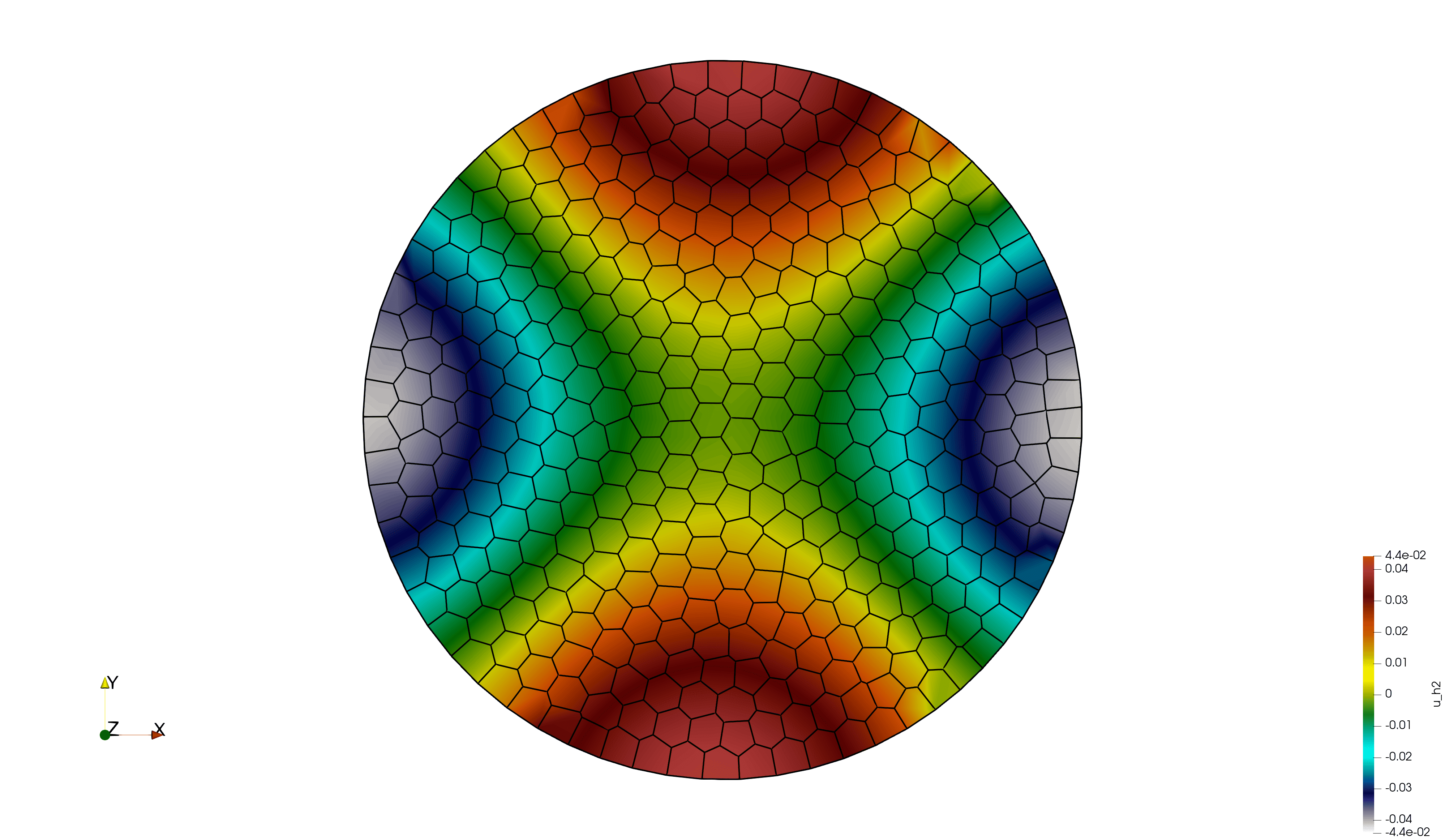}\\
		    {\footnotesize $ u_{3,h}$}
  \end{minipage}
  \begin{minipage}{0.32\linewidth}\centering
    \includegraphics[scale=0.064, trim=34cm 4cm 34cm 4cm,clip]{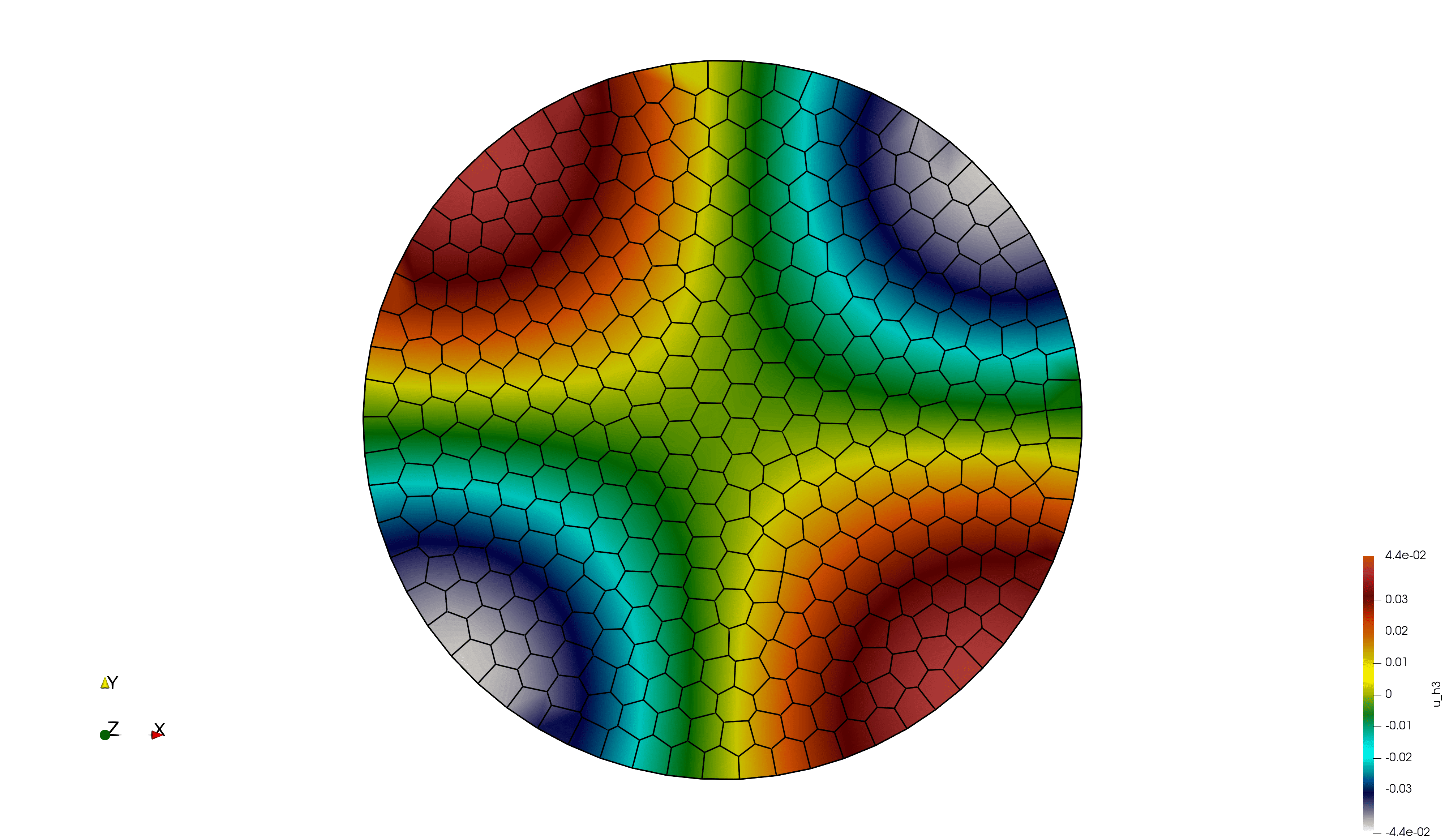}\\
		    {\footnotesize $ u_{4,h}$}
  \end{minipage}
  \hspace*{0.1cm}\begin{minipage}{0.32\linewidth}\centering
    \includegraphics[scale=0.125, trim=53cm 2cm 53cm 2cm,clip]{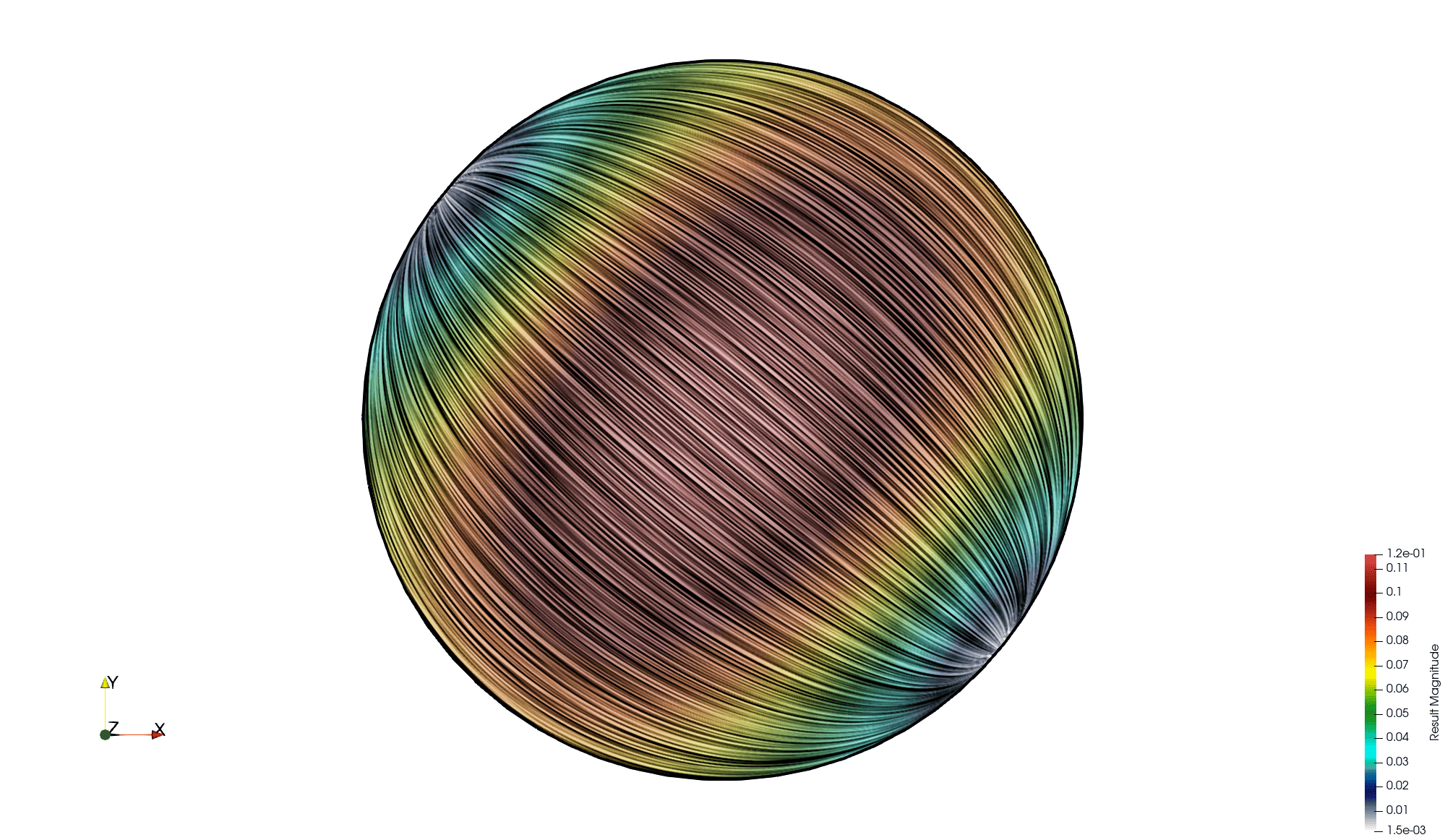}\\
		    {\footnotesize $\nabla u_{1,h}$}
  \end{minipage}
  \begin{minipage}{0.32\linewidth}\centering
    \includegraphics[scale=0.125, trim=53cm 2cm 53cm 2cm,clip]{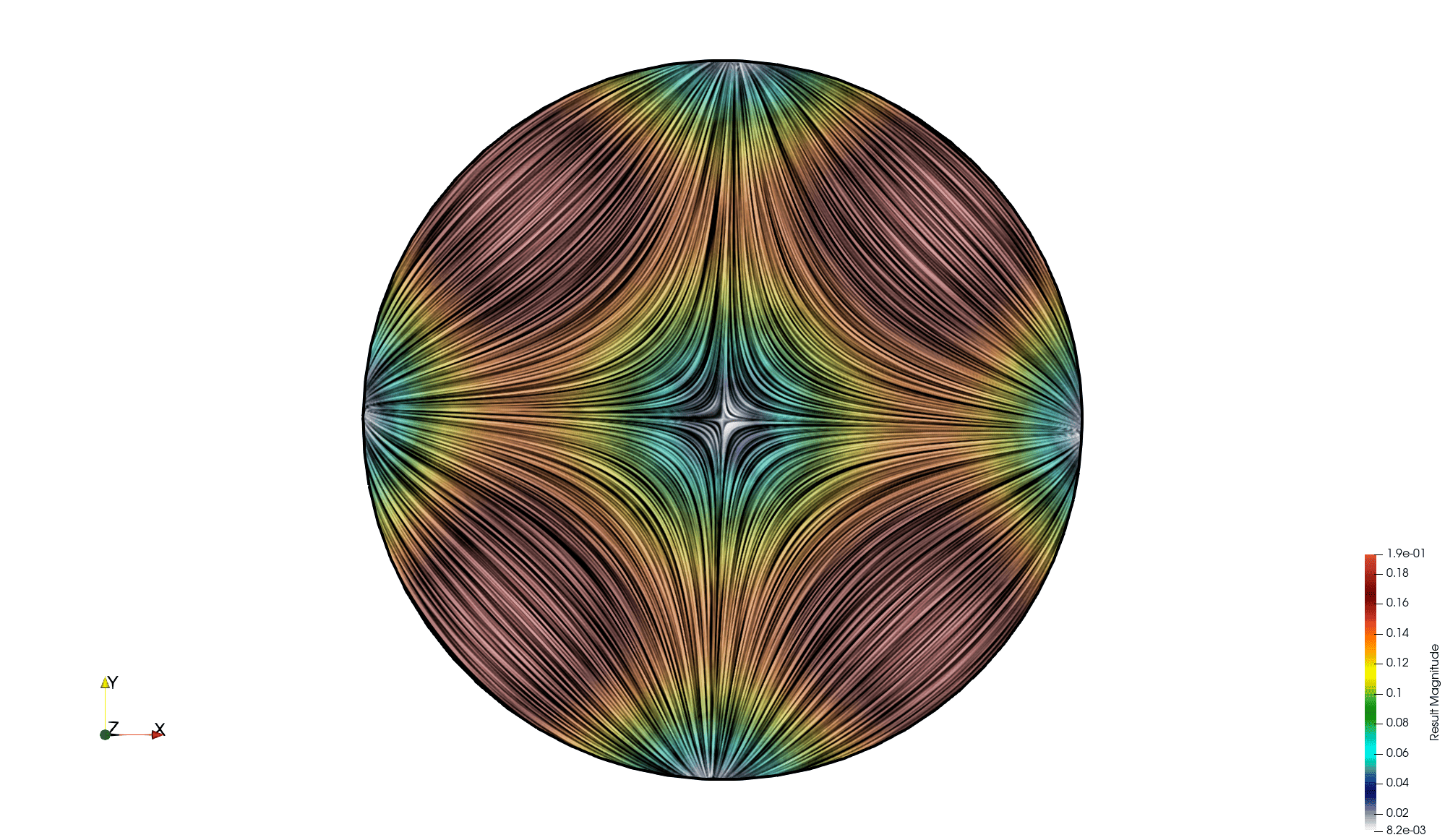}\\
		    {\footnotesize $\nabla u_{3,h}$}
  \end{minipage}
  \begin{minipage}{0.32\linewidth}\centering
    \includegraphics[scale=0.125, trim=53cm 2cm 53cm 2cm,clip]{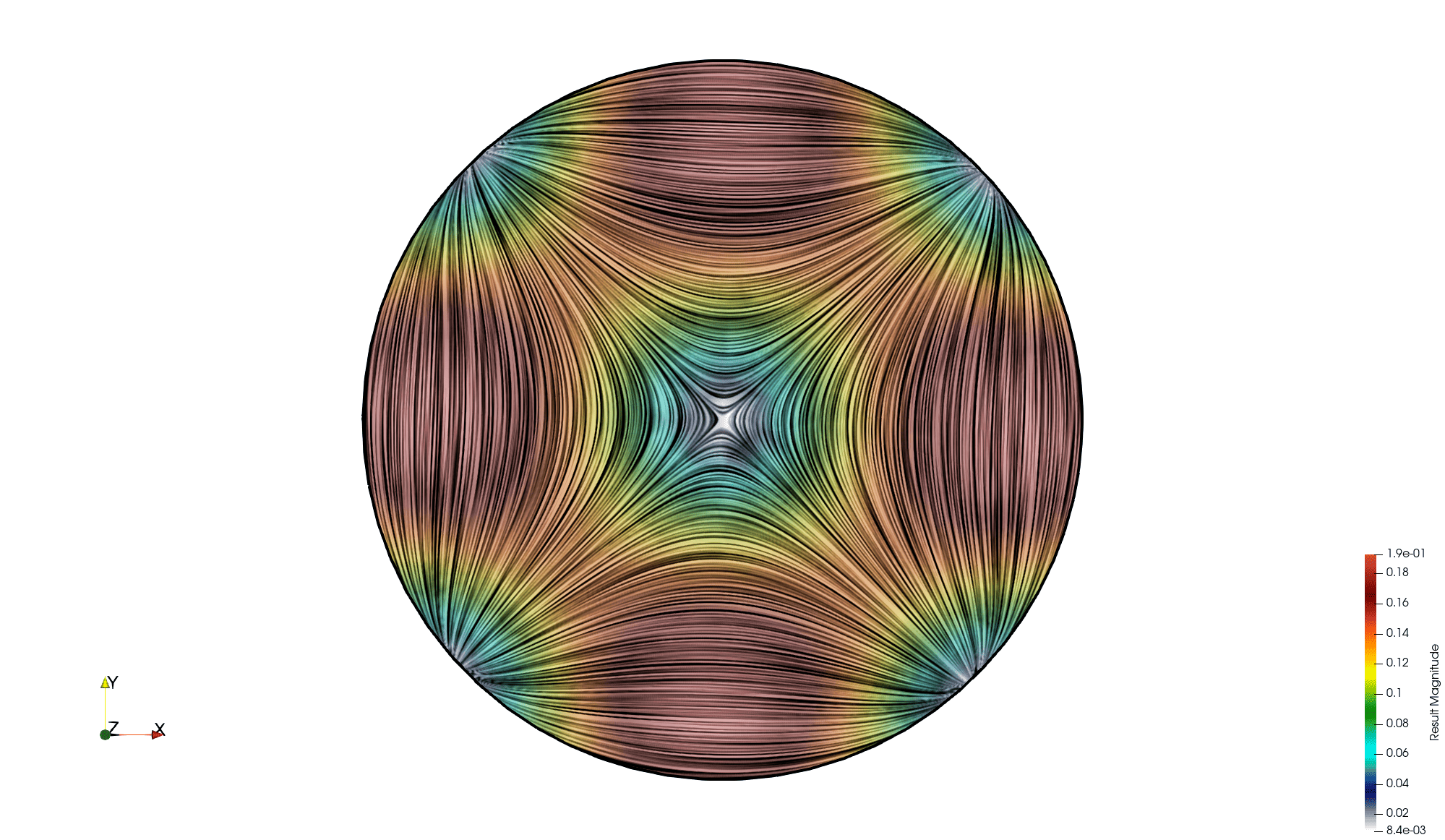}\\
		    {\footnotesize $\nabla u_{4,h}$}
  \end{minipage}
  \caption{Test~\ref{subsec:circle2D}. Surface plot and line integral convolution plots of the gradient fields of the first, third, and fourth lowest eigenvalues computed with non-conforming scheme on the mesh $\CT^4_h$.}
  \label{fig:circle2D-uh-graduh}
\end{figure}

\subsection{Circular Plate with Holes}\label{subsec:circle-holes}
We conclude the numerical section with a test that goes beyond the scope of the proposed theory. We consider a non-polygonal multiply connected domain defined by

\[
  \Omega := \Omega_m \setminus \bigcup_{i=1}^4 \Omega_i,
\]
where

$$
\begin{aligned}
  \Omega_m &:= \left\{(x,y)\in\mathbb{R}^2 : x^2+y^2 \leq 1 \right\},\\
  \Omega_1 &:= \left\{(x,y)\in\mathbb{R}^2 : (x-0.4)^2+(y-0.4)^2 \leq 0.04 \right\},\\
  \Omega_2 &:= \left\{(x,y)\in\mathbb{R}^2 : (x+0.4)^2+(y-0.4)^2 \leq 0.04 \right\},\\
  \Omega_3 &:= \left\{(x,y)\in\mathbb{R}^2 : (x+0.4)^2+(y+0.4)^2 \leq 0.04 \right\},\\
  \Omega_4 &:= \left\{(x,y)\in\mathbb{R}^2 : (x-0.4)^2+(y+0.4)^2 \leq 0.04 \right\}.
\end{aligned}
$$

Thus, $\Omega$ is the unit disk centered at the origin with four circular holes of radius $0.2$. Representative meshes are shown in Figure~\ref{fig:mesh_PolyGon-circle-hole}. We compare two boundary-condition settings. In the first one, the zero normal derivative condition is imposed on the whole boundary, $\partial\Omega_m \cup \bigcup_{i=1}^4 \partial\Omega_i$. In the second one, it is imposed only on the exterior boundary $\partial\Omega_m$. The comparison is performed using the nonconforming method.
\begin{figure}[t!]\centering
  \begin{minipage}{0.23\linewidth}\centering
    \includegraphics[scale=0.09, trim=17cm 2cm 17cm 2cm,clip]{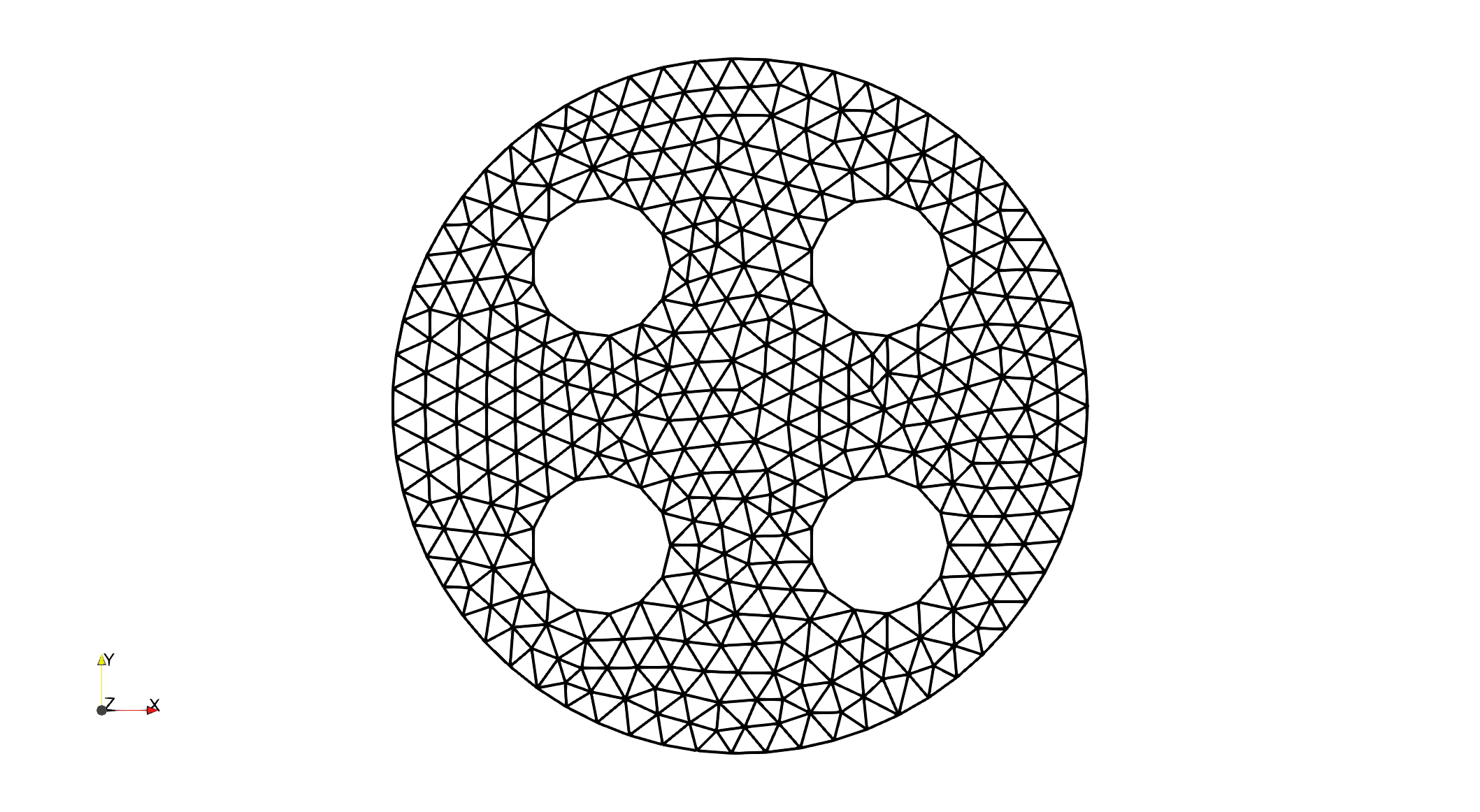}\\
		    {$\CT^9_h$}
  \end{minipage}
  \begin{minipage}{0.23\linewidth}\centering
    \includegraphics[scale=0.09, trim=17cm 2cm 17cm 2cm,clip]{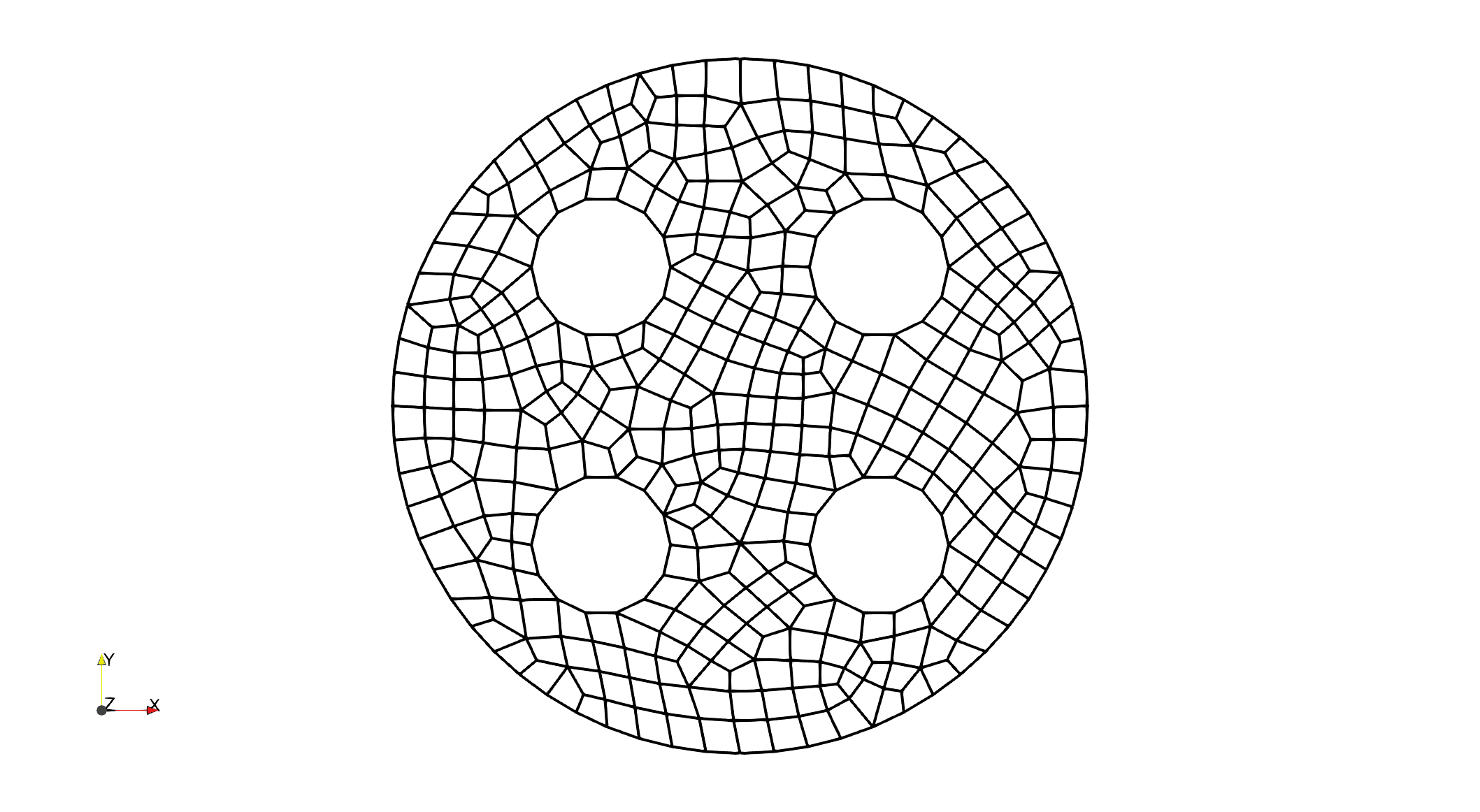}\\
		    {$\CT^{10}_h$}
  \end{minipage}
  \begin{minipage}{0.23\linewidth}\centering
    \includegraphics[scale=0.09, trim=17cm 2cm 17cm 2cm,clip]{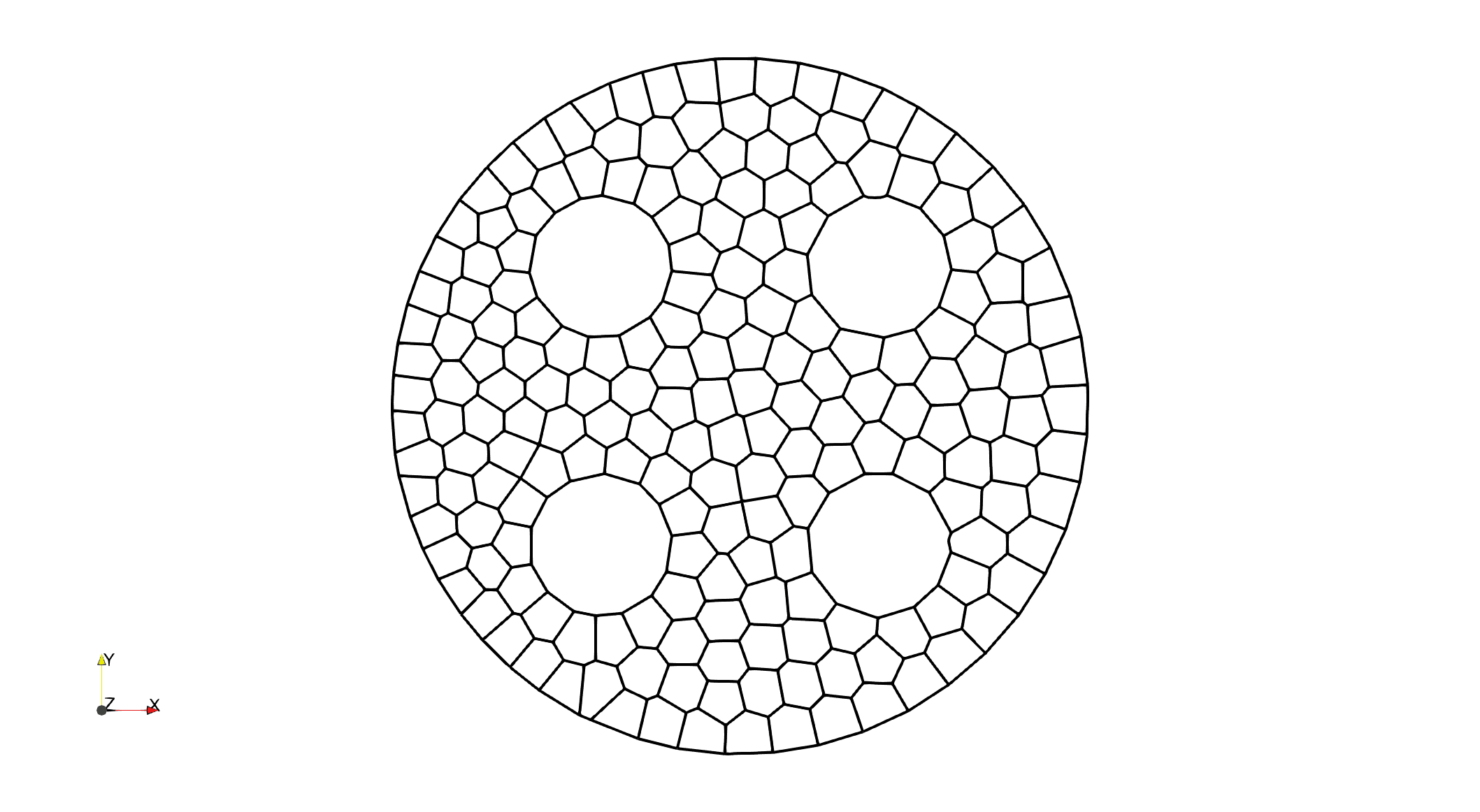}\\
		    {$\CT^{11}_h$}
  \end{minipage}
  \caption{Test \ref{subsec:circle-holes}. Examples of the mesh families used in the	eigenvalue calculations.}
  \label{fig:mesh_PolyGon-circle-hole} 
\end{figure}
Tables~\ref{tabla:circle-holes-zero-all} and \ref{tabla:circle-holes-zero-external} report the lowest positive computed eigenvalues for the two boundary-condition settings. In Table~\ref{tabla:circle-holes-zero-all}, the zero normal derivative condition is imposed on all boundary components, whereas in Table~\ref{tabla:circle-holes-zero-external}, it is imposed only on the exterior boundary. The results show that imposing the condition on all boundary components leads to larger eigenvalues. This is consistent with the additional constraint imposed on the admissible discrete space.

Figures~\ref{fig:circle_hole2DNfree-uh-graduh} and \ref{fig:circle_hole2Dfree-uh-graduh} compare selected eigenfunctions and their reconstructed gradient fields for both settings. Although the surface plots of the eigenfunctions are qualitatively similar, the gradient fields exhibit clear differences near the internal boundaries. 

When the zero normal derivative condition is imposed on all boundary components, the reconstructed gradients are tangential not only to the outer boundary but also to the boundaries of the holes. In contrast, when the condition is imposed only on $\partial\Omega_m$, the gradient fields are not constrained in the same way along the internal boundaries. Similar qualitative behavior is observed when the conforming method was used.
\begin{table}[t!]
  \footnotesize
  \caption{Test \ref{subsec:circle-holes}. Comparison between the
    lowest computed eigenvalue for different meshes with zero normal
    derivatives on $\partial\Omega_m\cup\partial\Omega_j,\,
    j=1,2,3,4$. The mesh level for each $\CT^i_h$, $i=9,10,11$ is
    $N=256$, and non-conforming method is chosen.}
  \label{tabla:circle-holes-zero-all}
  \begin{center}
    \begin{tabular}{|l | c  | c |  c| }
      \hline
      \hline
      &$\CT^9_h\, (N=10174)$ & $\CT^{10}_h, \, (N=5130)$ & $\CT^{11}_h\, (N=3072)$\\
      \hline
      \hline
      \multirow{5}{0.01\linewidth}{$\lambda_h$}	
      &	9.174694 &      9.363489 &      9.361099 \\
      & 9.176687 &      9.363798 &      9.361239 \\
      &22.861174 &     23.223021 &     23.220259 \\
      &33.325811 &     33.974838 &     33.968979 \\
      &63.381595 &     64.392810 &     64.387450 \\
      &63.420043 &     64.393832 &     64.390200 \\
      &97.137117 &     98.837564 &     98.819654 \\
      \hline  
      \hline
    \end{tabular}
  \end{center}
\end{table}
 
\begin{table}[h!]
  \footnotesize
  \caption{Test~\ref{subsec:circle-holes}. Comparison between the
    lowest computed eigenvalue for different meshes with zero normal
    derivatives on $\partial\Omega_m$. The mesh level for each
    $\CT^i_h,\, i=9,10,11,$ is $N=256$, and non-conforming method is
    chosen.}
  \label{tabla:circle-holes-zero-external}
  \begin{center}
    \begin{tabular}{|l | c  | c |  c| }
      \hline\hline
      &$\CT^9_h\, (N=10174)$ & $\CT^{10}_h, \, (N=5130)$ & $\CT^{11}_h\, (N=3072)$\\
      \hline
      \hline
      \multirow{5}{0.01\linewidth}{$\lambda_h$}	
      &      1.583082 &      1.583001 &      1.576237 \\
      &      1.583083 &      1.583007 &      1.576419 \\
      &     10.008667 &     10.007804 &      9.953473 \\
      &     15.501407 &     15.500153 &     15.433397 \\
      &     41.121759 &     41.117202 &     40.854942 \\
      &     41.122026 &     41.118265 &     40.868167 \\
      &     68.048024 &     68.036866 &     67.460732 \\
      \hline\hline       
    \end{tabular}
  \end{center}
\end{table}

\begin{figure}[!h]
  \centering
  \begin{minipage}{0.32\linewidth}\centering
    \includegraphics[scale=0.064, trim=37cm 4cm 37cm 4cm,clip]{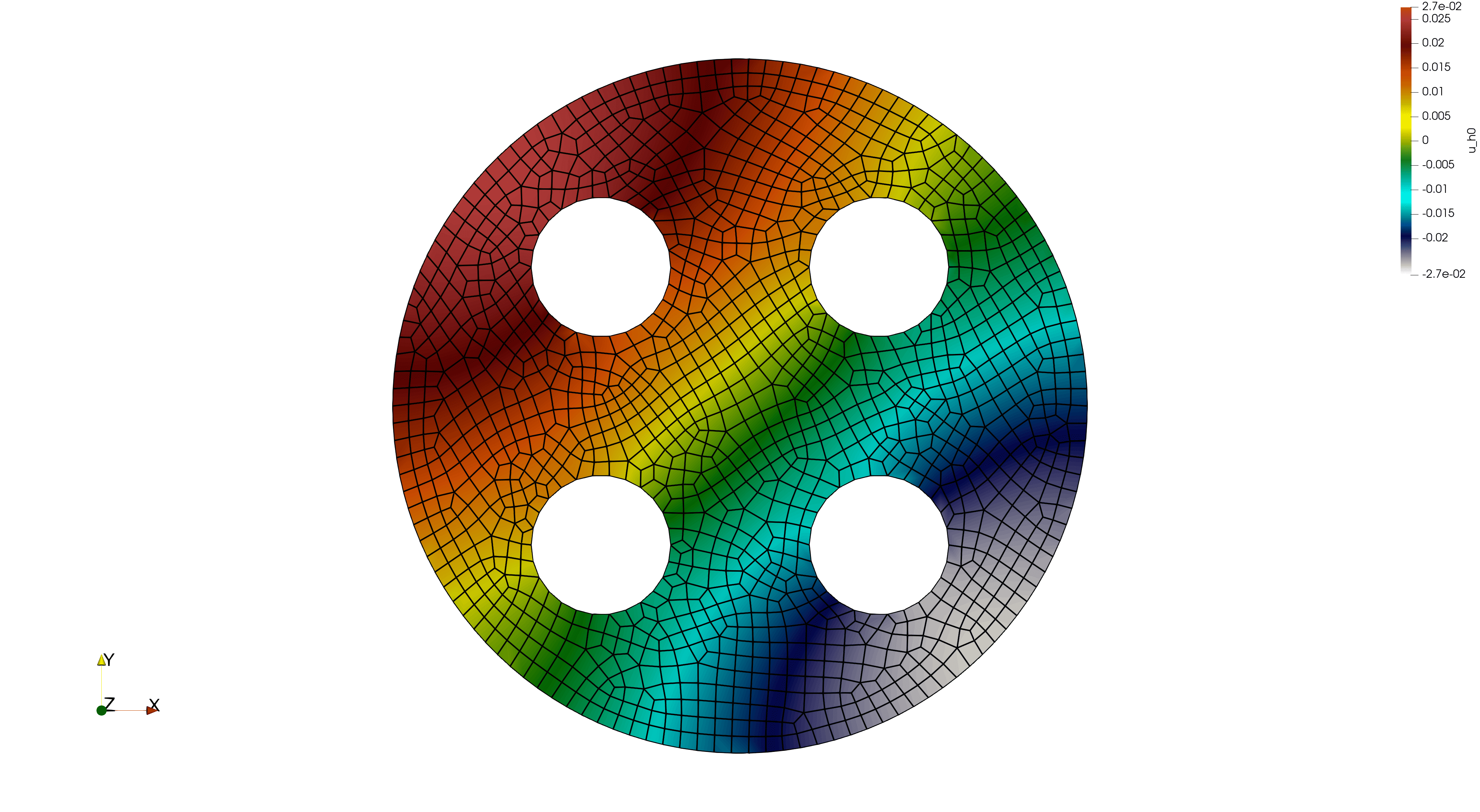}\\
 		    {\footnotesize $\bu_{1,h}$}
  \end{minipage}
  \begin{minipage}{0.32\linewidth}\centering
    \includegraphics[scale=0.064, trim=37cm 4cm 37cm 4cm,clip]{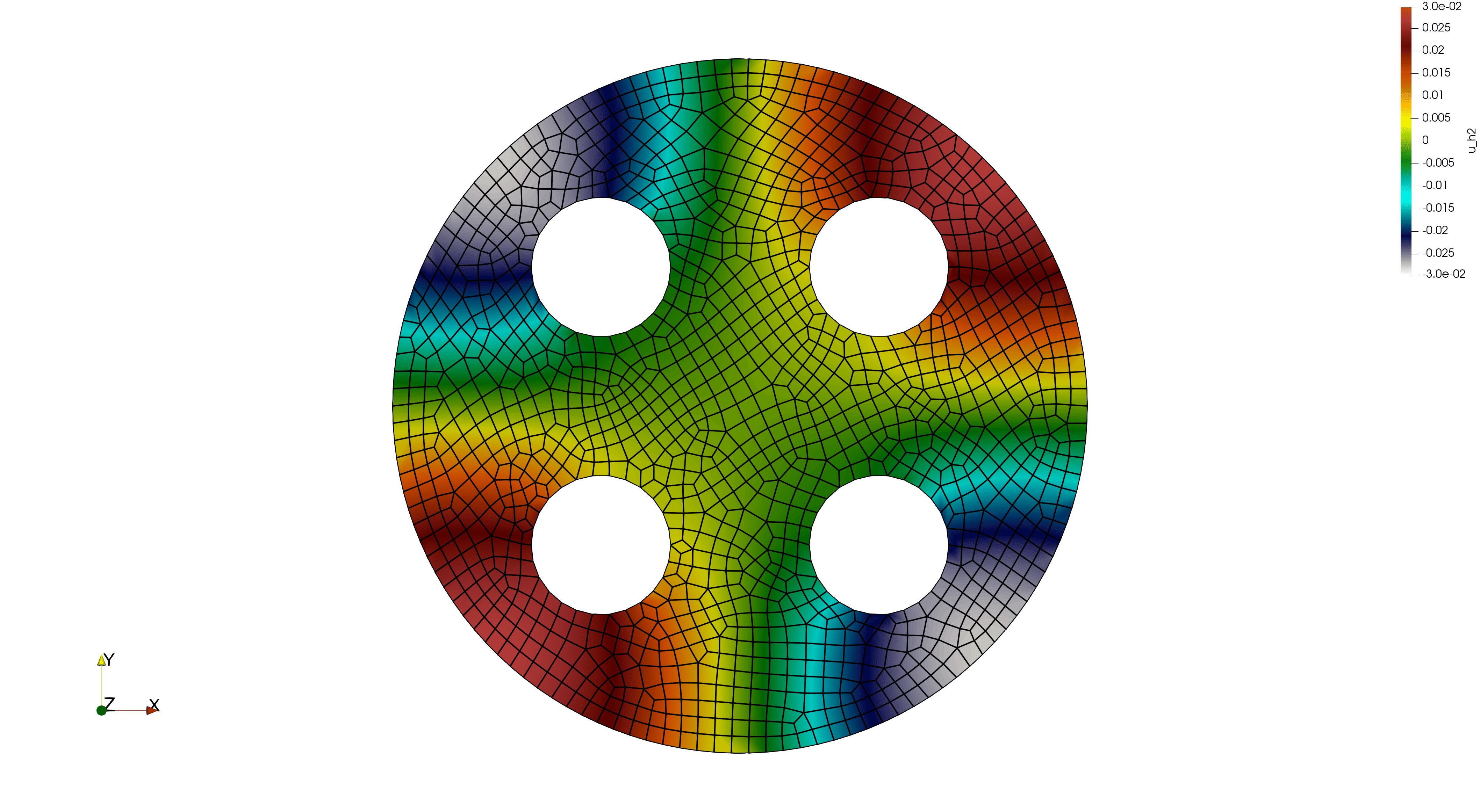}\\
 		    {\footnotesize $\bu_{3,h}$}
  \end{minipage}
  \begin{minipage}{0.32\linewidth}\centering
    \includegraphics[scale=0.064, trim=37cm 4cm 37cm 4cm,clip]{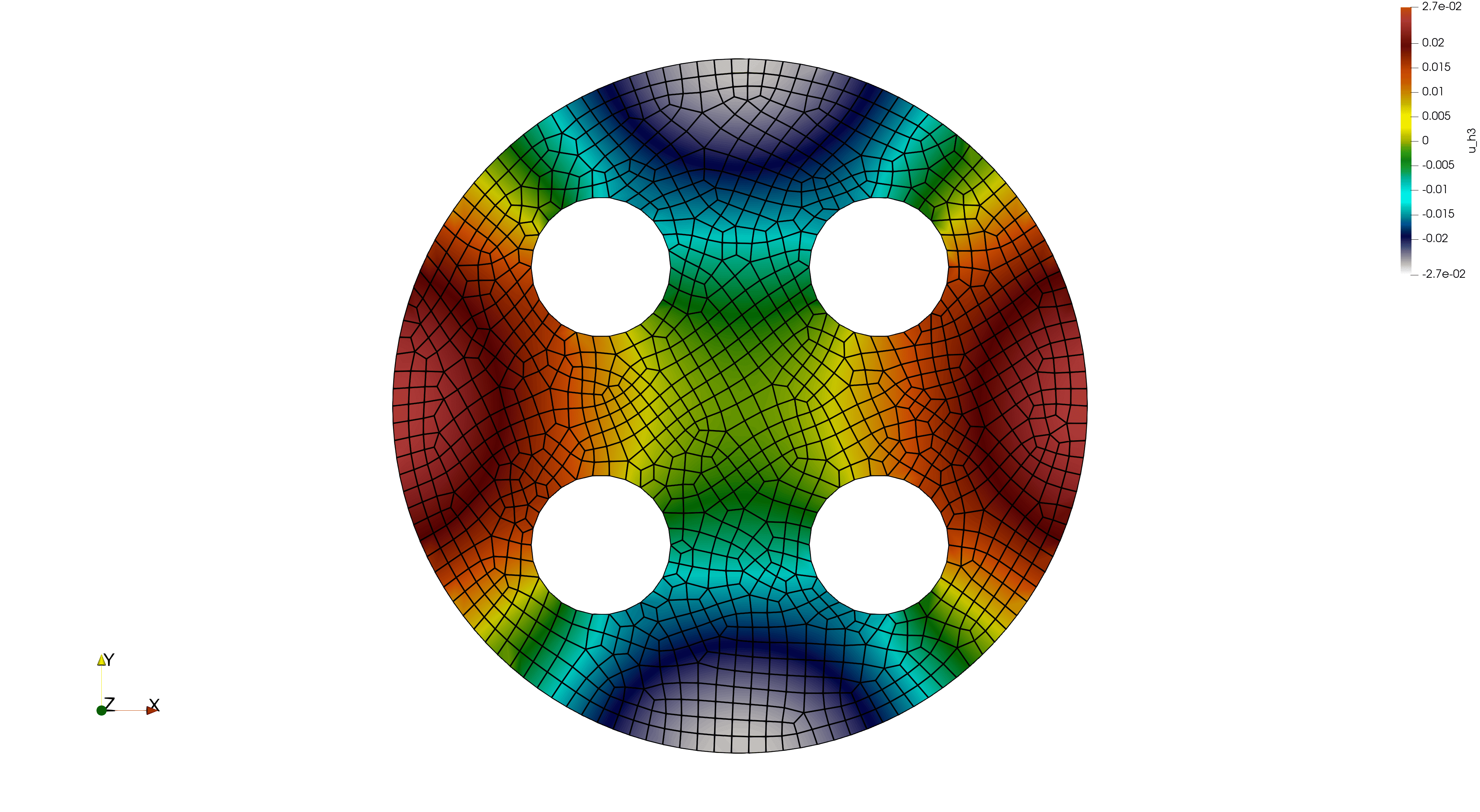}\\
 		    {\footnotesize $\bu_{4,h}$}
  \end{minipage}
  \hspace*{0.1cm}\begin{minipage}{0.32\linewidth}\centering
    \includegraphics[scale=0.13, trim=56cm 2cm 56cm 2cm,clip]{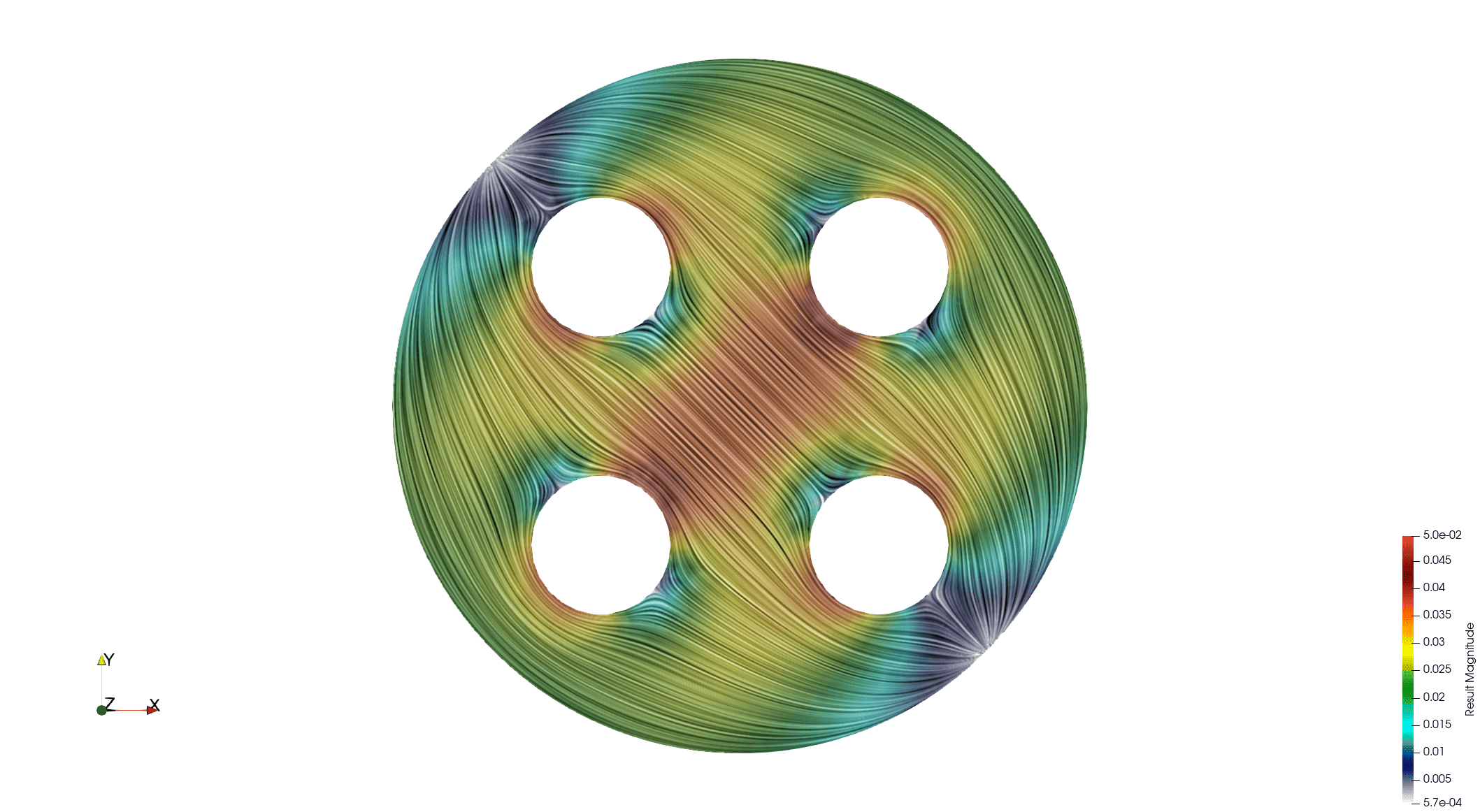}\\
 		    {\footnotesize $\nabla\bu_{1,h}$}
  \end{minipage}
  \begin{minipage}{0.32\linewidth}\centering
    \includegraphics[scale=0.13, trim=56cm 2cm 56cm 2cm,clip]{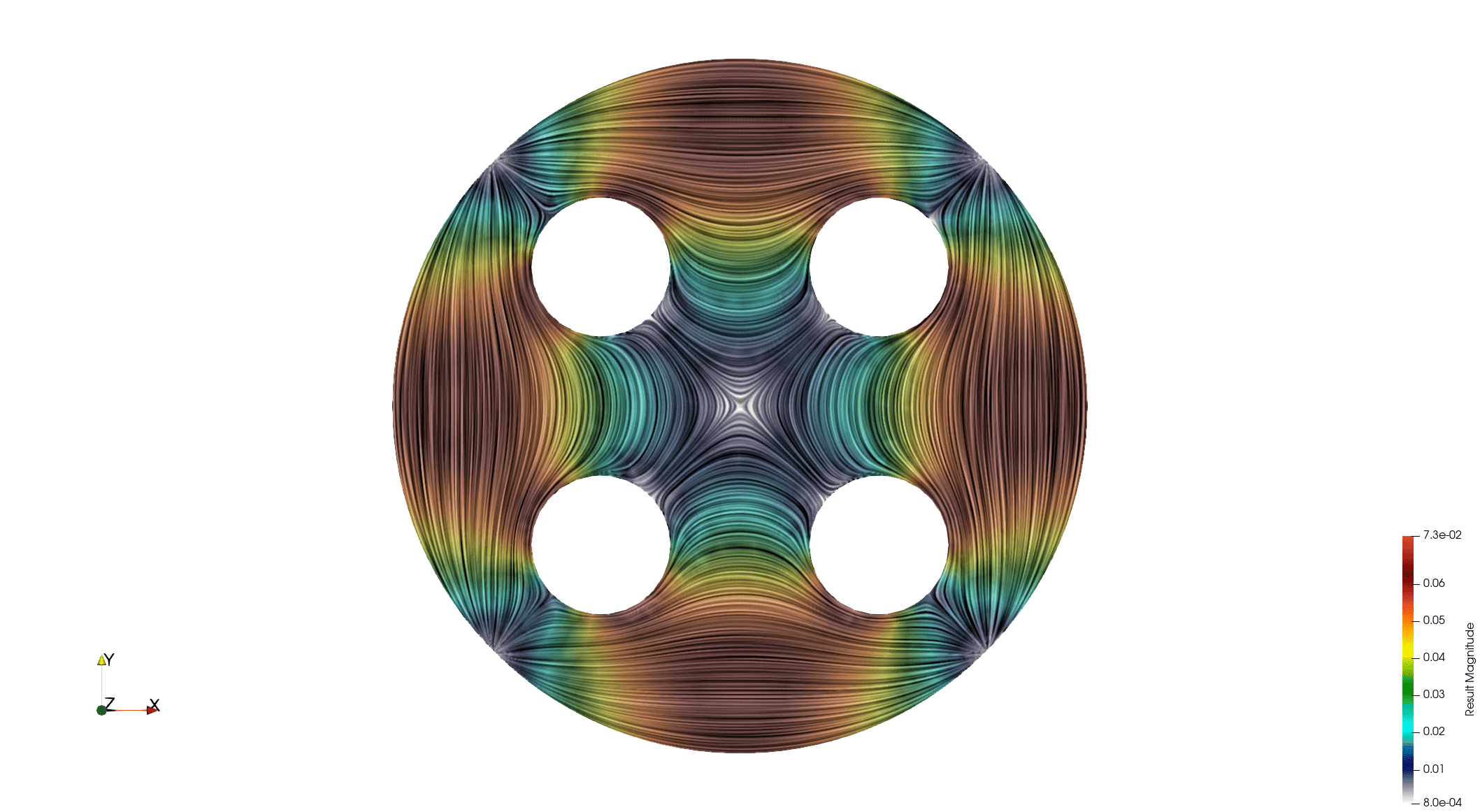}\\
 		    {\footnotesize $\nabla\bu_{3,h}$}
  \end{minipage}
  \begin{minipage}{0.32\linewidth}\centering
    \includegraphics[scale=0.13, trim=56cm 2cm 56cm 2cm,clip]{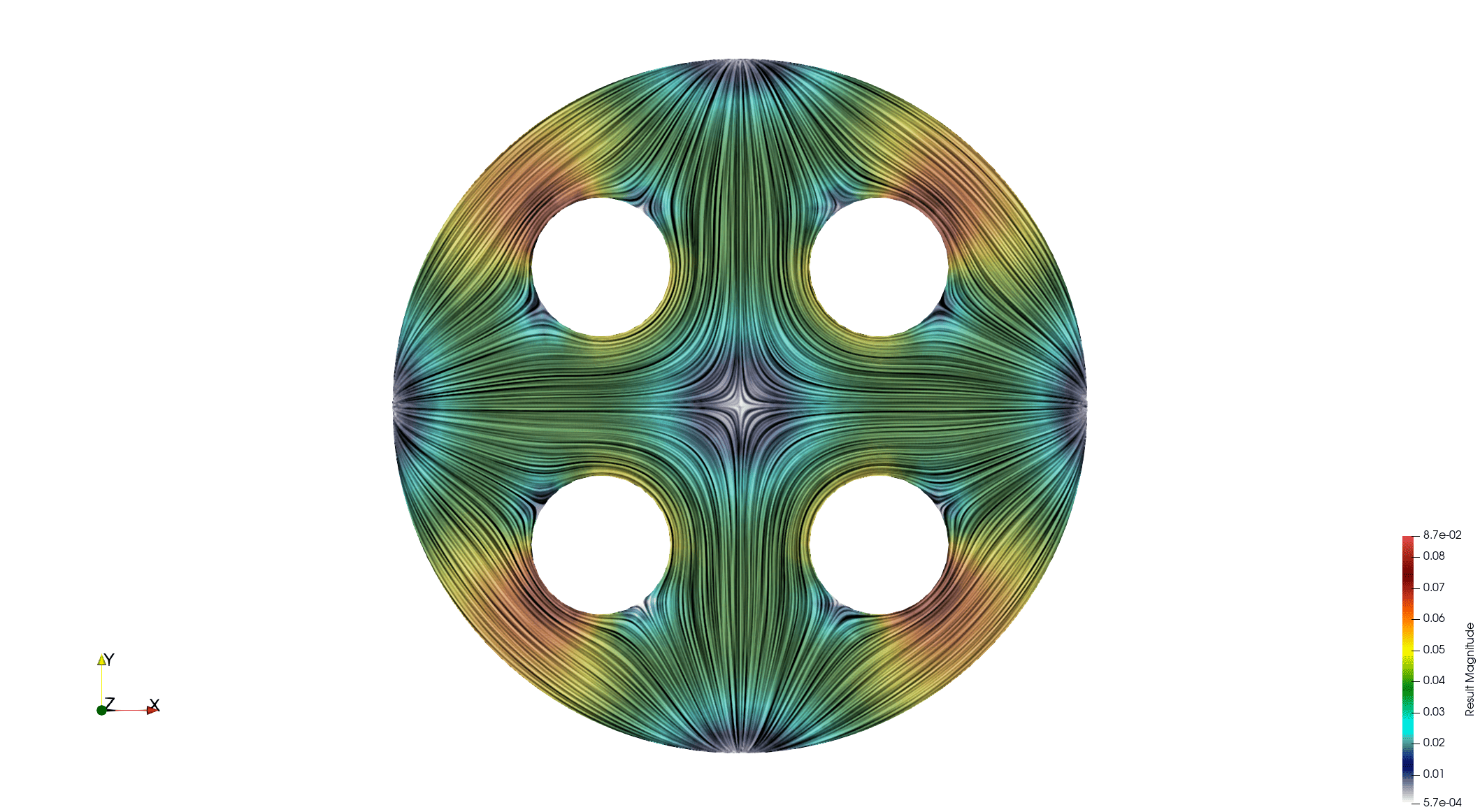}\\
 		    {\footnotesize $\nabla\bu_{4,h}$}
  \end{minipage}
  \caption{Example \ref{subsec:circle-holes}. Surface plots of selected computed eigenfunctions and line integral convolution plots of their reconstructed gradient fields. The results were obtained with the nonconforming method on the mesh $\mathcal{T}^{10}_h$ for $N=64$, with zero normal derivative imposed on $\partial\Omega_m \cup \bigcup_{i=1}^4 \partial\Omega_i$, $i=1,2,3,4$.}
  \label{fig:circle_hole2DNfree-uh-graduh}
 \end{figure}

\begin{figure}[!h]
  \centering
  \begin{minipage}{0.32\linewidth}\centering
    \includegraphics[scale=0.064, trim=37cm 4cm 37cm 4cm,clip]{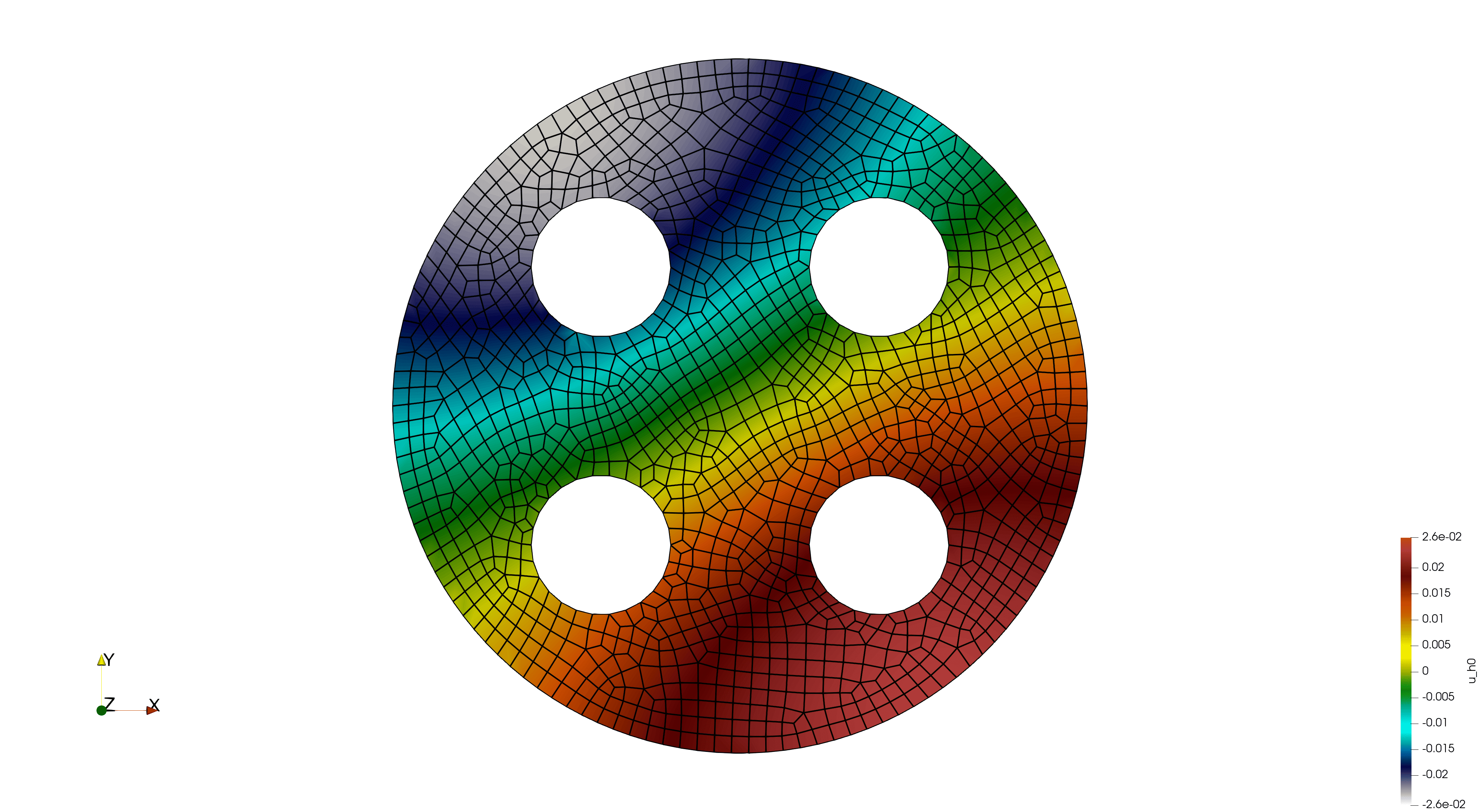}\\
		    {\footnotesize $\bu_{1,h}$}
  \end{minipage}
  \begin{minipage}{0.32\linewidth}\centering
    \includegraphics[scale=0.064, trim=37cm 4cm 37cm 4cm,clip]{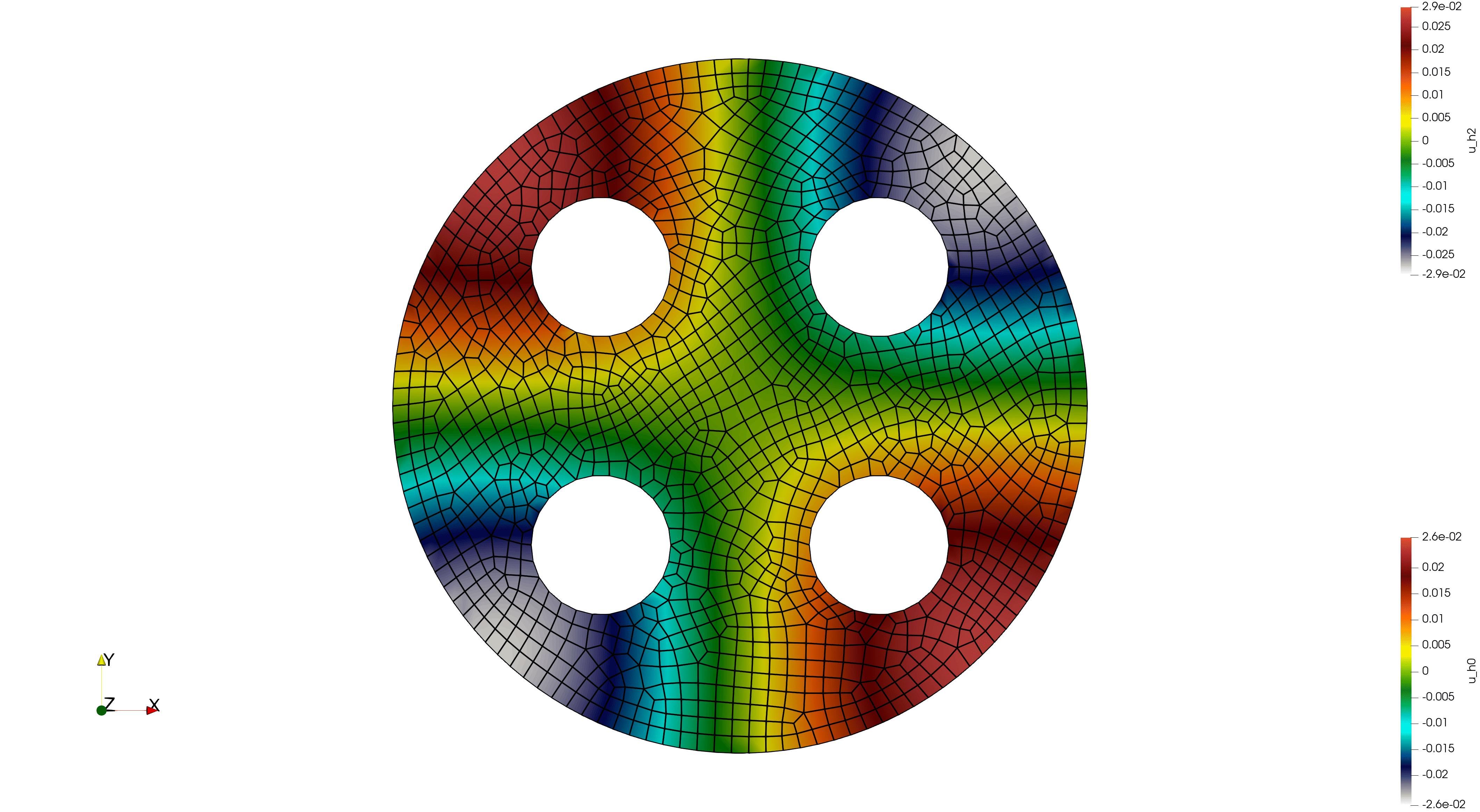}\\
		    {\footnotesize $\bu_{3,h}$}
  \end{minipage}
  \begin{minipage}{0.32\linewidth}\centering
    \includegraphics[scale=0.064, trim=37cm 4cm 37cm 4cm,clip]{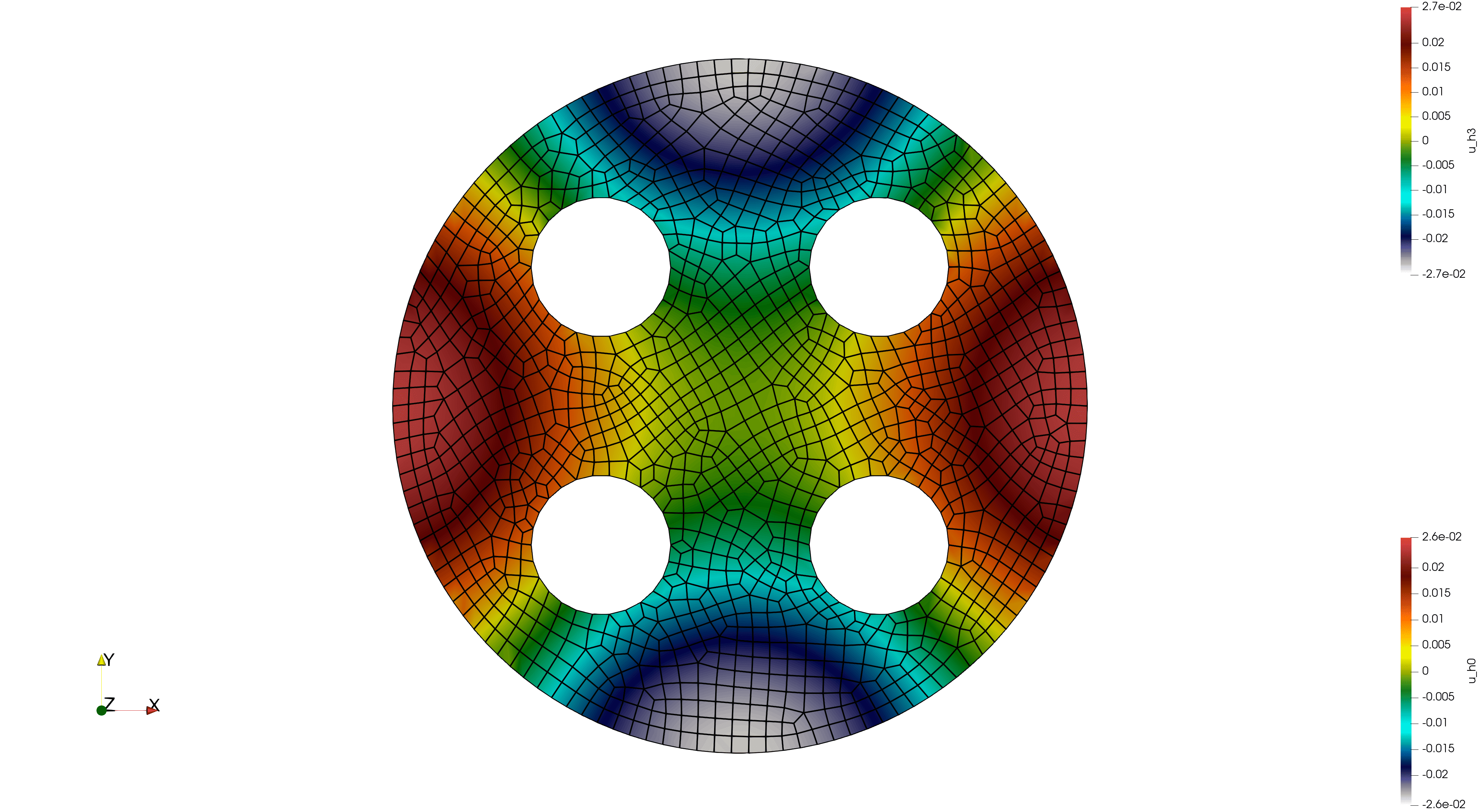}\\
		    {\footnotesize $\bu_{4,h}$}
  \end{minipage}
  \hspace*{0.1cm}\begin{minipage}{0.32\linewidth}\centering
    \includegraphics[scale=0.13, trim=56cm 2cm 56cm 2cm,clip]{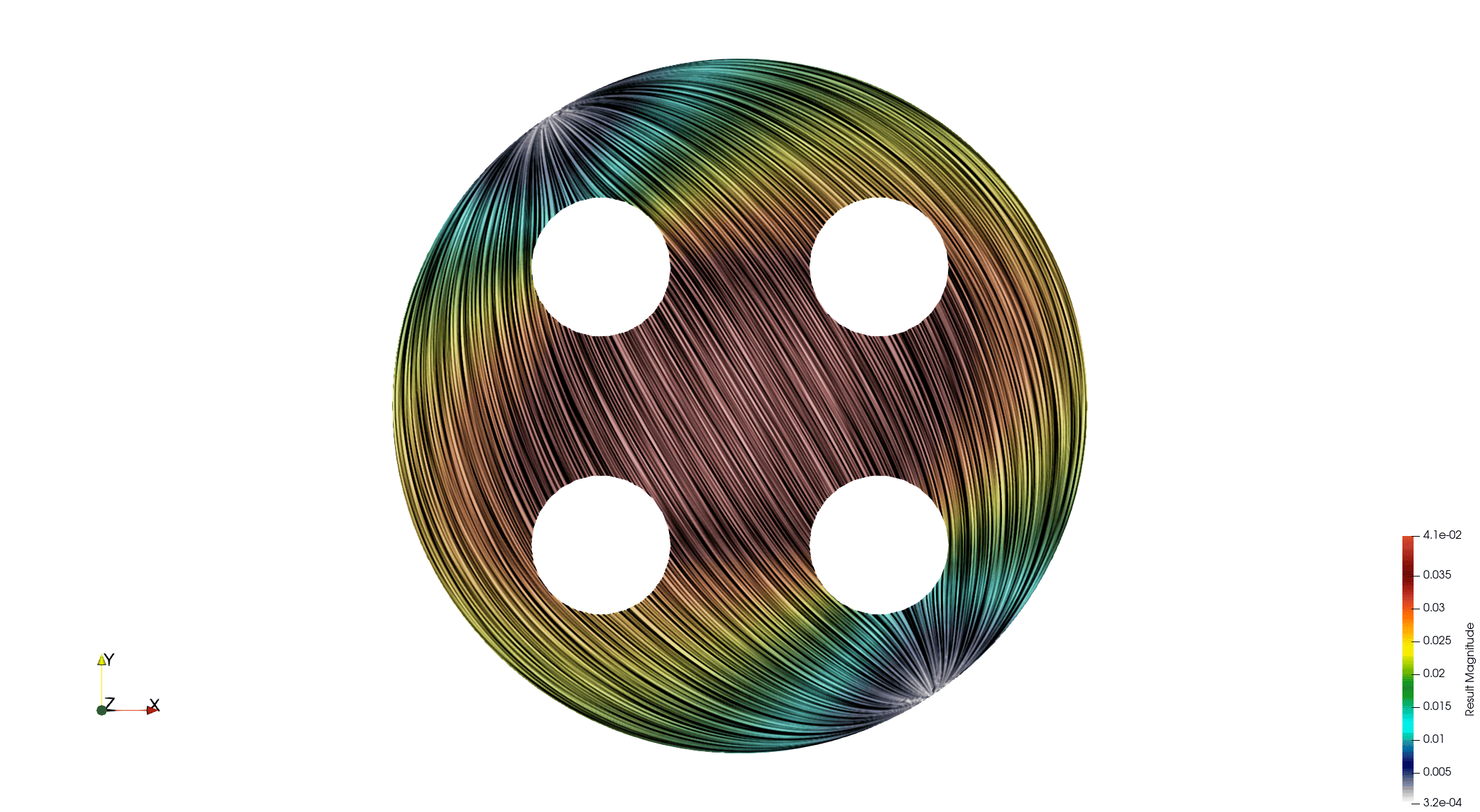}\\
		    {\footnotesize $\nabla\bu_{1,h}$}
  \end{minipage}
  \begin{minipage}{0.32\linewidth}\centering
    \includegraphics[scale=0.13, trim=56cm 2cm 56cm 2cm,clip]{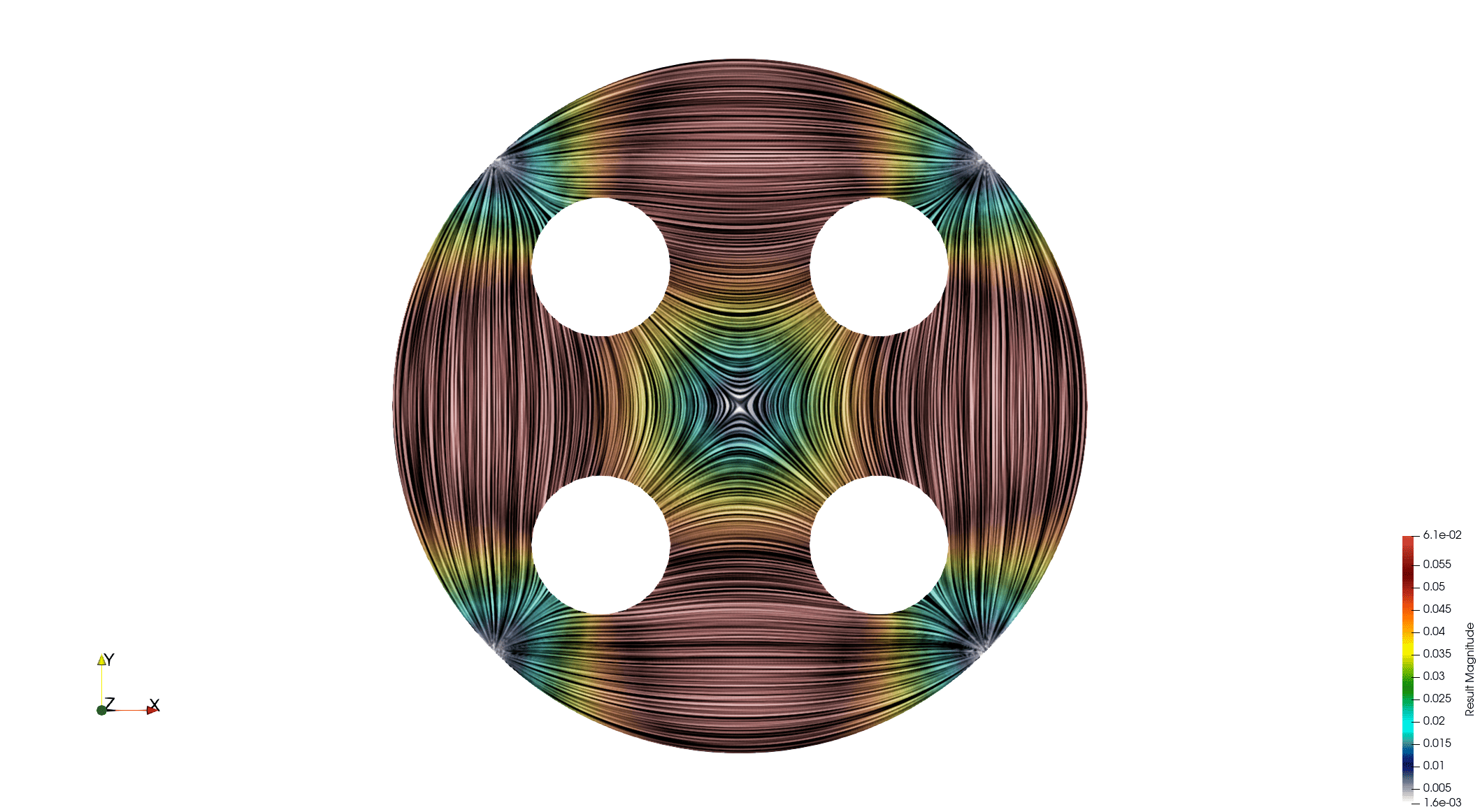}\\
		    {\footnotesize $\nabla\bu_{3,h}$}
  \end{minipage}
  \begin{minipage}{0.32\linewidth}\centering
    \includegraphics[scale=0.13, trim=56cm 2cm 56cm 2cm,clip]{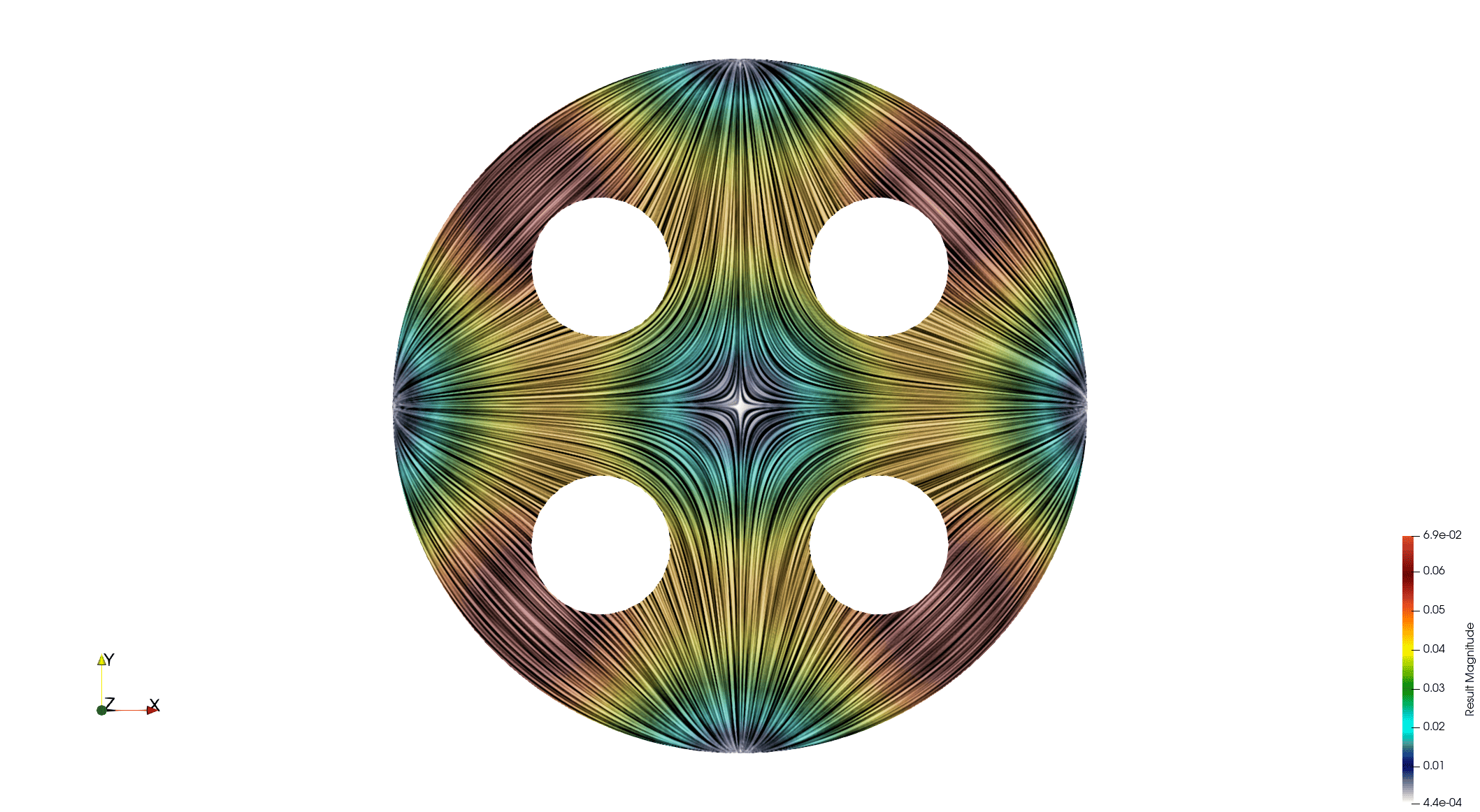}\\
		    {\footnotesize $\nabla\bu_{4,h}$}
  \end{minipage}
  \caption{Test~\ref{subsec:circle-holes}. Surface plot and gradients
    surface line integral convolution of the first, third, and fourth
    lowest eigenfunctions computed with non-conforming method on the
    mesh $\CT^{10}_h$ for $N=64$ and zero normal derivative on
    $\partial\Omega_m$.}
  \label{fig:circle_hole2Dfree-uh-graduh}
\end{figure}

\section{Conclusion} \label{conclusion}
We comprehensively analyzed the $H^2$-conforming and
$C^0$ non-conforming virtual element methods for the biharmonic Steklov
eigenvalue problem.

We highlight that the 
boundary integration in the variational formulation is well-defined also for functions in the $C^0$-non-conforming space. 
Since both the discrete conforming/non-conforming spaces $\CV_h^{\dag}
\subset H^1(\Omega)$, for $\dag \in \{ c,nc\}$, the discrete solution
operator $T_h$ is compact and well-defined.
By employing Babu\v{s}ka-Osborn spectral theory \cite{BO}, we proved
that discrete eigenvalues approximate analytical eigenvalues optimally
(double order of convergence).
However, we stress that the fully non-conforming Morley VEM might not
be a proper choice for the approximation since Morley VEM has global
$L^2(\Omega)$ regularity instead of $H^1(\Omega)$.

The theoretical analysis of the discrete problem
covers both conforming and non-conforming cases.
The numerical results confirm the theoretical expectations.
A potential area for future research could be the \emph{a posteriori}
analysis to this model problem.

\section*{Acknowledgments}
\noindent D. Adak has been partially supported by ANRF, New Delhi, India, through Project~RJF/2025/000114. \\
D. Boffi is a member of the Gruppo Nazionale Calcolo
Scientifico-Istituto Nazionale di Alta Mate\-ma\-tica (GNCS-INdAM) and his work was partially supported by KAUST - CRG13-2025 grant (6911).\\
F. Gardini is a member of the Gruppo Nazionale Calcolo
Scientifico-Istituto Nazionale di Alta Matematica (GNCS-INdAM).   \\
G. Manzini is a member of the Gruppo Nazionale Calcolo
Scientifico-Istituto Nazionale di Alta Matematica (GNCS-INdAM) and this 
work has been partially supported by 
INdAM - GNCS Project CUP\_E53C23001670001.

\bibliographystyle{siam}
\bibliography{bib_LR2}

@article{geuzaine2009gmsh,
author = {Geuzaine, C. and Remacle, J.},
title = {{Gmsh}: A {3-D} finite element mesh generator with built-in pre- and post-processing facilities},
journal = {International Journal for Numerical Methods in Engineering},
volume = {79},
number = {11},
pages = {1309-1331},
doi = {https://doi.org/10.1002/nme.2579},
url = {https://onlinelibrary.wiley.com/doi/abs/10.1002/nme.2579},
eprint = {https://onlinelibrary.wiley.com/doi/pdf/10.1002/nme.2579},
year = {2009}
}

@article{polymeshertalischi2012,
	title={PolyMesher: a general-purpose mesh generator for polygonal elements written in {M}atlab},
	author={Talischi, C. and Paulino, G. H. and Pereira, A. and Menezes, I.F.M.},
	journal={Structural and Multidisciplinary Optimization},
	volume={45},
	pages={309--328},
	year={2012},
	publisher={Springer}
}

@article{chinosi2016virtual,
  title={Virtual element method for fourth order problems: ${L}^2$-estimates},
  author={Chinosi, Claudia and Marini, L Donatella},
  journal={Computers \& Mathematics with Applications},
  volume={72},
  number={8},
  pages={1959--1967},
  year={2016},
  publisher={Elsevier}
}

@incollection {BO,
    AUTHOR = {Babu\v{s}ka, I. and Osborn, J.},
     TITLE = {Eigenvalue problems},
 BOOKTITLE = {Handbook of numerical analysis, {V}ol. {II}},
    SERIES = {Handb. Numer. Anal., II},
     PAGES = {641--787},
 PUBLISHER = {North-Holland, Amsterdam},
      YEAR = {1991},
   MRCLASS = {65-02 (65N25 65N30)},
  MRNUMBER = {1115240},
}

@article{monzon2019virtual,
  AUTHOR={Monz{\'o}n, G.},
  TITLE={A virtual element method for a biharmonic {S}teklov eigenvalue problem},
  JOURNAL={Adv. Pure Appl. Math. },
  FJOURNAL={Advances in Pure and Applied Mathematics},
  VOLUME={10},
  YEAR={2019},
  NUMBER={4},
  PAGES={325--337},
  publisher={De Gruyter}
}

@article{chinosi2016,
  AUTHOR={Chinosi, C. and Marini, L. D.},
  TITLE={Virtual element method for fourth order problems: ${L}^2$-estimates},
  JOURNAL={Comput. Math. with Appl. },
  FJOURNAL={Computers \& Mathematics with Applications},
  VOLUME={72},
  YEAR={2016},
  NUMBER={8},
  PAGES={1959--1967},
  publisher={Elsevier}
}

@article {MR3002804,
	AUTHOR = {Brezzi, Franco and Marini, L. Donatella},
	TITLE = {Virtual element methods for plate bending problems},
	JOURNAL = {Comput. Methods Appl. Mech. Engrg.},
	FJOURNAL = {Computer Methods in Applied Mechanics and Engineering},
	VOLUME = {253},
	YEAR = {2013},
	PAGES = {455--462},
	ISSN = {0045-7825},
	MRCLASS = {74G60 (65N30 74K20 74S05)},
	MRNUMBER = {3002804},
	DOI = {10.1016/j.cma.2012.09.012},
	URL = {https://ezproxyucor.unicordoba.edu.co:2129/10.1016/j.cma.2012.09.012},
}

@article {BasicVEM2013,
	AUTHOR = {Beir\~{a}o da Veiga, L. and Brezzi, F. and Cangiani, A. and
	Manzini, G. and Marini, L. D. and Russo, A.},
	TITLE = {Basic principles of virtual element methods},
	JOURNAL = {Math. Models Methods Appl. Sci.},
	FJOURNAL = {Mathematical Models and Methods in Applied Sciences},
	VOLUME = {23},
	YEAR = {2013},
	NUMBER = {1},
	PAGES = {199--214},
	ISSN = {0218-2025},
	MRCLASS = {65N06},
	MRNUMBER = {2997471},
	MRREVIEWER = {Bo\v{s}ko S. Jovanovi\'{c}},
	DOI = {10.1142/S0218202512500492},
	URL = {https://ezproxyucor.unicordoba.edu.co:2129/10.1142/S0218202512500492},
}

@article {MR3529253,
	AUTHOR = {Zhao, Jikun and Chen, Shaochun and Zhang, Bei},
	TITLE = {The nonconforming virtual element method for plate bending
	problems},
	JOURNAL = {Math. Models Methods Appl. Sci.},
	FJOURNAL = {Mathematical Models and Methods in Applied Sciences},
	VOLUME = {26},
	YEAR = {2016},
	NUMBER = {9},
	PAGES = {1671--1687},
	ISSN = {0218-2025},
	MRCLASS = {65N30 (65N12 74K20)},
	MRNUMBER = {3529253},
	DOI = {10.1142/S021820251650041X},
	URL = {https://ezproxyucor.unicordoba.edu.co:2129/10.1142/S021820251650041X},
}

@article {MR3460621,
	AUTHOR = {Beir\~{a}o da Veiga, L. and Brezzi, F. and Marini, L. D. and
	Russo, A.},
	TITLE = {Virtual element method for general second-order elliptic
	problems on polygonal meshes},
	JOURNAL = {Math. Models Methods Appl. Sci.},
	FJOURNAL = {Mathematical Models and Methods in Applied Sciences},
	VOLUME = {26},
	YEAR = {2016},
	NUMBER = {4},
	PAGES = {729--750},
	ISSN = {0218-2025},
	MRCLASS = {65N30 (35J25)},
	MRNUMBER = {3460621},
	DOI = {10.1142/S0218202516500160},
	URL = {https://ezproxyucor.unicordoba.edu.co:2129/10.1142/S0218202516500160},
}

@article {MR3875292,
	AUTHOR = {Mora, David and Rivera, Gonzalo and Vel\'{a}squez, Iv\'{a}n},
	TITLE = {A virtual element method for the vibration problem of
	{K}irchhoff plates},
	JOURNAL = {ESAIM Math. Model. Numer. Anal.},
	FJOURNAL = {ESAIM. Mathematical Modelling and Numerical Analysis},
	VOLUME = {52},
	YEAR = {2018},
	NUMBER = {4},
	PAGES = {1437--1456},
	ISSN = {0764-583X},
	MRCLASS = {65N25 (65N30 74H45 74K20)},
	MRNUMBER = {3875292},
	MRREVIEWER = {Jiguang Sun},
	DOI = {10.1051/m2an/2017041},
	URL = {https://ezproxyucor.unicordoba.edu.co:2129/10.1051/m2an/2017041},
}

@article {MR4019985,
	AUTHOR = {Beir\~{a}o da Veiga, L. and Mora, D. and Vacca, G.},
	TITLE = {The {S}tokes complex for virtual elements with application to
	{N}avier--{S}tokes flows},
	JOURNAL = {J. Sci. Comput.},
	FJOURNAL = {Journal of Scientific Computing},
	VOLUME = {81},
	YEAR = {2019},
	NUMBER = {2},
	PAGES = {990--1018},
	ISSN = {0885-7474},
	MRCLASS = {65N35 (35Q30)},
	MRNUMBER = {4019985},
	DOI = {10.1007/s10915-019-01049-3},
	URL = {https://ezproxyucor.unicordoba.edu.co:2129/10.1007/s10915-019-01049-3},
}

@article{mora2020virtual,
AUTHOR={Mora, David and Vel{\'a}squez, Iv{\'a}n},
TITLE={Virtual element for the buckling problem of {K}irchhoff--{L}ove plates},
JOURNAL={Comput. Methods Appl. Mech. Engrg.},
FJOURNAL={Computer Methods in Applied Mechanics and Engineering},
VOLUME={360},
YEAR={2020},
MRNUMBER={112687},
}

@article{adak2022convergence,
AUTHOR={Adak, D and Mora, D and Natarajan, Sundararajan},
TITLE={Convergence Analysis of Virtual Element Method for Nonlinear Nonlocal Dynamic Plate Equation},
JOURNAL={J. Sci. Comput.},
FJOURNAL = {Journal of Scientific Computing},
VOLUME={91},
NUMBER={1},
YEAR={2022},
PAGES={1--37},
}

@article{cangiani2016nonconforming,
AUTHOR={Cangiani, Andrea and Gyrya, Vitaliy and Manzini, Gianmarco},
TITLE={The nonconforming virtual element method for the {S}tokes equations},
JOURNAL={SIAM J. Numer. Anal.},
FJOURNAL={SIAM Journal on Numerical Analysis},
VOLUME={54},
NUMBER={6},
YEAR={2016},
PAGES={3411--3435},
}

@article{zhao2018morley,
AUTHOR={Zhao, Jikun and Zhang, Bei and Chen, Shaochun and Mao, Shipeng},
TITLE={The {M}orley-type virtual element for plate bending problems},
JOURNAL={J. Sci. Comput.},
FJOURNAL={Journal of Scientific Computing},
VOLUME={76},
NUMBER={1},
YEAR={2018},
PAGES={610--629},
}

@article{zhao2016nonconforming,
AUTHOR={Zhao, Jikun and Chen, Shaochun and Zhang, Bei},
TITLE={The nonconforming virtual element method for plate bending problems},
JOURNAL={Math. Models Methods Appl. Sci.},
FJOURNAL={Mathematical Models and Methods in Applied Sciences},
VOLUME={26},
NUMBER={09},
YEAR={2016},
PAGES={1671--1687},
}

@article{antonietti2018fully,
AUTHOR={Antonietti, Paola F and Manzini, Gianmarco and Verani, Marco},
TITLE={The fully nonconforming virtual element method for biharmonic problems},
JOURNAL={Math. Models Methods Appl. Sci.},
FJOURNAL={Mathematical Models and Methods in Applied Sciences},
VOLUME={28},
NUMBER={02},
YEAR={2018},
PAGES={387--407},
}

@article{huang2021medius,
AUTHOR={Huang, Jianguo and Yu, Yue},
TITLE={A medius error analysis for nonconforming virtual element methods for {P}oisson and biharmonic equations},
JOURNAL={J. Comput. Appl. Math.},
FJOURNAL={Journal of Computational and Applied Mathematics},
VOLUME={386},
MRNUMBER={113229},
YEAR={2021},
}

@article{bermudez2000finite,
  AUTHOR={Berm{\'u}dez, Alfredo and Rodr{\'\i}guez, Rodolfo and Santamarina, Duarte},
  TITLE={A finite element solution of an added mass formulation for coupled fluid-solid vibrations},
  JOURNAL={Numer. Math.},
  VOLUME={87},
  PAGES={201--227},
  YEAR={2000}
}

@article{bermudez2003finite,
AUTHOR={Berm{\'u}dez, Alfredo and Rodr{\'\i}guez, Rodolfo and Santamarina, Duarte},
  TITLE={Finite element computation of sloshing modes in containers with elastic baffle plates},
  JOURNAL={Int. J.  Numer. Methods Eng.},
  VOLUME={56},
  NUMBER={3},
  PAGES={447--467},
  YEAR={2003}
  }

@article{kuttler1972remarks,
  AUTHOR={Kuttler, James R},
  TITLE={Remarks on a {S}tekloff eigenvalue problem},
  JOURNAL={SIAM J. NumeR. Anal.},
  VOLUME={9},
  NUMBER={1},
  PAGES={1--5},
  YEAR={1972},
}

@article{bucur2009first,
AUTHOR={Bucur, Dorin and Ferrero, Alberto and Gazzola, Filippo},
  TITLE={On the first eigenvalue of a fourth order {S}teklov problem},
  JOURNAL={Calculus of Variations and Partial Differential Equations},
  VOLUME={35},
  PAGES={103--131},
  YEAR={2009}
}

@article{ferrero2005fourth,
AUTHOR={Ferrero, Alberto and Gazzola, Filippo and Weth, Tobias},
  TITLE={On a fourth order {S}teklov eigenvalue problem},
  PAGES={315-332},
  YEAR={2005},
}

@article{wang2009sharp,
AUTHOR={Wang, Qiaoling and Xia, Changyu},
  TITLE={Sharp bounds for the first non-zero Stekloff eigenvalues},
  JOURNAL={Journal of Functional Analysis},
  VOLUME={257},
  NUMBER={8},
  PAGES={2635--2644},
  YEAR={2009}
}

@article{bi2011conforming,
AUTHOR={Bi, Hai and Ren, Shixian and Yang, Yidu},
  TITLE={Conforming finite element approximations for a fourth-order {S}teklov eigenvalue problem},
  JOURNAL={Mathematical Problems in Engineering},
 VOLUME={2011},
  YEAR={2011},
}

@article{wang2022priori,
AUTHOR={Wang, Gang and Meng, Jian and Wang, Ying and Mei, Liquan},
  TITLE={A priori and a posteriori error estimates for a virtual element method for the non-self-adjoint Steklov eigenvalue problem},
  JOURNAL={IMA Journal of Numerical Analysis},
  VOLUME={42},
  NUMBER={4},
  PAGES={3675--3710},
  YEAR={2022}
}

@article{yang2020non,
AUTHOR={Yang, Yidu and Zhang, Yu and Bi, Hai},
  TITLE={Non-conforming Crouzeix-Raviart element approximation for Stekloff eigenvalues in inverse scattering},
  JOURNAL={Advances in Computational Mathematics},
  VOLUME={46},
  PAGES={1--25},
  YEAR={2020},
}

@article{zhang2020multigrid,
 AUTHOR={Zhang, Yu and Bi, Hai and Yang, Yidu},
  TITLE={A multigrid correction scheme for a new Steklov eigenvalue problem in inverse scattering},
  JOURNAL={International Journal of Computer Mathematics},
  VOLUME={97},
  NUMBER={7},
  PAGES={1412--1430},
  YEAR={2020}
}

@article{dedner2021higher,
AUTHOR={Dedner, Andreas and Hodson, Alice},
  TITLE={A higher order nonconforming virtual element method for the {C}ahn-{H}illiard equation},
  JOURNAL={arXiv preprint arXiv:2111.11408},
  YEAR={2021}
}

@article{antonietti2016c,
AUTHOR={Antonietti, Paola F and Da Veiga, L Beirao and Scacchi, Simone and Verani, Marco},
  TITLE={A ${C}^1$ virtual element method for the {C}ahn--{H}illiard equation with polygonal meshes},
  JOURNAL={SIAM Journal on Numerical Analysis},
  VOLUME={54},
  NUMBER={1},
  PAGES={34--56},
  YEAR={2016},
}

@article{adak2023c0,
AUTHOR={Adak, Dibyendu and Mora, David and Vel{\'a}squez, Iv{\'a}n},
  TITLE={A ${C}^0$-nonconforming virtual element methods for the vibration and buckling problems of thin plates},
  JOURNAL={Computer Methods in Applied Mechanics and Engineering},
  VOLUME={403},
 PAGES={115763},
  YEAR={2023},
}

@article{adak2023morley,
AUTHOR={Adak, D and Mora, D and Silgado, A},
  TITLE={A {M}orley-type virtual element approximation for a wind-driven ocean circulation model on polygonal meshes},
  JOURNAL={Journal of Computational and Applied Mathematics},
  VOLUME={425},
  PAGES={115026},
  YEAR={2023},
}

@article{adak2021virtual,
AUTHOR={Adak, Dibyendu and Mora, David and Natarajan, Sundararajan and Silgado, Alberth},
  TITLE={A virtual element discretization for the time dependent {N}avier--{S}tokes equations in stream-function formulation},
  JOURNAL={ESAIM: Mathematical Modelling and Numerical Analysis},
  VOLUME={55},
  NUMBER={5},
  PAGES={2535--2566},
  YEAR={2021},
}

@article{zhang2020nonconforming,
AUTHOR={Zhang, Bei and Zhao, Jikun and Chen, Shaochun},
  TITLE={The nonconforming virtual element method for fourth-order singular perturbation problem},
  JOURNAL={Advances in Computational Mathematics},
  VOLUME={46},
  PAGES={1--23},
  YEAR={2020},
}

@incollection{mora2022virtual,
AUTHOR={Mora, David and Silgado, Alberth},
  TITLE={Virtual Element Methods for a Stream-Function Formulation of the {O}seen Equations},
  BOOKTITLE={The Virtual Element Method and its Applications},
  PAGES={321--361},
  YEAR={2022},
}

@article{da2023fully,
 AUTHOR={da Veiga, L Beir{\~a}o and Mora, D and Silgado, A},
  TITLE={A fully-discrete virtual element method for the nonstationary {B}oussinesq equations in stream-function form},
  JOURNAL={Computer Methods in Applied Mechanics and Engineering},
  VOLUME={408},
  PAGES={115947},
  YEAR={2023},
}

@article{de2016nonconforming,
AUTHOR={de Dios, Blanca Ayuso and Lipnikov, Konstantin and Manzini, Gianmarco},
  TITLE={The nonconforming virtual element method},
  JOURNAL={ESAIM: Mathematical Modelling and Numerical Analysis},
  VOLUME={50},
  NUMBER={3},
  PAGES={879--904},
  YEAR={2016},
}

@article{liu2019nonconforming,
AUTHOR={Liu, Xin and Chen, Zhangxin},
  TITLE={The nonconforming virtual element method for the Navier-Stokes equations},
  JOURNAL={Advances in Computational Mathematics},
  VOLUME={45},
  PAGES={51--74},
  YEAR={2019}
}

@article{zhang2021divergence,
AUTHOR={Zhang, Bei and Zhao, Jikun and Li, Meng},
  TITLE={The divergence-free nonconforming virtual element method for the Navier--Stokes problem},
  JOURNAL={Numerical Methods for Partial Differential Equations},
  YEAR={2021}
}

@article{zhao2019divergence,
AUTHOR={Zhao, Jikun and Zhang, Bei and Mao, Shipeng and Chen, Shaochun},
  TITLE={The divergence-free nonconforming virtual element for the Stokes problem},
  JOURNAL={SIAM Journal on Numerical Analysis},
  VOLUME={57},
  NUMBER={6},
  PAGES={2730--2759},
  YEAR={2019},
}

@book{G,
	author    = {Grisvard, Pierre},
	title     = {Elliptic Problems in Non-Smooth Domains},
	publisher = {Pitman, Boston},
	year      = {1985},
}

@book{adams2003sobolev,
	title={Sobolev Spaces},
	author={Adams, Robert A and Fournier, John JF},
	year={2003},
	publisher={Elsevier}
}

@article{adak2023conforming,
  title={Conforming and nonconforming virtual element methods for fourth order nonlocal reaction diffusion equation},
  author={Adak, Dibyendu and Anaya, Veronica and Bendahmane, Mostafa and Mora, David},
  journal={Mathematical Models and Methods in Applied Sciences},
  year={2023},
  publisher={World Scientific}
}

@article{brenner2013morley,
  title={A {M}orley finite element method for the displacement obstacle problem of clamped {K}irchoff plates},
  author={Brenner, Susanne C and Sung, Li-yeng and Zhang, Hongchao and Zhang, Yi},
  journal={Journal of Computational and Applied Mathematics},
  volume={254},
  pages={31--42},
  year={2013},
  publisher={Elsevier}
}

@book{chen2014bridge,
   author    = {Chen, W.-F. and Duan, L.},
   title     = {Bridge Engineering Handbook},
   edition   = {2nd},
   publisher = {CRC Press},
   year      = {2014}
 }

@book{naeim1999design,
   author    = {Naeim, F. and Kelly, J. M.},
   title     = {Design of Seismic Isolated Structures: 
                From Theory to Practice},
   publisher = {John Wiley \& Sons},
   year      = {1999}
 }

@book{senturia2001microsystem,
   author    = {Senturia, S. D.},
   title     = {Microsystem Design},
   publisher = {Kluwer Academic Publishers},
   year      = {2001}
 }

@misc{Gilardi-dispensa,
author = {G. Gilardi},
title = {Alcuni risultati sugli spazi di {H}ilbert},
howpublished = {\url{https://mate.unipv.it/gilardi/WEBGG/PSPDF/hilbert.pdf}},
note = {Accessed: March 25, 2026},
}

\appendix
\renewcommand{\thesection}{\Alph{section}}
\section{Coercivity of the bilinear form $\widehat{a}(\cdot,\cdot)$}
\label{sec:norms}

In this appendix, we prove the coercivity of the bilinear form $\widehat{a}(\cdot,\cdot)$~\eqref{eq:coercivity}.
This property follows from the next theorem,
which is, in fact, an
abstract version of the Friedrichs--Poincar\'{e} inequality on Sobolev
spaces.
A detailed exposition can be found in~\cite[Theorem 3.1]{Gilardi-dispensa}
\begin{theorem}
\label{thm:Friedrichs-Poincare}
Let $X$, $H$, $G$, and~$W$ be four Hilbert spaces, and let
$A \in \mathcal{L}(X,G)$ and $B \in \mathcal{L}(X,W)$ be two linear
and continuous operators.  Suppose that the following condition
\begin{align}
\label{eq:3.1}
& X \subseteq H \quad \text{with compact inclusion;} \\[4pt]
\label{eq:3.2}
& \|\cdot\|_{H} + \|A(\cdot)\|_{G}
  \text{ is a norm on } X \text{ equivalent to } \|\cdot\|_{X}; \\[4pt]
\label{eq:3.3}
& \text{if } \; v \in X, \quad Av = 0, \quad Bv = 0,
  \quad \text{then } \; v = 0.
\end{align}
Then there exists a constant~$C$ such that
\begin{equation}
\label{eq:3.4}
\|v\|_{H} \leq C\bigl(\|Av\|_{G} + \|Bv\|_{W}\bigr)
\quad \forall\, v \in X.
\end{equation}
\end{theorem}

\begin{proposition}
  \label{prop:coercivity-H1}
  Assume that $\Omega$ is a bounded Lipschitz domain. 
  Then, the
  coercivity inequality~\eqref{eq:coercivity} holds with the estimate:
  \begin{equation}
    \label{eq:FP-H1}
    \|v\|_{H^2(\Omega)}
    \leq C \bigl( \|D^2 v\|_{L^2(\Omega)} + \|v\|_{L^2(\partial\Omega)} \bigr)
    \qquad \forall\, v \in V.
  \end{equation}
\end{proposition}

\begin{proof}
  We apply Theorem~\ref{thm:Friedrichs-Poincare} with the
  identifications:
  \begin{alignat*}{2}
    X &= V=\Big\{
    v\in H^2(\Omega)~\textrm{such~that}~
    \partial_{\mathbf{n}}v=0~\text{on}~\partial\Omega
    \Big\}, &\qquad
    H &= H^1(\Omega), \\[0.3em]
    G &= L^2(\Omega;\mathbb{R}^{d \times d}_{\mathrm{sym}}), &\qquad
    W &= L^2(\partial\Omega),
  \end{alignat*}
  and the operators $A$ and $B$ defined as follows: $Av = D^2 v$ and $Bv = v|_{\partial\Omega}$.

  \medskip
  \noindent
  \textit{Verification of~\eqref{eq:3.1}.}
  The embedding $H^2(\Omega) \hookrightarrow H^1(\Omega)$ is compact
  by the Rellich--Kondrachov theorem.  Since $V$ is a closed subspace
  of~$H^2(\Omega)$, the inclusion $V \subset H^1(\Omega)$ is also
  compact.

  \medskip
  \noindent
  \textit{Verification of~\eqref{eq:3.2}.}
  We must show that $\|v\|_{H^1(\Omega)} + \|D^2 v\|_{L^2(\Omega)}$
  defines a norm on~$V$ equivalent to~$\|v\|_{H^2(\Omega)}$.  This is
  immediate from the definition of the $H^2$~norm:
  \[
  \|v\|_{H^2(\Omega)}^2
  = \|v\|_{L^2(\Omega)}^2 + \|\nabla v\|_{L^2(\Omega)}^2
  + \|D^2 v\|_{L^2(\Omega)}^2
  = \|v\|_{H^1(\Omega)}^2 + \|D^2 v\|_{L^2(\Omega)}^2.
  \]
  Hence
  \[
  \frac{1}{\sqrt{2}} \|v\|_{H^2(\Omega)}
  \leq \|v\|_{H^1(\Omega)} + \|D^2 v\|_{L^2(\Omega)}
  \leq \sqrt{2} \, \|v\|_{H^2(\Omega)},
  \]
  which establishes~\eqref{eq:3.2}.

  \medskip
  \noindent
  \textit{Verification of~\eqref{eq:3.3}.}
  If $v \in V$ satisfies
  $D^2 v = 0$ and $v|_{\partial\Omega} = 0$, then $v$ is affine, the
  Neumann condition forces $v$ to be constant, and the Dirichlet
  condition yields $v \equiv 0$.

  \medskip
  Theorem~\ref{thm:Friedrichs-Poincare} now
  yields~\eqref{eq:FP-H1}.  Coercivity follows immediately:
  \[
  \|v\|_{H^2(\Omega)}^2
  = \|v\|_{H^1(\Omega)}^2 + \|D^2 v\|_{L^2(\Omega)}^2
  \leq C \bigl( \|D^2 v\|_{L^2(\Omega)}^2
  + \|v\|_{L^2(\partial\Omega)}^2 \bigr)
  = C \, \widehat{a}(v,v). \qedhere
  \]
\end{proof}





\end{document}